\documentclass[12pt, letterpaper]{report}
\usepackage{amsmath, amsthm, amssymb, amsfonts, verbatim, wrapfig, graphicx}

\usepackage[bookmarks]{hyperref}

\usepackage[top = 4cm, left = 4cm, bottom = 2.5cm, right = 2.5cm]{geometry}
\usepackage{setspace}
\usepackage{enumitem}
\setlist[enumerate]{label=(\arabic*)}

\newtheoremstyle{myremark}
{10pt}
{20pt}
{ }
{ }
{\bfseries}
{ }
{ }
{\thmname{#1}\thmnumber{ #2}\thmnote{ (#3)}}

\newtheorem{thm}{Theorem}[section]
\newtheorem{cor}[thm]{Corollary}
\newtheorem{lem}[thm]{Lemma}
\newtheorem{prop}[thm]{Proposition}
\theoremstyle{definition}
\newtheorem{defn}[thm]{Definition}
\theoremstyle{myremark}
\newtheorem{remark}[thm]{Remark}
\newtheorem{assumption}[thm]{Assumption}

\numberwithin{equation}{section}
\numberwithin{thm}{section}

\newcommand{\norm}[1]{\left\Vert#1\right\Vert}
\newcommand{\abs}[1]{\left\vert#1\right\vert}
\newcommand{\goesto}{\rightarrow}
\newcommand{\R}{{\mathbb{R}}}

\newcommand{\N}{{\mathbb{N}}}

\newcommand{\D}{{\mathcal{D}}}
\newcommand{\F}{{\mathcal{F}}}

\newcommand{\G}{{\mathcal{G}}}
\newcommand{\A}{{\mathcal{A}}}

\newcommand{\Schw}{{\mathcal{S}}}

\newcommand{\lap}{\Delta}
\newcommand{\pd}{\partial}

\newcommand{\dv}{\mbox{div }}
\newcommand{\diam}[1]{\mbox{diam}\left(#1\right)}
\newcommand{\supp}[1]{\mbox{supp}\left(#1\right)}

\DeclareMathOperator*{\essliminf}{ess \ lim inf}

\def\be{\begin{equation}}
\def\ee{\end{equation}}
\def\d{\delta}
\def\a{\alpha}
\def\e{\epsilon}
\def\b{\beta}
\def\l{\lambda}
\def\O{\Omega}
\def\w{\omega}
\def\n{\nu}

\begin{document}

	\title{$p$-HARMONIC FUNCTIONS IN $\R^N_+$ WITH NONLINEAR NEUMANN BOUNDARY CONDITIONS AND MEASURE DATA}
	\author{Natham Matias Aguirre Qui\~{n}onez}
	\date{December 2018}
	
\pagenumbering{gobble}

\begin{wrapfigure}{L}{0.08\textwidth}
		\includegraphics[scale=0.7]{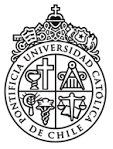} 
\end{wrapfigure}

\hfill

{\small PONTIFICIA UNIVERSIDAD CAT\'{O}LICA DE CHILE

FACULTAD DE MATEM\'{A}TICAS

\hspace{2.5\parindent}DEPARTAMENTO DE MATEM\'{A}TICA}

\vspace{0.125 \textheight}
\renewcommand{\setstretch}{2}
\begin{center}
	{$p$-HARMONIC FUNCTIONS IN $\R^N_+$ WITH NONLINEAR NEUMANN BOUNDARY CONDITIONS AND MEASURE DATA\par}
\end{center}
\renewcommand{\setstretch}{1.5}
\vspace{2\baselineskip}
\begin{center}
	POR
	
	NATHAM MATIAS AGUIRRE QUI\~{N}ONEZ
	
Tesis presentada a la Facultad de Matem\'{a}ticas de la Pontificia Universidad Cat\'{o}lica de Chile para optar al grado acad\'{e}mico de Doctor en Matem\'{a}tica
\end{center}
\vspace{2\baselineskip}
\begin{center}
	\textit{Directores de Tesis:}
	
	Marta Eugenia Garc\'ia-Huidobro Campos, Pontificia Universidad Cat\'{o}lica de Chile. \\	
	Laurent V\'{e}ron, Universit\'{e} Fran\c{c}ois Rabelais, Tours, Francia. 
	
	\textit{Comisi\'{o}n informante:}
	
	Gueorgui Dimitrov Raykov, Pontificia Universidad Cat\'{o}lica de Chile.\\
	Ignacio Guerra Benavente, Universidad de Santiago de Chile.
\end{center}
\vspace{1\baselineskip}
\begin{center}
	Diciembre, 2018\\
	Santiago, Chile.
\end{center}

\newpage

\pagenumbering{gobble}

\noindent\textcopyright 2018, Natham Matias Aguirre Qui\~{n}onez.\par



\newpage

\pagenumbering{roman}
\addcontentsline{toc}{chapter}{Dedicatoria}

\hphantom{1pt}
\vfill
\begin{flushright}
	\textit{A mi amado esposo,} \par \textit{y a nuestro futuro juntos.}
\end{flushright}
\par
\newpage

\addcontentsline{toc}{chapter}{Agradecimientos}

\begin{center}
	AGRADECIMIENTOS
\end{center}
\par

Esta tesis fue parcialmente financiada por el programa de colaboraci\'{o}n ECOS-CONICYT C14E08, por CONICYT-PCHA/Doctorado Nacional/2014-21140322, y por la Vicerrector\'{\i}a de Investigaci\'{o}n de la Pontificia Universidad Cat\'{o}lica de Chile a trav\'{e}s de sus programas `Estad\'{\i}a en el Extranjero para Tesistas de Doctorado' y `Beneficio de Residencia para Alumnos en V\'{\i}as de Graduaci\'{o}n'. 
Agradezco tambi\'{e}n al Institut Henri-Poincar\'{e} por facilitarme un agradable lugar de trabajo en mis visitas a Paris. 
\newpage


\tableofcontents
\newpage


\addcontentsline{toc}{chapter}{Resumen}

\begin{center}
	\textbf{Resumen}
\end{center}

En este trabajo proponemos y estudiamos un concepto de soluci\'{o}n renormalizada al problema
\[
\begin{cases}
-\lap_pu=0 & \mbox{ en } \R^N_+ \\
\abs{\nabla u}^{p-2}u_\n + g(u) = \mu & \mbox{ en } \pd\R^N_+
\end{cases}
\]
donde $1<p\leq N$, $N\geq 2$, $\R^N_+=\left\lbrace(x',x_N):x'\in\R^{N-1}, x_N>0\right\rbrace $, $u_\n$ es la derivada normal de $u$, $\mu$ es una medida de Radon acotada, y $g:\R\goesto\R$ es un t\'ermino no lineal. Obtenemos resultados de estabilidad y, haciendo uso de la simetr\'{\i}a del dominio, estimaciones en hiperplanos, y m\'{e}todos de potenciales, mostramos variados resultados de existencia. En particular, mostramos existencia de soluciones para problemas con t\'{e}rminos no lineales del tipo sumidero tanto en el caso subcr\'{\i}tico como el supercr\'{\i}tico. En el problema con fuente estudiamos el t\'{e}rmino no lineal $g(u)=-u^q$, mostrando existencia en el caso supercr\'{\i}tico, y no existencia en el caso subcr\'{\i}tico. Adem\'{a}s, damos una caracterizaci\'{o}n de conjuntos removibles cuando $\mu\equiv 0$ y $g(u)=-u^q$ en el caso supercr\'{\i}tico. 

Debemos resaltar que este trabajo est\'{a} motivado por resultado obtenidos para la ecuaci\'{o}n $-\lap_p u + g(x,u) = \mu$ en dominios acotados. Notamos que existen algunos resultados de existencia para problemas similares al aqu\'{i} estudiado, pero en dominios acotados, y que estos son bastante restrictivos ya que, por ejemplo, no admiten cualquier medida acotada de Radon $\mu$ como dato. En este sentido, los principales resultados aqu\'{\i} presentados son completamente nuevos.

\newpage


\addcontentsline{toc}{chapter}{Abstract}

\begin{center}
	\textbf{Abstract}
\end{center}

	We propose and study a concept of renormalized solution to the problem
	\[
	\begin{cases}
		-\lap_pu=0 & \mbox{ in } \R^N_+ \\
		\abs{\nabla u}^{p-2}u_\n + g(u) = \mu & \mbox{ on } \pd\R^N_+
	\end{cases}
	\]
	where $1<p\leq N$, $N\geq 2$, $\R^N_+=\left\lbrace(x',x_N):x'\in\R^{N-1}, x_N>0\right\rbrace $, $u_\n$ is the normal derivative of $u$, $\mu$ is a bounded Radon measure, and $g:\R\goesto\R$ is a nonlinear term. We develop stability results and, using the symmetry of the domain, apriori estimates on hyperplanes, and potential methods, we obtain several existence results. In particular, we show existence of solutions for problems with nonlinear terms of the absorption type in both the subcritical and supercritical case. For the problem with source we study the power nonlinearity $g(u)=-u^q$, showing existence in the supercritical case, and nonexistence in the subcritical one. We also give a characterization of removable sets when $\mu\equiv 0$ and $g(u)=-u^q$ in the supercritical case. 

	We remark that this work is motivated by results obtained for the problem $-\lap_pu + g(x,u)=\mu$ in bounded domains. We note that there are some existence results for similar problems to the one that we propose here, although in bounded domains, and that these are fairly restrictive since, for example, not any Radon measure $\mu$ is allowed as datum. In this sense, the main results presented here are completely new.

\newpage
\pagenumbering{arabic}
\chapter{Introduction}

In this work we consider the problem of finding solutions to 
\be\label{equation}
\begin{cases}
	-\lap_pu=0 & \mbox{ in } \R^N_+ \\
	\abs{\nabla u}^{p-2}u_\n + g(u) = \mu & \mbox{ on } \pd\R^N_+
\end{cases}
\ee
where $1<p\leq N$, $N\geq 2$, $\R^N_+=\left\lbrace(x',x_N):x'\in\R^{N-1}, x_N>0\right\rbrace $, and $\mu\in \mathfrak{M}_b\left( \pd\R^N_+\right) $. Here $\mathfrak{M}_b\left( \pd\R^N_+\right) $ is the space of Radon measures in $\R^N $ with bounded total variation which are supported in $\pd\R^N_+=\left\lbrace(x',x_N):x'\in\R^{N-1}, x_N=0\right\rbrace $, $u_\n$ is the normal derivative of $u$, $g:\R\goesto\R$ is a nonlinear term, and
\[
-\lap_pu:=-\dv \left( \abs{\nabla u}^{p-2}\nabla u\right) .
\]

Consider the related problem of finding a solution to
\be\label{equation_in_bounded_domain}
\begin{cases}
	-\lap_pu + g(x,u)=\mu & \mbox{ in } \O \\
	u= 0 & \mbox{ on } \pd\O
\end{cases}
\ee
where $\O$ is a bounded domain in $\R^N$, $\mu\in\mathfrak{M}_b\left( \O\right) $, and $g:\R^N\times \R\goesto \R$. 
If $p>N$ a unique solution can be obtained by the theory of monotone operators from $W^{1,p}_0(\O)$ into its dual $W^{-1,p'}(\O)$, since in this case any bounded measure in $\O$ belongs to this dual. 

When $1<p\leq N$ the study of problem \eqref{equation_in_bounded_domain} is based upon the theory of \textit{renormalized solutions}. In the case $g\equiv 0$ the concept of renormalized solution was first introduced in \cite{DMOP99}, wherein the authors showed existence and partial uniqueness results. The proof of existence relies in a delicate and very technical stability result. The concept of renormalized solution has been since then the main tool to study degenerate elliptic problems with measure data. We refer the reader to \cite{Veron16} for an overview of the concept and further references. Let us note that in the special case $\mu\in L^1\left( \O\right) + W^{-1,p'}\left( \O\right) $ the concept of renormalized solution coincides with the concept of \textit{entropy solutions} to \eqref{equation_in_bounded_domain} introduced in \cite{BBGGPV95} (see also \cite{BGO96}), in which it is shown both existence and uniqueness for the case $g\equiv 0$.

When $g$ is nontrivial, the nature of equation \eqref{equation_in_bounded_domain} depends on the sign of $g(x,u)u$. To emphasize this fact, and in accordance with the literature, we call it a problem with \textit{absorption} when $g(x,u)u\geq0$, and a problem with \textit{source} when $g(x,u)u\leq0$. Further, we say that the problem is subcritical whenever there are conditions imposed on the growth of $\abs{g(x,s)}$; otherwise we say that the problem is supercritical. 

For the problem with absorption, existence of renormalized solutions to \eqref{equation_in_bounded_domain} in subcritical cases has been shown in \cite{B-V03} and \cite{Veron16}. In \cite{B-V03} the author considers $g(u)=\abs{u}^{q-1}u$, while more general nonlinearities are considered in \cite{Veron16}. Let us mention that in the power case one obtains the sufficient condition $q\in (0,\frac{N(p-1)}{N-p})$ if $1<p<N$, and $q\geq 0$ if $p=N$. In the case $p=N$ exponential-type nonlinearities are considered, but under a restriction in the size of the measures.

The problem with absorption in supercritical cases has been studied in \cite{B-VHV14}. There the authors show that given a fixed nonlinear term $g$ existence of renormalized solutions to \eqref{equation_in_bounded_domain} holds for a certain class of measures. Their results, which are quite general, are based mainly on a delicate study of the Wolff potential and can be applied to establish sufficient conditions for existence of renormalized solutions when the form of the nonlinearity is more explicit. For example, when $g= \abs{u}^{q-1}u$, $q>p-1$, a sufficient conditions on the measure is that it be absolutely continuous with respect to the $L^{p,\frac{q}{q-p+1}}\left( \R^N\right) $ Bessel capacity. A boundedness condition on the measure is obtained when the nonlinearity is of exponential-type. 
 
Problem \eqref{equation_in_bounded_domain} is rather more difficult when it is of the source type. In fact, in this case only nonnegative solutions are considered. In \cite{PV08} the authors show equivalent sufficient conditions to obtain nonnegative renormalized solutions to \eqref{equation_in_bounded_domain} when $g(u)=-u^q$ and $\mu$ is a nonnegative measure. It is also shown that, if the measure is compactly supported in $\O$, these conditions are necessary. Their results, which are obtained through a careful study of the Wolff potential and its relationship to the Bessel and Riesz potential, show in particular that if $q\in (p-1, \frac{N(p-1)}{N-p})$ and $1<p<N$, or $q\geq p-1$ and $p=N$ (i.e., the subcritical case), then any nonnegative measure with small enough norm admits a solution. In the supercritical case, a sufficient condition is that the measure must be `Lipschitz' continuous with respect to the $L^{p,\frac{q}{q-p+1}}\left( \R^N\right) $ Bessel capacity. We note also that their main results allows them to present a complete characterization of removable sets for \eqref{equation_in_bounded_domain} in terms of some fractional Bessel capacities, as well as to prove Liouville-type results for problems in the whole $\R^N$.

Going back to problem \eqref{equation}, in the case $p>N$ we can similarly obtain a unique solution in the space $W^{1,p}\left( \R^N_+\right) $ by using the theory of monotone operators. This follows from the fact that if $p>N$ then functions in $W^{1,p}\left( \R^N_+\right) $ have well defined continuous and bounded traces in $\pd\R^N_+$, and so any element in $\mathfrak{M}_b\left( \pd\R^N_+\right)$ can be seen as an element in the dual of $W^{1,p}\left( \R^N_+\right) $. Of course, this works whenever $\mu$ is in the dual of $W^{1,p}\left( \R^N_+\right) $ (even if $1<p\leq N$). 
In fact, this is the approach used in \cite{WZA11}, where it is proven the existence of weak solutions to the subcritical problem with absorption
\[
\begin{cases}
-\lap_p u + \e\abs{u}^{q-1}u = f(x) & \mbox{ in } \O \\
\abs{\nabla u}^{p-2}u_\n = g(x) & \mbox{ on } \pd\O ,
\end{cases} 
\]
where $f\in L^{p'}\left( \O\right)$, $g\in W^{\frac{1}{p}-1, p'}\left( \pd\O\right) $, $\e$ is a nonnegative constant, $\O$ is a bounded domain, $\frac{2N}{N+1}<p$, and $q\geq 0$ satisfies $q\leq \frac{Np}{N-p}-1$ if $p<N$. 

In the case $1<p\leq N$ we turn to the idea of renormalized solutions. In \cite{ABBO13} a concept of renormalized solution was proposed for a Neumann problem in bounded domains and with nonnegative measures in $L^1$ (see also \cite{IO16}). However, to our knowledge, there is no proposed definition of renormalized solutions to Neumann problems such as \eqref{equation} for general bounded Radon measures. In this work we propose such a definition and then prove existence of renormalized solutions for various types of nonlinearities. Indeed, we have Theorem \ref{existence_half_space} for the case $g\equiv0$, Theorems \ref{existence_subcritical_p<N} and \ref{existence_subcritical_p=N} for subcritical problems with absorption, Theorem \ref{B-VHV14_trace} for supercritical problems with absorption, and Theorem \ref{Sufficiency_supercritical_source} for a supercritical problem with source. On the other hand, in Theorem \ref{Liouville_Theorem} we show nonexistence of nontrivial nonnegative solutions for the same problem with source but in the subcritical case.

Our approach to solving problem \eqref{equation} is to turn it into an associated problem in the whole $\R^N$. Indeed, formally, if $u$ is a solution to \eqref{equation} then we expect that $\bar{u}$, its even reflection across $\pd\R^N_+$, should be a solution to 
\be\label{equation_in_R_N}
-\lap_p\bar{u} + 2g(\bar{u})\mathcal{H}=2\mu \mbox{ in } \R^N
\ee
where $\mathcal{H}$ is a normalized $\left( N-1\right) $-dimensional Hausdorff measure concentrated in $\pd\R^N_+$. Note however that not every solution of the above problem would yield a solution to \eqref{equation}, unless it is a symmetric solution, and so the problems are not equivalent. 

The advantage of looking at this extended problem is that we can obtain a solution to \eqref{equation_in_R_N} by applying the theory developed in \cite{DMOP99}, \cite{Veron16}, \cite{B-VHV14}, and \cite{PV08}, to an increasing sequence of bounded domains. In order for this approach to work we need to establish some stability results. Then, to recover a solution to \eqref{equation}, we show that the solutions obtained through this process might in fact be taken to be symmetric with respect to $\pd\R^N_+$. It is worth mentioning that with our definition of renormalized solution to problem \eqref{equation}, $\bar{u}$ becomes in fact a \textit{local} renormalized solution of the associated problem in $\R^N$, as defined for example in \cite{B-V03} and \cite{Veron16}. 

This thesis is organized as follows. In Chapter \ref{Capacities} we collect all the relevant preliminary definitions and results we shall need both to define renormalized solutions and to obtain the existence results. Since we consider measures and functions in both $\R^N$ and $\pd\R^N_+$, we will need to consider the problem of obtaining well defined traces, as well as the interplay of the Bessel capacities defined in $\R^N$ and $\pd\R^N_+$. In particular, we will make use of trace and extension operators in the Lizorkin-Triebel spaces. 

In Chapter \ref{Renormalized solution} we give the definitions of renormalized solutions in bounded domains and of local renormalized solutions in general domains, define renormalized solutions to problem \eqref{equation}, and state estimates and other results from \cite{DMOP99} which will be used in the sequel. We prove some properties of renormalized solutions to problem \eqref{equation}, and also show that renormalized solutions to \eqref{equation} do in fact exists by proving that the fundamental solution to problem \eqref{equation}, with $g\equiv 0$, is a renormalized solution in the sense of our definition with $\mu$ the Dirac mass.

In Chapter \ref{Extension} we study the problem of obtaining local renormalized solutions to $-\lap_p u = \mu $ in $\R^N$. The existence of such solutions is given in Theorem \ref{existenceExtended}. The proof is based on two lemmas, both of which are of use later when dealing with nonlinear terms. The first states that given a sequence $\left\lbrace u_m\right\rbrace _m$ of renormalized solutions to \eqref{equation_in_bounded_domain} with $g\equiv 0$, $\O=B_m(0)=\left\lbrace\abs{x}<m\right\rbrace $, and data $\mu_m$, we can find a convergent subsequence such that the limit function $u$ has the necessary properties to be a local renormalized solution in $\R^N$. This is proven by a slight modification of the argument found in \cite{DMOP99}. The second lemma is a stability result that will be used to show that the function $u$ is indeed a local renormalized solution. This is proven for more general equations than the ones considered in the first lemma, and its proof is based on the argument given in \cite{M05} to bypass the rather involved arguments developed in \cite{DMOP99}. 

In Chapter \ref{Symmetry} we show that the solution obtained by the above method is symmetric with respect to $\pd\R^N_+$. We do this by showing that, in a bounded domain, if a measure is concentrated in $\pd\R^N_+$ and the domain is symmetric with respect to $\pd\R^N_+$ then any renormalized solution has the same symmetry. For this we use the partial uniqueness result obtained in \cite{DMOP99}. Then, in Theorem \ref{existence_half_space}, we use this symmetry to recover a solution to the original problem \eqref{equation} when $g\equiv 0$.

In Chapter \ref{absorption}, Section \ref{subcritical}, we consider the problem of obtaining renormalized solutions to \eqref{equation} in subcritical cases with absorption. Here we use the theory developed in \cite{Veron16} in the same spirit as the previous chapters, obtaining local renormalized solutions in $\R^N$ as the limit of solutions in bounded domains, and then using symmetry to obtain solutions to \eqref{equation}. Our approach is to use the existence results developed for problem \eqref{equation_in_bounded_domain} to obtain solutions to 
\be\label{equation_bounded_hausdorff}
\begin{cases}
	-\lap_pu + g(u)\mathcal{H}=\mu & \mbox{ in } \O \\
	u= 0 & \mbox{ on } \pd\O ,
\end{cases} 
\ee
as an intermediate step towards solving \eqref{equation}. We obtain solutions to the above equation by solving \eqref{equation_in_bounded_domain} when $g$ is multiplied by a sequence $\zeta_n(x_N)$ that is concentrating at the origin and then letting $n\goesto\infty$. The main result in this regard is Lemma \ref{general_passage_to_trace_nonlinearity} where we show, under very general assumptions, that if $u_n$ are the solutions with nonlinear term $\zeta_ng(u_n)$ then $\zeta_ng(u_n)$ converges, in a suitable sense, to $g(u)\mathcal{H}$. The result is proven by making a decomposition of the domain in order to use the assumptions on $g$ as well as the continuity properties of $W^{1,p}$ functions and their traces. 

In the case $p<N$ the existence result is given in Theorem \ref{existence_subcritical_p<N}, while for the case $p=N$ is given in Theorem \ref{existence_subcritical_p=N}. Let us mention that when $g(s)=\abs{s}^q$, $q\geq 0$, $1<p<N$, Theorem \ref{existence_subcritical_p<N} guarantees existence of renormalized solutions to \eqref{equation} provided 
\[
q<\frac{(N-1)(p-1)}{N-p} .
\]
If $p=N$ then Theorem \ref{existence_subcritical_p=N} only requires $q\geq 0$, and in fact exponential-type nonlinearities are allowed, but this imposes conditions on the size of the measure.

In Section \ref{supercritical} we consider supercritical problems with absorption under the condition $1<p<N$. Here we use mainly the work in \cite{B-VHV14}. We have left the definition of the Wolff potential of a measure, and other related quantities, to this chapter since they are only used from this point forward. As in the previous section, we obtain solutions to \eqref{equation_bounded_hausdorff} as an intermediate step towards solving \eqref{equation}. The main tool for this is the improvement of an estimate of renormalized solutions, in terms of the Wolff potential of their respective measures, from $a.e.$ in $\O$ to $a.e.$ in any hyperplane. We also show that, given a fixed measure and a nondecreasing sequence of domains, it is possible to obtain a nondecreasing sequence of renormalized solutions in said domains. Both results are needed to show the main existence result, Theorem \ref{B-VHV14_trace}. This theorem is then used to obtain explicit conditions on the measure when more is known about the rate of growth of the nonlinearity. For example, when $g(s)=\abs{s}^{q-1}s$, $q>p-1$, we obtain as sufficient condition that the measure must be absolutely continuous with respect to the $L^{p-1,\frac{q}{q-p+1}}\left( \R^{N-1}\right) $ Bessel capacity. Exponential-type nonlinearities are also considered.

Finally, in Chapter \ref{source} we consider nonnegative solutions to the problem with source $g(u)=-u^q$ when $q>p-1$ and $1<p<N$. Our work here follows closely the ideas in \cite{PV08}, particularly those used to treat problem \eqref{equation_in_bounded_domain} when $\O=\R^N$. We begin by establishing necessary and sufficient conditions for existence of nonnegative renormalized solutions to \eqref{equation}. In particular, in the supercritical case, we obtain existence of renormalized solutions to \eqref{equation} when the nonnegative measure $\mu$ is `Lipschitz' continuous with respect to the $\dot{L}^{p-1,\frac{q}{q-p+1}}\left( \R^{N-1}\right) $ Riesz capacity. 

In Corollary \ref{improved_testing_inequality} we note that if $u$ is a solution with datum $\mu$, then $u^q\mathcal{H}+\mu$ must satisfy the above condition which, together with the properties of the Riesz capacity, implies the nonexistence of nontrivial nonnegative solutions to \eqref{equation} in the subcritical case, that is, when  
\[
q\in \begin{cases}
\left( p-1,\frac{(N-1)(p-1)}{N-p}\right] & \mbox{ if } p<N \\
\left( p-1,\infty\right)  & \mbox{ if } p=N .
\end{cases} 
\]

This result is not surprising as it is a natural counterpart to the nonexistence result in \cite{PV08}. It is also in agreement with the nonexistence result in \cite{HU94} for the linear case (i.e., $p=2$) where it is shown that any classical, but possibly singular at the origin, nonnegative solution to \eqref{equation}, with $\mu\equiv0$, $g(u)=-u^q$, and $q$ in the range $[1, \frac{N-1}{N-2}]$, must be trivial.

We finish with the problem of characterizing when a compact set $K\subset\pd\R^N_+$ is removable for \eqref{equation} in the case $g(u)=-u^q$ and $\mu\equiv 0$. We say that such a set $K$ is removable if every nonnegative $p$-harmonic function $u$ satisfying the Neumann boundary condition $\abs{\nabla u}^{p-2}u_\n=u^q$ in $\pd\R^N_+\setminus K$ can be extended to a solution of \eqref{equation}. The necessary and sufficient conditions for existence given in Corollary \ref{improved_testing_inequality} and Corollary \ref{equivalence_supercritical_source} allows us to show that a set is removable in the supercritical case if and only if its $\dot{L}^{p-1,\frac{q}{q-p+1}}\left( \R^{N-1}\right) $ Riesz capacity is zero. 

\chapter{Preliminary definitions and results}\label{Capacities}

Here we collect basic definitions and results needed in the sequel. 
We remark that, as in \cite{DMOP99}, we shall make use of Bessel capacities to decompose measures in $\mathfrak{M}_b\left( \pd\R^N_+\right) $. Since we will frequently consider the behavior of functions and measures when restricted to hyperplanes, we will also consider capacities in $\R^{N-1}$. To this end, we will introduce the more general capacities associated with the Lizorkin-Triebel spaces $F_\a^{p,q}$.

Let us first introduce some notation. For any measurable set $E\subset \R^N$ we denote by $\abs{E}$ its Lebesgue measure. When $E\subset \pd\R^N_+\simeq \R^{N-1}$ we take this measure to be the $(N-1)$-dimensional Lebesgue measure.
We let $B_M(x)$ be the open ball of radius $M>0$ centered at $x$ (simply $B_M$ when $x=0$). Depending on the context, when $x\in \pd\R^N_+$ this could be either a $N$-dimensional ball in $\R^N$ or a $(N-1)$-dimensional ball in $\pd\R^N_+$. 
For any set $E$, we let $\chi_E$ be the characteristic function of $E$. 
The truncation of functions will be very important in the sequel. For any $k>0$, we let $T_k(s)=\min(k,\max(-k,s))$. 

By an abuse of notation, we define $W^{1,p}_{loc}\left( \R^N_+\right) :=\bigcap_{M\in \N} W^{1,p}\left( B_M\cap \R^N_+\right) $. Similarly, we define $L^{s}_{loc}\left( \R^N_+\right) :=\bigcap_{M\in \N} L^{s}\left( B_M\cap \R^N_+\right) $. We remark that given a domain $\O$, $L^s\left( \O\right) $ are the usual Lebesgue spaces, while $W^{1,p}\left( \O\right) $ are the usual Sobolev spaces. The norm in the $L^s\left( \O\right) $ spaces will be written indistinctly as $\norm{\cdot}_{L^s\left( \O\right) }$, $\norm{\cdot}_{L^s}$, or simply $\norm{\cdot}_s$. 
The $C^k\left( \O\right) $ space, $k\in \N\cup \{\infty\}$, is the usual space of $k-$ times continuously differentiable functions, and $C_0^k\left( \O\right) $ is the subspace of elements with compact support in $\O$. 

\section{Bessel capacities} We start with the standard definition of Bessel capacities in $\R^N$ (see \cite{AdamsHedberg} for details). For any compact set $K\subset\R^N$ we let 
\[
\w_K = \left\lbrace \phi \in \Schw\left( \R^N\right) \ : \ \phi\geq \chi_K\right\rbrace 
\]
where $\Schw\left( \R^N\right) $ is the Schwartz class, and define for any $\a>0$ and $1<p<\infty$
\[
cap_{\a,p,N}\left( K\right) = \inf\left\lbrace \norm{\phi}_{\a,p,\R^N}^p \ : \ \phi \in \w_K\right\rbrace 
\]
with the convention that $\inf\emptyset = +\infty$. Here $\norm{\cdot}_{\a,p,\R^N}$ is the norm in the Bessel potential spaces $L^{\a,p}\left( \R^N\right) $ of functions $f=\G_\a \ast g$ with $g\in L^p$, where $\G_\a(x)=\F^{-1}\left[ \left( 1+\abs{\cdot}^2\right)^{-\a/2}\right](x) $ is the Bessel kernel of order $\a\in \R$, defined as $\norm{f}_{\a,p,\R^N}=\norm{g}_{L^p\left( \R^N\right) }$ (note that $\Schw\left( \R^N\right) $ is a dense subset of this space). Then we extend the definition to open sets $G\subset\R^N$ by 
\[
cap_{\a,p,N}\left( G\right) = \sup\left\lbrace cap_{\a,p,N}\left( K\right)  \ : \ K\subset G \ , \ K \mbox{ compact}\right\rbrace 
\]
and finally to arbitrary sets $E\subset R^N$ by 
\[
cap_{\a,p,N}\left( E\right) = \inf\left\lbrace cap_{\a,p,N}\left( G\right)  \ : \ E\subset G \ , \ G \mbox{ open}\right\rbrace \ .
\]

Note that when $\a\in \N$ we have $L^{\a,p}\left( \R^N\right) = W^{\a,p}\left( \R^N\right) $, and so in this case the Bessel capacities can be defined using Sobolev spaces. We do not follow this approach since we will need to consider the case when $\a\in (0,1)$. On the other hand, we remark that we have the following equivalent definition of capacity:
\[
cap_{\a,p,N}\left( E\right) = \inf\left\lbrace \norm{f}^p_{L^p\left( \R^N\right) } \ : \ f \in \O_E \right\rbrace 
\]
where
\[
\O_E:=\left\lbrace f \in L^p\left( \R^N\right) \ : \ f\geq 0 \ \forall x\in \R^N , \ \G_\a \ast f\geq1 \ \forall x\in E\right\rbrace 
\]
(see Proposition 2.3.13 of \cite{AdamsHedberg}).

We will also use the Riesz capacities. They can be defined as the capacities associated to the Riesz potential spaces $\dot{L}^{\a,p}\left( \R^N\right) $, i.e., the space of functions $f=\mathcal{I}_\a \ast g$ with $g\in L^p$ , where $\mathcal{I}_\a(x)=C(N,\a)\abs{x}^{-(N-\a)}$ is the Riesz kernel of order $\a\in (0,N)$. We will denote them by $cap_{I_\a,p,N}\left( \cdot\right)$. Let us explicitly state however that our main interest are the $cap_{1,p,N}$ capacity in $\R^N$ and the $cap_{1-\frac{1}{p},p,N-1}$ capacity in $\pd\R^N_+$ (which we identify as $\R^{N-1}$). 

We say that a property holds $cap_{\a,p,N}-$quasi-everywhere in $\O$ (abbreviated as $cap_{\a,p,N}-q.e.$) if there exists a set $E$ such that the property holds in $\O\setminus E$ and $cap_{\a,p,N}\left( E\right) =0$. 

We say that a function $\w$ is $cap_{\a,p,N}-$ quasi-continuous in $\O$ if for every $\e>0$ there is an open set $E$ such that $cap_{\a,p,N}(E)<\e$ and $\w\in C\left( \O\setminus E\right) $. Unless otherwise stated, we assume that $cap_{\a,p,N}-$ quasi-continuous functions are $cap_{\a,p,N}-q.e.$ finite. Whenever we cannot assert that a $cap_{\a,p,N}-$ quasi-continuous function $w$ is $cap_{\a,p,N}-q.e.$ finite, the statement $\w\in C\left( \O\setminus E\right) $ means that $w:\O\setminus E \goesto [-\infty,\infty]$ is continuous with respect to the topology of the extended real line.

We say that a set $E \subset \R^N$ is quasi-open if for every $\e>0$ there exists an open set $\O$ such that $E\subset \O$ and $cap_{1,p,N}(\O\setminus E)<\e$. Clearly, countable unions of quasi-open sets are quasi-open. It is also immediate that if $w$ is $cap_{1,p,N}-$ quasi-continuous then the sets $\left\lbrace w>k\right\rbrace $ and $\left\lbrace w<k\right\rbrace $ are quasi-open. By a result of \cite{D83}, for every bounded quasi-open set $E$ there exists a nonnegative sequence $w_n \in W^{1,p}\left( \R^N\right)$ such that $w_n\leq \chi_E$ and $w_n\uparrow \chi_E$ $cap_{1,p,N}-q.e.$ in $\R^N$ (see Lemma 2.2 of \cite{M05}).

\begin{remark}
	When dealing with a bounded domain $\O$, it is more natural to define and use the so called \textit{condenser} capacity associated with $\O$ (see for example Section 7.6 of \cite{AdamsHedberg}). Indeed, this condenser capacity is the capacity used in works such as \cite{DMOP99} and \cite{M05}. However, Theorem 2.38 of \cite{HKM} shows that the condenser capacity is equivalent to our definition of capacity whenever $\O=B_M$ for any fixed $M$ (see also section 2.7 of \cite{AdamsHedberg}). Since in our applications we always ultimately have $\O=B_M$ for some $M \in \N$, we see that we can always assume the two definitions are equivalent. 
\end{remark}

\section{Decomposition of measures} Let $\mathfrak{M}_b\left( \R^N\right) $ be the set of Radon measures of bounded total variation. For any Borel set $\O\subset \R^N$ we let $\mathfrak{M}_b\left( \O\right) $ be the set of measures in $\mathfrak{M}_b\left( \R^N\right) $ supported in $\O$. For measures $\mu \in \mathfrak{M}_b\left( \O\right)$ we let
\[\norm{\mu}_{\mathfrak{M}_b}=\abs{\mu}\left( \O\right) \]
be its total variation in $\O$. 

We will work mainly with measures supported in $\pd\R^N_+$. Such measures can be naturally identified with measures in $\R^{N-1}$. Indeed, if $\mu \in \mathfrak{M}_b\left( \R^N\right) $ is supported in $\pd\R^N_+$ then $\tilde{\mu}(E):=\mu\left( E\times \left\lbrace x_N=0\right\rbrace \right) $ is the natural representative of $\mu$ in $\mathfrak{M}_b\left( \R^{N-1}\right) $. Similarly, if $\mu \in \mathfrak{M}_b\left( \R^{N-1}\right) $ then $\hat{\mu}\left( E\right) := \mu\left( \left\lbrace x'\in \R^{N-1} \ : \ (x',0)\in E\right\rbrace \right)$ belongs to $\mathfrak{M}_b\left( \R^N\right)$ and is supported in $\pd\R^N_+$. This gives a bijection between $\mathfrak{M}_b\left( \pd\R^N_+\right)$ and $\mathfrak{M}_b\left( \R^{N-1}\right) $. Hence, whenever convenient, we will identify the two spaces under the above construction. 

We let $\mathcal{H}$ be the $(N-1)$-dimensional Hausdorff measure concentrated in $\pd\R_+^N$, normalized so that $\mathcal{H}\left( E\right) = \abs{ E\cap\pd\R^N_+}$ for any measurable set $E\subset \R^N$. Then, we define $L^s\left( \O\cap \pd\R^N_+\right) := L^s\left( \O ; d\mathcal{H}\right)$ for any domain $\O$ and any $1\leq s\leq \infty$ (when $\O=\R^N$ we omit it from the notation). If a function $g$ belongs to $L^1_{loc}\left( \O\cap \pd \R^N_+\right) $ we write $g\mathcal{H}$ as shorthand for the measure $gd\mathcal{H}$.

We will say that a sequence of measures $\mu_n \in \mathfrak{M}_b\left( \R^N\right) $ converges to a measure $\mu\in\mathfrak{M}_b\left( \R^N\right) $ in the narrow topology of measures in a domain $\O$ if and only if
\[
\lim_{n\goesto\infty}\int_\O \phi d\mu_n = \int_\O \phi d\mu
\]
for all functions $\phi$ continuous and bounded in $\O$. We recall that the convergence is in the weak-$\ast$ topology of $\mathfrak{M}_b\left( \O\right) $ if the above holds for all functions $\phi \in C_0\left( \O\right)$. Here $C_0\left( \O\right)$ is the space of continuous functions with compact support in $\O$. 

It is standard that $cap_{\a,p,N}$ is a countably subadditive nonnegative set function (see for example \cite{AdamsHedberg}). This implies that any measure $\mu\in \mathfrak{M}_b\left( \R^N\right) $ can be uniquely decomposed as
\[ 
\mu=\mu_0+\mu_s
\]
where $\mu_0$ is absolutely continuous with respect to $cap_{\a,p,N}$, and $\mu_s$ is singular with respect to $cap_{\a,p,N}$ (see \cite{FST91}, Lemma 2.1). That is, $\mu_0(E)=0$ for every Borel set $E$ such that $cap_{\a,p,N}(E)=0$, while $\mu_s$ is supported in a Borel set $E$ such that $cap_{\a,p,N}(E)=0$. Moreover, by the Jordan decomposition theorem, one can write uniquely
\[
\mu_s=\mu_s^+-\mu_s^-
\]
where $\mu_s^+$ and $\mu_s^-$ are the positive and negative part of $\mu_s$.

In what follows we shall denote by $\mathfrak{M}_0\left( \R^N\right) $ the set of measures in $\mathfrak{M}_b\left( \R^N\right) $ that are absolutely continuous with respect to $cap_{1,p,N}$. Similarly, $\mathfrak{M}_0\left( \O\right) $ is the set of measures in $\mathfrak{M}_0\left( \R^N\right) $ which are supported in $\O$.

We remark that, whenever $\a>0$, the $N$-dimensional Lebesgue measure is absolutely continuous with respect to $cap_{\a,p,N}$ (see \cite{AdamsHedberg}).

The following result is proved in \cite{BGO96}.
\begin{thm}\label{BGO96_thm}
	Let $\O$ be a bounded domain and $\mu\in \mathfrak{M}_b\left( \O\right) $. Then $\mu\in \mathfrak{M}_0\left( \O\right) $ if and only if $\mu \in L^1\left( \O\right) + W^{-1,p'}\left( \O\right)$. Thus, if  $\mu\in \mathfrak{M}_0\left( \O\right) $ then $\mu=f -\dv g$ in the sense of distributions for some functions $f\in L^1\left( \O\right)$ and $g \in \left( L^{p'}\left( \O\right) \right) ^N$. Moreover, $\mu=f-\dv g$ also holds when acting on functions in $W_0^{1,p}\left( \O\right) \cap L^\infty\left( \O\right)$.   
\end{thm}

We note that in the above result one can further assume $\norm{f}_{L^1\left( \O\right) } + \norm{g}_{W^{-1,p'}\left( \O\right)}\leq 3\norm{\mu}_{\mathfrak{M}_b}$ (see Lemma 3.6 of \cite{B-VHV14}).

\section{Lizorkin-Triebel capacities} Now we consider the spaces $F_\a^{p,q}\left( \R^N\right) $ mentioned earlier. The literature concerning these spaces is very extensive. Here we only record a few facts about them and refer the reader to \cite{AdamsHedberg} and \cite{Triebel92} for details. Let us begin with their definition.

Let $\phi$ be any function in $C_0^\infty\left( \R^N\right) $ such that $\supp{\phi}\subset \left\lbrace\zeta \in \R^N \ : \ \abs{\zeta}\leq 2\right\rbrace $ and $\phi\equiv 1$ in $\left\lbrace\zeta\in \R^N\ : \ \abs{\zeta}\leq 1\right\rbrace $. For $j\in \N$ let 
\[
\phi_j(\zeta)=\phi(2^{-j}\zeta)-\phi(2^{-j+1}\zeta)
\]
so that $\supp{\phi_j}\subset \left\lbrace\zeta \in \R^N \ : \ 2^{j-1}\leq \zeta \leq 2^{j+1}\right\rbrace $ and, setting $\phi_0=\phi$,
\[
\sum_{k=0}^\infty \phi_k(\zeta)=1
\]
 in $\R^N$. Let $\Schw'\left( \R^N\right) $ be the set of tempered distributions, and for any $f\in \Schw'\left( \R^N\right)$ let 
 \[
 f_k=\F^{-1}\left[ \phi_k \F\left[ f\right] \right] 
 \]
 where $\F$ is the Fourier transform. Then $f_k$ is an entire analytic function and it can be shown that
 \[
 f=\sum_{k=0}^\infty f_k 
 \]
 in the topology of $S'\left( \R^N\right) $. For $1<p,q<\infty$ and $s\in \R$ we define
 \[
 \norm{f}_{F_s^{p,q}\left( \R^N\right) }:=\norm{\left( \sum_{k=0}^\infty 2^{ksq}\abs{f_k(x)}^q\right) ^{\frac{1}{q}}}_{L^p\left( \R^N\right) } 
 \]
 and
 \[
 F_s^{p,q}\left( \R^N\right) :=\left\lbrace f\in \Schw'\left( \R^N\right) \ :  \norm{f}_{F_s^{p,q}\left( \R^N\right) }<\infty\right\rbrace .
 \]
 It is proven in \cite{Triebel92} that this definition does not depend on the choice of $\phi$, and that $F_s^{p,q}\left( \R^N\right)$ is a Banach space.

It can be shown that the spaces $F_\a^{p,q}\left( \R^N\right)$ can be realized as potential spaces, and thus they can be used to define corresponding $F_\a^{p,q}\left( \R^N\right)$ capacities, which we denote by $cap\left( \cdot, F_\a^{p,q}\left( \R^N\right) \right) $ (see \cite{AdamsHedberg} for the details).

The connection of these spaces with the Bessel potential spaces is given by the fact that for any $\a>0$, and $1<p<\infty$, there holds $F_\a^{p,2}\left( \R^N\right) = L^{\a,p}\left( \R^N\right) $ in the sense of normed spaces. Given the above observation it is to be expected that the $F_\a^{p,2}\left( \R^N\right)$ capacities are equivalent to the corresponding $L^{\a,p}\left( \R^N\right)$ Bessel capacities. A surprising result (see Proposition 4.4.4 of \cite{AdamsHedberg}) is that in fact for all $\a>0$, $1<q<\infty$, and $1<p\leq \frac{N}{\a}$, the $F_\a^{p,q}\left( \R^N\right) $ capacity is equivalent to the corresponding Bessel $L^{\a,p}\left( \R^N\right) $ capacity; we will point this out by writing 
\[
cap\left( \cdot ; F_\a^{p,q}\left( \R^N\right) \right) \sim cap_{\a,p,N}\left( \cdot\right)  \ .
\]
An advantage of considering the more general $F_\a^{p,q}\left( \R^N\right) $ spaces is the following theorem, which can be found in Chapter 4.4 of \cite{Triebel92}.
\begin{thm}\label{Trace&Extension}
	Let $1<p,q<\infty$, and $\a p>1$. Then the map 
	\[
	Tr  :  f(x',x_N)\mapsto f(x',0)
	\]
	is a bounded linear operator from $F_\a^{p,q}\left( \R^N\right) $ onto $F_{\a-\frac{1}{p}}^{p,p}\left( \R^{N-1}\right) $. Moreover, there exists a linear bounded extension operator $Ex$ from $F_{\a-\frac{1}{p}}^{p,p}\left( \R^{N-1}\right) $ into $F_\a^{p,q}\left( \R^N\right) $ such that
	\[
	Tr\circ Ex = Id \mbox{ in } F_{\a-\frac{1}{p}}^{p,p}\left( \R^{N-1}\right) \ .
	\]
\end{thm}

Thanks to the above theorem, we can define the trace $Tr(w)=w(x',0) \in F_{\a-\frac{1}{p}}^{p,p}\left( \R^{N-1}\right)$ of any function $w(x',x_N)$ in $F_\a^{p,q}\left( \R^N\right)$.

Let us mention that we will also use Sobolev's embedding-type results for these spaces in the sequel. We will point this out later. 

The following proposition, which follows from Theorem \ref{Trace&Extension}, shows that the `trace' of the $cap_{1,p,N}$ capacity in $\pd\R^N_+$ is the $cap_{1-\frac{1}{p},p,N-1}$ capacity. 

\begin{prop}\label{trace&extensionE}
	There exists a constant $C(N,p)$ such that for all Borel sets $E\subset \R^N$ and $E'\subset \pd\R^N_+$
	\begin{enumerate}
		\item $C(N,p)cap_{1,p,N}(E) \geq  cap_{1-\frac{1}{p},p,N-1}(E\cap\pd\R^N_+)$, and
		\item $C(N,p)cap_{1-\frac{1}{p},p,N-1}(E') \geq cap_{1,p,N}(E')$.
	\end{enumerate}
\end{prop}
\begin{proof}
	By the definition of capacity and the capacitability of Borel sets it is enough to consider $E$ and $E'$ compact. Let $g\in \Schw\left( \R^N\right)$ be such that $g\geq \chi_E$. By Theorem \ref{Trace&Extension} $g$ has a trace $\bar{g}=Tr(g)$ such that 
	\[
	\norm{\bar{g}}_{F_{1-\frac{1}{p}}^{p,p}\left( \pd\R^N_+\right)}\leq C(N,p)\norm{g}_{F_{1}^{p,2}\left( \R^N\right)}=C(N,p) \norm{g}_{1,p,\R^N}.
	\]
	Since $\bar{g}\geq \chi_{E\cap \pd\R^N_+}$  we have $C(N,p)cap_{1,p,N}\left( E\right)\geq cap\left( E\cap\pd\R^N_+;F_{1-\frac{1}{p}}^{p,p}\left( \pd\R^N_+\right) \right)$, and since $ cap\left( E\cap\pd\R^N_+;F_{1-\frac{1}{p}}^{p,p}\left( \pd\R^N_+\right) \right) \sim cap_{1-\frac{1}{p},p,N-1}\left( E\cap\pd\R^N_+\right) $ we conclude 
	\[
	C(N,p) cap_{1,p,N}\left( E\right) \geq  cap_{1-\frac{1}{p},p,N-1}\left( E\cap\pd\R^N_+\right).
	\]
	For the second assertion we consider the extension operator. Suppose $g\in \Schw\left( \pd\R^N_+\right)$ and $g\geq \chi_{E'}$. Then its extension $\bar{g}=Ex(g)$ belongs to $F_1^{p,2}\left( \R^N\right)= L^{1,p}\left( \R^N\right) $ with 
	\[
	\norm{\bar{g}}_{1,p,\R^N}=\norm{\bar{g}}_{ F_1^{p,2}\left( \R^N\right)}\leq C(N,p)\norm{g}_{F_{1-\frac{1}{p}}^{p,p}\left( \pd\R^N_+\right)}.
	\]
	Again, since $cap\left( E';F_{1-\frac{1}{p}}^{p,p}\left( \pd\R^N_+\right) \right) \sim cap_{1-\frac{1}{p},p,N-1}\left( E'\right)$ we conclude
	\[
	C(N,p)cap_{1-\frac{1}{p},p,N-1}\left( E'\right)\geq cap_{1,p,N}(E').
	\]
\end{proof}

Thanks to the above result, we can describe the relationship between the decomposition of a measure in $\mathfrak{M}_b\left( \R^{N-1}\right)$ and its representative in $\mathfrak{M}_b\left( \pd\R^{N}_+\right)$.
\begin{prop}\label{decomposition}
	Let $\mu\in \mathfrak{M}_b\left( \R^{N-1}\right) $ and let
	\[\mu=\mu_0+\mu_s\]
	be its decomposition with respect to $cap_{1-\frac{1}{p},p,N-1}$. Let $\bar{\mu}$ denote its identification as an element of $\mathfrak{M}_b\left( \pd\R^{N}_+\right)$.
	If 
	\[\bar{\mu}=\bar{\mu}_0+\bar{\mu}_s\]
	is the decomposition of $\bar{\mu}$ with respect to $cap_{1,p,N}$ then 
	\[\overline{\mu_0}=\bar{\mu}_0 \mbox{ and } \overline{\mu_s}=\bar{\mu}_s \ . \]
	In particular 
	\[\overline{\mu_s^\pm}=\bar{\mu}_s^\pm \ . \]
\end{prop}
\begin{proof}
	We only prove the first assertion since the second follows easily. Let $E$ be such that $cap_{1,p,N}(E)=0$ and $\bar{\mu}_s(E^c)=0$. By Proposition \ref{trace&extensionE} we have $cap_{1-\frac{1}{p},p,N-1}\left( E\cap\pd\R^N_+\right) =0$ and thus for any Borel set $A\subset \R^N$
	\begin{align*}
	\bar{\mu}_0\left( A\right) &= \bar{\mu}_0\left( A\cap E^c\right) =\bar{\mu}\left( A\cap E^c\right) = \mu\left( A\cap E^c \cap \pd\R^N_+\right) \\
	&=\mu_0\left( A\cap \pd\R^N_+\right) +\mu_s\left( A\cap E^c\cap \pd\R^N_+\right) \\
	&=\overline{\mu_0}\left( A\right) + \overline{\mu_s}\left( A\cap E^c\right) 
	\end{align*}
	and
	\begin{align*}
	\bar{\mu}_s\left( A\right) &= \bar{\mu}_s\left( A\cap E\right) =\bar{\mu}\left( A\cap E\right) = \mu\left( A\cap E \cap \pd\R^N_+\right) \\
	&=\mu_s\left( A\cap E \cap\pd\R^N_+\right) \\
	&=\overline{\mu_s}\left( A\cap E\right) \ .
	\end{align*}
	Similarly, let $E_0$ be such that $cap_{1-\frac{1}{p},p,N-1}\left( E_0\right) =0$ and $\mu_s\left( E_0^c\right) =0$. By Proposition \ref{trace&extensionE} we have $cap_{1,p,N}(E_0)=0$ and so
	\begin{align*}
	\overline{\mu_s}\left( A\right) &=\mu_s\left( A\cap\pd\R^N_+\right) =\mu_s\left( A\cap E_0\cap \pd\R^N_+\right) = \mu\left( A\cap E_0\cap\pd\R^N_+\right) \\
	&=\bar{\mu}\left( A\cap E_0\right) \\
	&=\bar{\mu}_s\left( A\cap E_0\right) \ .
	\end{align*}
	The previous inequality implies in particular that 
	\[
	\overline{\mu_s}\left( E^c\right) =\bar{\mu}_s\left( E^c\cap E_0\right) = \overline{\mu_s}\left( E^c\cap E_0\cap E\right) =0
	\]
	from which the proposition follows since then
	\[
	\overline{\mu_s}\left( A\cap E\right)=\overline{\mu_s}\left( A\right) \ , \ \overline{\mu_s}\left( A\cap E^c\right) =0 \ . 
	\]
\end{proof}

\section{Finer properties of $W^{1,p}$ functions} 

As we noted in the previous section, Theorem \ref{Trace&Extension} guarantees the existence of a trace $w(x',0) \in F_{\a-\frac{1}{p}}^{p,p}\left( \R^{N-1}\right)$ whenever $w(x',x_N)$ belongs to $F_\a^{p,q}\left( \R^N\right)$. Since
\[
F_1^{p,2}\left( \R^N\right) = L^{1,p}\left( \R^N\right) = W^{1,p}\left( \R^N\right) 
\]
we see that every function $w(x',x_N)\in W^{1,p}\left( \R^N\right) $ has a trace $w(x',0)$ in $F_{1-\frac{1}{p}}^{p,p}\left( \R^{N-1}\right) $.

Since we want to integrate along the boundary of $\R^N_+$ we study the regularity of these traces. The following proposition shows that functions in $W^{1,p}\left( \R^N_+\right) $ also have well defined traces and that, by selecting and adequate representative, we can assume they are quasi-continuous. 
\begin{prop}\label{qcTrace}
	Let $\w\in W^{1,p}\left( \R^N_+\right)$. Then $\w$ has a $cap_{1,p,N}-$ quasi-continuous representative, defined in $\overline{\R^N_+}$, which is unique up to sets of zero $cap_{1,p,N}$ capacity. In particular, identifying $\w$ with this representative, the trace of $\w$ is $cap_{1-\frac{1}{p},p,N-1}-$ quasi-continuous and unique $cap_{1-\frac{1}{p},p,N-1}-q.e.$ in $\pd\R^N_+$.
\end{prop}
\begin{proof}
	Since $\R^N_+$ is an extension domain we consider $\w$ as an element in $W^{1,p}\left( \R^N\right) $. Recalling that $W^{1,p}\left( \R^N\right) =L^{1,p}\left( \R^N\right)$, we obtain the existence of a $cap_{1,p,N}-$ quasi-continuous representative which is unique in $\overline{\R^N_+}$ modulo sets of zero capacity (see Theorem 6.1.4 of \cite{AdamsHedberg}). In view of Proposition \ref{trace&extensionE} the rest of the proposition follows easily.
\end{proof}

\begin{remark}\label{remark_quasincontinous_trace}
	Thanks to the above proposition from now on we identify function in $W^{1,p}\left( \R^N_+\right)$ with their $cap_{1,p,N}-$ quasi-continuous representative in $\overline{\R^N_+}$ and refer to their $cap_{1-\frac{1}{p},p,N-1}-$ quasi-continuous trace in $\pd\R^N_+$ whenever necessary. Note that this result also applies to functions in $W^{1,p}\left( \R^N\right)$, or in $W^{1,p}_{0}\left( B_M\right)$ by identifying elements in this space with their extension by zero.
\end{remark}

\begin{remark}\label{qcTrace_local}
	For a function $\w \in W^{1,p}_{loc}\left( \R^N_+\right)$ one can still define the boundary values of $\w$. Indeed, for any fixed $m$ we have $\w\in W^{1,p}\left( B_m\cap \R^N_+\right) $ and thus we can extend $\w$ to a function in $W^{1,p}\left( \R^N\right) $ first by even reflection and then using that $B_m$ is an extension domain. The resulting extension has a $cap_{1,p,N}-$ quasi-continuous representative, which we call $\w_m$, coinciding with $\w$ $a.e.$ in $B_m\cap \R^N_+$. If we take $m'>m$ then any $cap_{1,p,N}-$ quasi-continuous representative $\w_{m'}$ coincides with $\w_m$ $a.e.$ in $B_m\cap \R^N_+$ and thus, by Theorem 6.1.4 of \cite{AdamsHedberg}, $\w_{m'}=\w_m$ $cap_{1,p,N}-q.e.$ in $\overline{B_m\cap \R^N_+}$. Hence, from now on, we identify functions in $W^{1,p}_{loc}\left( \R^N_+\right)$ with this locally defined, and $cap_{1,p,N}-q.e.$ unique, $cap_{1,p,N}-$ quasi-continuous representative in $\overline{\R^N_+}$. In particular, if $x'\in\R^{N-1}$, we define the trace $\w(x',0)$ to be the value at $(x',0)$ of \textit{any} representative $\w_m$ such that $\abs{x'}<m$. By the above considerations, and Proposition \ref{trace&extensionE}, the trace is $cap_{1-\frac{1}{p},p,N-1}-$quasi-continuous and unique $cap_{1-\frac{1}{p},p,N-1}-q.e.$ in $\pd \R^N_+$.
\end{remark}

We shall make use of the following propositions regarding integrability and convergence with respect to measures in $\mathfrak{M}_0\left( \pd\R^N_+\right) $. 

\begin{prop}\label{L^1mu_0}
	Let $\mu \in \mathfrak{M}_0\left( \pd\R^N_+\right)$ and let $w\in W^{1,p}\left( \R^N\right)$. Then $w$ is measurable with respect to $\mu$. Furthermore, if the trace of $w$ belongs to $L^\infty\left( \pd\R^N_+\right)$ then it belongs to $L^\infty\left( \R^N;d\mu\right) $.
\end{prop}
\begin{proof}
	Since $w\in W^{1,p}\left( \R^N_+\right) $ we have by Proposition \ref{qcTrace} that $w$ has a $cap_{1-\frac{1}{p},p,N-1}-$ quasi-continuous trace in $\pd\R^N_+$, which is the restriction of any $cap_{1,p,N}-$ quasi-continuous representative of $w$. Since every $cap_{1,p,N}-$ quasi-continuous function coincides $cap_{1,p,N}-q.e.$ with a Borel function it follows that $w$ is measurable with respect to any (Radon) measure $\mu\in \mathfrak{M}_0\left( \pd\R^N_+\right)$. If moreover $\abs{w}\leq k$ $a.e.$ on $\pd\R^N_+$ then it holds $\abs{w}\leq k$ $cap_{1-\frac{1}{p},p,N-1}-q.e.$ on $\pd\R^N_+$. That this is so follows from an application of Theorem 6.1.4 of \cite{AdamsHedberg} to the $cap_{1-\frac{1}{p},p,N-1}-$ quasi-continuous functions $(w-k)_+$ and $(w+k)_-$. Since $\mu$ is absolutely continuous with respect to $cap_{1-\frac{1}{p},p,N-1}$ we see that $\abs{w}\leq k$ $\mu-a.e.$ (see Proposition \ref{decomposition}).
\end{proof}
One can similarly obtain the following proposition
\begin{prop}\label{L^infty_mu_0}
	Let $\mu \in \mathfrak{M}_0\left( \R^N\right)$ and let $w\in W^{1,p}\left( \R^N\right)$. Then $w$ is measurable with respect to $\mu$. Furthermore, if $w$ belongs to $L^\infty\left( \R^N\right)$ then it belongs to $L^\infty\left( \R^N;d\mu\right) $.
\end{prop}

\begin{remark}\label{corolary_to_T_6.1.4}
	We will use the following fact: if $u\leq v$ $a.e.$ in $\R^N$, where $u$ and $v$ are $cap_{\a,p,N}-$ quasi-continuous functions, then $u\leq v$ $cap_{\a,p,N}-q.e.$ in $\R^N$. This can be proven by applying Theorem 6.1.4 of \cite{AdamsHedberg} to the quasi-continuous function $w=\max\left\lbrace u-v,0\right\rbrace $, which satisfies $w=0$ $a.e.$ in $\R^N$.
\end{remark}

Combining the last proposition with Lebesgue's Dominated Convergence Theorem we obtain:
\begin{prop}\label{LDC_mu_0}
	Let $f_n\goesto f$ $cap_{1,p,N}-q.e.$ in $\R^N$ with $f_n$ in $W^{1,p}\left( \R^N\right)\cap L^\infty\left( \R^N\right) $ and uniformly bounded in $ L^\infty\left( \R^N\right)$. Then for any measure $\mu \in \mathfrak{M}_0\left( \R^N\right)$, $f_n\goesto f$ $\mu-$q.e. and
	\[
	\lim_{n\goesto\infty}\int_{\R^N}f_n d\mu = \int_{\R^N} fd\mu .
	\]
\end{prop}

The following result is Proposition 2.8 in \cite{DMOP99}. It is a consequence of Egorov's Theorem.
\begin{prop}\label{prop_Egorov}
	Let $\O$ be a bounded open subset of $\R^N$. Let $\rho_\e$ be a sequence in $L^1\left( \O\right) $ that converges to $\rho$ weakly in $L^1\left( \O\right) $, and let $\sigma_\e$ be a sequence uniformly bounded in $L^\infty\left( \O\right) $ that converges to $\sigma$ $a.e.$ in $\O$. Then,
	\[
	\lim_{\e\goesto0}\int_\O \rho_\e\sigma_\e dx=\int_\O \rho\sigma dx.
	\]
\end{prop}

\section{$p-$ superharmonic functions} Although not the focus of this work, we will use several results concerning $p-$ superharmonic functions, especially on the relationship between them and renormalized solutions. We will give proper references whenever necessary, but most of the results are classic (see \cite{HKM}). Here we record some definitions and basic properties. 

Let $\O$ be any domain. A $p-$ superharmonic function is a lower semicontinuous function $u:\O\goesto (-\infty,\infty]$, not identically infinite, such that for all open sets $\O'\subset\subset \O$ and for all $h$ $p$-harmonic in $\O'$ and continuous in $\overline{\O'}$ we have that $h\leq u$ on $\pd\O'$ implies $h\leq u$ in $\O'$. 

It is well-known that if $u$ is $p-$ superharmonic then its truncation $\min\left\lbrace u,k\right\rbrace $ belongs to $W^{1,p}_{loc}\left( \O\right)$. This allows us to define its gradient in the same generalized sense as we will do for renormalized solutions (see Chapter \ref{Renormalized solution}), and in particular, it makes sense to define $-\lap_p u$ in the sense of distributions. In particular, when we say that a $p-$ superharmonic function $u$ solves $-\lap_p u=\mu$ in $\O$ for some (not necessarily bounded) Radon measure $\mu$, we mean it precisely in the sense of distributions, where the derivative of $u$ is to be understood in the generalized sense. It is also known that if $u$ is $p-$ superharmonic function in $\O$ then $-\lap_pu$ is a nonnegative distribution, and so there exists a nonnegative Radon measure $\mu$ such that $-\lap_p u=\mu$ in $\D'\left( \O\right) $.

Finally, we remark that when we say that a $p-$ superharmonic function $u$ solves $-\lap_p u= g(u)\sigma + \mu$ in $\O$, for some Radon measures $\mu$ and $\sigma$, we imply that $g(u) \in L^1_{loc}\left( \O,d\sigma\right)$ so that the right hand side is actually a Radon measure. 

\chapter{Renormalized solutions}\label{Renormalized solution}

\section{Renormalized solutions in bounded domains}

We start with the definition of renormalized solution given in \cite{DMOP99} for bounded domains. In order to do this we first need to generalize the definition of $\nabla u$. 

Let $T_k(s)$ be truncation by $k$, i.e., $T_k(s)=\min(k,\max(-k,s))$. Then for any measurable and $a.e.$ finite $u$ such that $T_k(u)\in W^{1,p}_{0}\left( \O\right) $ for every $k>0$ there exists a measurable vector-valued function $v:\O\goesto\R^N$ such that 
\[
\nabla T_k(u)=v\chi_{\left\lbrace\abs{u}<k\right\rbrace }
\]
$a.e.$ in $\O$ for all $k>0$ (see \cite{BBGGPV95}, Lemma 2.1). This function is unique $a.e.$ and so we define $v$ as the gradient of $u$ and write $\nabla u =v$. One similarly obtains that if $T_k(u)\in W^{1,p}_{loc}\left( \R^N_+\right) $ for every $k>0$ then there exists a measurable vector-valued function $v:\R^N_+\goesto\R^N$ such that 
\[
\nabla T_k(u)=v\chi_{\left\lbrace\abs{u}<k\right\rbrace }
\]
$a.e.$ in $\R^N_+$ for all $k>0$.

\begin{remark}\label{remark_on_distributional gradient}
	 We note that, in general, $v$ is not the gradient of $u$ used in the definition of Sobolev spaces. In fact, $u$ may not even belong to $L^1_{loc}\left( \O\right) $ (see \cite{DMOP99} for details).
\end{remark}

\begin{defn}\label{definition_DMOP99}
	Let $\O$ be a bounded domain in $\R^N$. Let $\mu \in \mathfrak{M}_b\left( \O\right) $ have a decomposition $\mu = \mu_0 + \mu_s$ with respect to $cap_{1,p,N}$. Then a function $u$ is a \textit{renormalized solution} of
	\be\label{renormalized_solution_DMOP99}
	\begin{cases}
		-\lap_p u = \mu & \mbox{ in } \O \\
		u=0 & \mbox{ on } \pd\O 
	\end{cases}
	\ee
	if 
	\begin{enumerate}
		\item $u$ is measurable, finite $a.e.$, and $T_k(u)\in W^{1,p}_{0}\left( \O\right) $ for all $k>0$;
		\item $\abs{\nabla u}^{p-1}\in L^q\left( \O\right) $ for all $1\leq q<\frac{N}{N-1}$;
		\item there holds
		\[
		\int_{\O}\abs{\nabla u}^{p-2}\nabla u\cdot\nabla w dx= \int_{\O} w d\mu_0 + \int_{\O}w^{+\infty}d\mu_s^+ - \int_{\O}w^{-\infty}d \mu_s^- 
		\]
		for all $w\in W^{1,p}_0\left( \O\right) \cap L^\infty\left( \O\right) $ satisfying the following condition: there exist $k>0$, $r>N$, and functions $w^{\pm\infty}\in W^{1,r}\left( \O\right) \cap L^{\infty}\left( \O\right) $ such that
		\[
		\begin{cases}
		w=w^{+\infty} & \mbox{ $a.e.$ in } \left\lbrace x\in \O \ : \ u>k\right\rbrace \\
		w=w^{-\infty} & \mbox{ $a.e.$ in } \left\lbrace x\in \O \ : \ u<-k\right\rbrace  .
		\end{cases} 
		\]
	\end{enumerate}
\end{defn}

\begin{remark}
	Note that the set of functions $w$ for which $(3)$ holds is not empty. Indeed, it contains $C_0^\infty\left( \O\right) $ since the condition is satisfied by any $w$ in $C_0^\infty\left( \O\right) $ choosing any $k>0$ and $r>N$, and setting $w=w^{+\infty}=w^{-\infty}$. But there are more admissible functions. In particular, $T_k(u)$ is admissible with $w^{\pm\infty}=\pm k$. 
\end{remark}

\begin{remark}\label{remark_on_equivalent_definitions_DMOP99}
	Theorem 2.33 of \cite{DMOP99} shows that there are several equivalent definitions of renormalized solution. In particular, the last condition above can be replaced with the following one: for every $k>0$ there exists two nonnegative measures $\l_{k}^+$, $\l_{k}^- \in \mathfrak{M}_0\left(\O\right) $ supported in $\left\lbrace u=k\right\rbrace $ and $\left\lbrace u=-k\right\rbrace $ respectively, such that $\l_{k}^\pm \goesto \mu _s^\pm $, as $k\goesto \infty$, in the narrow topology of $\mathfrak{M}_b\left(\O\right)$, and the truncations $T_k(u)$ satisfy
	\[
	\int_{\left\lbrace\abs{u}<k\right\rbrace }\abs{\nabla T_k(u)}^{p-2}\nabla T_k(u) dx\cdot \nabla v = \int_{\left\lbrace\abs{u}<k\right\rbrace }vd\mu_0 + \int_{\O} v d\l_{k}^+ - \int_{\O} v d\l_{k}^- 
	\]	
	for every $v\in W^{1,p}_0\left( \O\right) \cap L^\infty\left( \O\right) $.
	Whenever convenient we use this equivalent formulation.
\end{remark}

\begin{remark}\label{remark_quasicontinuity_DMOP99}
	The conditions stated in definition \ref{definition_DMOP99} imply that any renormalized solution has a $cap_{1,p,N}-$ quasi-continuous representative which is in fact finite $cap_{1,p,N}-q.e.$ in $\O$ (see remark 2.18 of \cite{DMOP99}). We always identify renormalized solutions with this representative.
\end{remark}

The following theorem is proved in \cite{DMOP99} using Lemma 4.1 and 4.2 of \cite{BBGGPV95}.

\begin{thm}\label{DMMOP99_level_set_estimates}
	Let $u$ be a renormalized solution of \eqref{renormalized_solution_DMOP99}. Then
	\be\label{estimate1_DMMOP99}
	\int_{\left\lbrace n\leq\abs{u}<n+k\right\rbrace } \abs{\nabla u}^pdx\leq k\abs{\mu}\left( \O\right)  \ , \ \forall n\geq 0, \ k>0 .
	\ee
	If $p<N$ then for every $k>0$,
	\be	\label{estimate2_DMMOP99}
	\abs{\left\lbrace\abs{u}>k\right\rbrace }\leq C(N,p)\frac{\left( \abs{\mu}\left( \O\right) \right) ^{\frac{N}{N-p}}}{k^{\frac{N(p-1)}{N-p}}},
	\ee
	\be\label{estimate3_DMMOP99}
	\abs{\left\lbrace\abs{\nabla u}>k\right\rbrace }\leq C(N,p)\frac{\left( \abs{\mu}\left( \O\right) \right) ^{\frac{N}{N-1}}}{k^{\frac{N(p-1)}{N-1}}} .
	\ee
	If $p=N$ then for every $k>0$,
	\be	\label{estimate4_DMMOP99}
	\abs{\left\lbrace\abs{u}>k\right\rbrace }\leq C(r,N,p)\frac{\left( \abs{\mu}\left( \O\right) \right) ^{r}}{k^{r(p-1)}},
	\ee
	for every $r>1$, and 
	\be\label{estimate5_DMMOP99}
	\abs{\left\lbrace\abs{\nabla u}>k\right\rbrace }\leq C(s,N,p)\frac{\left( \abs{\mu}\left( \O\right) \right) ^{\frac{N}{N-1}}}{k^{s}} .
	\ee
	for every $s<N$.
\end{thm}

We note explicitly that the above constants do not depend on the domain $\O$. Note also that by putting $n=0$ in \eqref{estimate1_DMMOP99} we get
\be
\int_{\left\lbrace\abs{u}<k\right\rbrace } \abs{\nabla u}^pdx\leq k\abs{\mu}\left( \O\right)   \ , \ \forall  k>0 \ .
\ee	

The following result is proven in Section 5.1 of \cite{DMOP99} as a first step in the proof of their stability result. It will be useful for us later when dealing with nonlinear terms. 

\begin{thm}\label{DMOP99_convergence_of_u_n_to_limit}
	Let $u_n$ be renormalized solutions to problem \eqref{renormalized_solution_DMOP99} with respective measures $\mu_n \in \mathfrak{M}_b\left( \O\right) $. Assume $\norm{\mu_n}_{\mathfrak{M}_b}$ are uniformly bounded. Then there exists a function $u$ such that, up to a subsequence, $u_n\goesto u$ $a.e.$ in $\O$. Moreover, $u$ satisfies $(1)$ and $(2)$ of the definition of renormalized solution, as well as all the estimates stated in Theorem \ref{DMMOP99_level_set_estimates} (with $\sup\norm{\mu_n}_{\mathfrak{M}_b}$ instead of $\norm{\mu}_{\mathfrak{M}_b}$), and  
	\begin{enumerate}
		\item $\nabla T_k(u_n)\goesto \nabla T_k(u)$ and $\nabla u_n \goesto \nabla u$ $a.e.$ in $\O$,
		\item $\abs{\nabla u_n}^{p-2}\nabla u_n \goesto \abs{\nabla u}^{p-2}\nabla u$ strongly in $\left( L^q\left( \O\right)\right) ^N$ for any $1\leq q<\frac{N}{N-1}$,
		\item $	T_k(u_n)\goesto T_k(u)$ weakly in $W_0^{1,p}\left( \O\right)$.
	\end{enumerate}
\end{thm}
\begin{remark}\label{remark_quasicontinuity_DMOP99_limit}
	It follows from Remark 2.11 of \cite{DMOP99} that the function $u$ in the above theorem has a $cap_{1,p,N}-$ quasi-continuous representative which is in fact finite $cap_{1,p,N}-q.e.$ in $\O$. We identify $u$ with this representative.  
\end{remark}

\section{Local renormalized solutions}

A closely related concept is the one of \textit{local} renormalized solutions (see \cite{B-V03}, \cite{Veron16}) on domains which are not necessarily bounded. It is closer to our definition of renormalized solution of \eqref{equation}, and we will use it in the sequel. We remark that the derivative here is to be understood in the same generalized sense as described previously.

\begin{defn}\label{definition_local_solution}
	Let $\O$ be any domain in $\R^N$. Let $\mu \in \mathfrak{M}_b\left( \O\right) $ have a decomposition $\mu = \mu_0 + \mu_s$ with respect to $cap_{1,p,N}$. Then a function $u$ is a \textit{local renormalized solution} of
	\[
		-\lap_p u = \mu \ \mbox{ in } \O 
	\]
	if
	\begin{enumerate}
		\item $u$ is measurable, finite $a.e.$, and $T_k(u)\in W^{1,p}_{loc}\left( \O\right) $ for all $k>0$;
		\item $\abs{\nabla u}^{p-1}\in L^q_{loc}\left( \O\right) $ for all $1\leq q<\frac{N}{N-1}$;
		\item $\abs{u}^{p-1} \in L^q_{loc}\left( \O\right) $ for all $1<q<\frac{N}{N-p}$ ($1<q<\infty$ if $p=N$);
		\item there holds
		\[
		\int_{\O}\abs{\nabla u}^{p-2}\nabla u\cdot\nabla wdx = \int_{\O} w d\mu_0 + \int_{\O}w^{+\infty}d\mu_s^+ - \int_{\O}w^{-\infty}d \mu_s^-
		\]
		for all $w\in W^{1,p}\left( \O\right) \cap L^\infty\left( \O\right) $ compactly supported in $\O$ satisfying the following condition: there exist $k>0$, $r>N$, and functions $w^{\pm\infty}\in W^{1,r}\left( \O\right) \cap L^{\infty}\left( \O\right) $ such that
		\[
		\begin{cases}
		w=w^{+\infty} & \mbox{ $a.e.$ in } \left\lbrace x\in \O \ : \ u>k\right\rbrace \\
		w=w^{-\infty} & \mbox{ $a.e.$ in } \left\lbrace x\in \O \ : \ u<-k\right\rbrace  .
		\end{cases} 
		\]
	\end{enumerate}
\end{defn}

\begin{remark}
	We remark that all functions $w$ in $C_0^\infty\left( \O\right) $ are admissible functions for $(4)$. Note however that $T_k(u)$ is no longer a valid test function. On the other hand, if $w\in C_0^\infty\left( \O\right) $ then $wT_k(u)$ is admissible with $w^{\pm\infty}=\pm k w$. 
\end{remark}

\begin{remark}
	Just as in the case of Definition \ref{renormalized_solution_DMOP99}, condition $(4)$ can be replaced by some other equivalent conditions (see Theorem 2.2 of \cite{B-V03}). We will use this fact in the proof of Lemma \ref{general_existence_global_solution}.
	On the other hand, our definition of local renormalized solution is not exactly the same as the definition in \cite{B-V03} since there the author does not require that $\mu$ is bounded. We have chosen to add this extra condition since we will need it when solving problem \eqref{equation}. 
\end{remark}

\begin{remark}\label{remark_on_p_superharmonic_representative}
	A fact that we will use frequently is that if $\mu$ is nonnegative and $u$ is a local renormalized solution of $-\lap_pu=\mu$ in $\O$, then $u$ coincides $a.e.$ with a $p-$ superharmonic function solving the same equation (see Theorem 4.3.2 of \cite{Veron16}).
\end{remark}

\begin{remark}\label{remark_on_L^p_estimate}
	Note that the estimates in Theorem \ref{DMMOP99_level_set_estimates} show that if $\O$ is bounded then any renormalized solution of \eqref{renormalized_solution_DMOP99} is also a local renormalized solution of the corresponding equation. Indeed, we only need to show $(3)$. To this end, we recall the known identity
	\begin{align*}
	\int_\O \abs{u}^\a dx&= \int_\O\int_0^\infty \a t^{\a-1}\chi(t)_{[0,\abs{u}]}dtdx=\int_0^\infty\int_\O\a t^{\a-1}\chi(x)_{\left\lbrace\abs{u}\geq t\right\rbrace }dxdt \\ &= \int_0^\infty\a t^{\a-1}\abs{\left\lbrace\abs{u}\geq t \right\rbrace } dt
	\end{align*}
	which holds for \textit{any} measurable function $u$, and any $\a>0$. From this identity one obtains the estimate
	\be\label{eq_estimate_L^p_from_Level_set}
	\int_{\O} \abs{u}^\a dx \leq t_0^\a\abs{\O} + \a\int_{t_0}^\infty t^{\a-1}\abs{\left\lbrace\abs{u}\geq t \right\rbrace } dt .
	\ee
	In particular, if $u$ is a renormalized solution in a bounded domain $\O$, and $p<N$, then combining the above estimate and \eqref{estimate2_DMMOP99} we have 
	\[
	\int_{\O}\abs{u}^sdx \leq \abs{\O} + sC(p,N,\mu)\int_1^\infty t^{s-1} t^{-\frac{N(p-1)}{N-p}}dt
	\]
	which is finite if $1<s<\frac{N(p-1)}{N-p}$. If $p=N$ then we use instead estimate \eqref{estimate4_DMMOP99} obtaining, for any fixed $r>1$, the condition $1<s<r(p-1)$. Hence in this case any $1<s<\infty$ is allowed. 
\end{remark}

\section{Renormalized solutions to the Neumann problem in the half-space}

We now define a renormalized solution to \eqref{equation}. Recall that by the discussion of the previous chapter, any measure $\mu\in \mathfrak{M}_b\left( \pd\R^N_+\right) $ can be decomposed uniquely as 
\[
\mu=\mu_0+\mu_s^+-\mu_s^-
\]
where $\mu_0$ is absolutely continuous with respect to $cap_{1,p,N}$, and $\mu_s^\pm$ are singular with respect to $cap_{1,p,N}$ and nonnegative.

\begin{defn}\label{defn}
	Let $\mu\in \mathfrak{M}_b\left( \pd\R^N_+\right)$ and $g:\R\goesto\R$. A function $u$ defined in $\R^N_+$ is a \textit{renormalized solution} to \eqref{equation} provided the following holds:
	\begin{enumerate}
		\item $u$ is measurable, finite $a.e.$, and $T_k(u)\in W^{1,p}_{loc}\left( \R^N_+\right) $ for all $k>0$;
		\item $\abs{\nabla u}^{p-1}\in L_{loc}^q\left( \R^N_+\right) $ for all $1\leq q<\frac{N}{N-1}$;
		\item $\abs{u}^{p-1}\in L_{loc}^q\left( \R^N_+\right) $ for all $1<q<\frac{N}{N-p}$ ($1<q<\infty$ if $p=N$);
		\item $u$ is finite $a.e.$ in $\pd\R^N_+$, and $g(u)\in L^1\left( \pd\R^N_+\right) $;
		\item there holds
	\end{enumerate}
\end{defn}
\[
\int_{\R^N_+}\abs{\nabla u}^{p-2}\nabla u\cdot\nabla wdx + \int_{\pd\R_+^N} g(u) wdx' = \int_{\pd\R^N_+} w d\mu_0 + \int_{\pd\R^N_+}w^{+\infty}d \mu_s^+ - \int_{\pd\R^N_+}w^{-\infty}d\mu_s^-
\]
\begin{enumerate}[label={}]
	\item for all $w\in W^{1,p}\left( \R^N_+\right) $ compactly supported in $\overline{\R^N_+}$, with trace in $L^\infty\left( \pd\R^N_+\right) $, and satisfying the following condition: there exist $k>0$, $r>N$, and functions $w^{\pm\infty}\in W^{1,r}\left( \R^N_+\right) $ such that
	\[
	\begin{cases}
	w=w^{+\infty} & a.e. \mbox{ in } \left\lbrace x\in \R^N_+ \ : \ u>k\right\rbrace \\
	w=w^{-\infty} & a.e. \mbox{ in } \left\lbrace x\in \R^N_+ \ : \ u<-k\right\rbrace  .
	\end{cases} 
	\]
\end{enumerate}

\begin{remark}\label{remark_quasicontinuity}
	We remark that it makes sense to talk about the boundary values of a renormalized solution since, in fact, any $a.e.$ finite and measurable function $u$ defined in $\R^N_+$ such that $T_k(u) \in W^{1,p}_{loc}\left( \R^N_+\right)$ for all $k>0$ has a locally defined $cap_{1,p,N}-$ quasi-continuous representative in $\overline{\R^N_+}$ which, however, could be infinite on a set of positive $cap_{1,p,N}$ capacity. Indeed, by Remark \ref{qcTrace_local} we can locally identify $T_k(u)$ with a $cap_{1,p,N}-$ quasi-continuous representative in $\overline{\R^N_+}$. Then, it can be directly verified that $v=\sup_{k\in \N}T_k(u)$ defines (locally) a $cap_{1,p,N}-$ quasi-continuous function that coincides with $u$ $a.e.$ in $\R^N_+$ and which is unique $cap_{1,p,N}-q.e$ in $\overline{\R^N_+}$. Notice that, in general, $v$ may be infinite on a set of positive $cap_{1,p,N}$ capacity and so its trace could be infinite. We remark that similar considerations hold for $a.e.$ finite and measurable functions $u$ defined in $\R^N$ such that $T_k(u) \in W^{1,p}_{loc}\left( \R^N\right)$.
	From now on, we always identify renormalized solutions to \eqref{equation} with their $cap_{1,p,N}-$ quasi-continuous representative in $\overline{\R^N_+}$. In particular, under this identification, the trace of $u$ is $cap_{1-\frac{1}{p},p,N-1}-$ quasi-continuous and unique $cap_{1-\frac{1}{p},p,N-1}-q.e.$ in $\pd\R^N_+$. Since $u$ could be infinite on a set of positive capacity, we explicitly ask that the trace must be finite $a.e.$ in $\pd\R^N_+$. We will show below that in fact renormalized solutions are always finite $cap_{1,p,N}-q.e.$ in $\overline{\R^N_+}$.
\end{remark}

\begin{remark}\label{remark_test_trace}
	If we consider $cap_{1,p,N}-$ quasi-continuous representatives in $\overline{\R^N_+}$, then the condition $w=w^{+\infty}$ $a.e.$ in $\left\lbrace x\in \R^N_+ : u>k\right\rbrace $ implies that $w=w^{+\infty}$ $cap_{1,p,N}-q.e$ in $\left\lbrace x\in \overline{\R^N_+} : u>k\right\rbrace $. To see this, apply Theorem 6.1.4 of \cite{AdamsHedberg} to extend $(w-w^{+\infty})(u-k)^+=0$ from $a.e.$ in $\overline{\R^N_+}$ to $cap_{1,p,N}-q.e$ in $\overline{\R^N_+}$ (see also Remark \ref{corolary_to_T_6.1.4}). It follows that $w=w^{+\infty}$ $cap_{1-\frac{1}{p},p,N-1}-q.e$ in $\left\lbrace x\in \pd\R^N_+ : u>k\right\rbrace $, and in particular $w=w^{+\infty}$ $a.e.$ in $\left\lbrace x\in \pd\R^N_+ : u>k\right\rbrace $. Similarly, $w=w^{-\infty}$ $a.e.$ in $\left\lbrace x\in \pd\R^N_+ : u<-k\right\rbrace $.
\end{remark}

We verify that under the given assumptions all the integrals above are well defined and finite. The first integral on the left hand side can be divided into three integrals with domains of integration given by $\left\lbrace x:  \abs{u}\leq k\right\rbrace $, $\left\lbrace x :  u>k\right\rbrace $, and $\left\lbrace x  :  u<-k\right\rbrace $. In the first case $T_k(u)=u$ so $\abs{\nabla u}^{p-1}\in L^{p'}_{loc}\left( \R^N_+\right) $ and the integral is finite since $w \in W^{1,p}\left( \R^N_+\right) $ has compact support. For the second case $w=w^{+\infty} \in W^{1,r}\left( \R^N_+\right) $ and $r>N$ implies $r'<\frac{N}{N-1}$ so by assumption $\abs{\nabla u}^{p-1} \in L^{r'}_{loc}\left( \R^N_+\right) $ and the integral is also finite since we integrate over the support of $w$. The third case can be treated similarly. The second integral on the left hand side is obviously finite since $g(u(x',0))\in L^1\left( \pd\R^N_+\right) $ while $w\in L^\infty\left( \pd\R^N_+\right) $. 

As for the right hand side, observe first that since $r>N$ we have $w^{\pm\infty}\in C\left( \overline{\R^N_+}\right) $ with the supremum norm. Since $\mu_s^\pm$ are bounded we conclude that the integrals with respect to the singular measures are well defined and finite. For the remaining integral Proposition \ref{qcTrace} guarantees that $w$ has a well defined trace, while Proposition \ref{L^1mu_0} and the boundedness of $\mu_0$ gives $w\in L^1\left( \pd\R^N_+;d\mu_0\right) $.

\begin{remark}\label{renormalized_solutions_are_local}
	It follows directly from the definitions that if $u$ is a renormalized solution of \eqref{equation} then $\bar{u}$, the extension of $u$ by even reflection across $\pd\R^N_+$, is a local renormalized solution of $-\lap_p \bar{u} = \tilde{\mu}:= 2\mu-2g(u)\mathcal{H}$ in $\R^N$ (where $g(u)\mathcal{H}$ has the meaning indicated in Chapter \ref{Capacities}). 
\end{remark}

\begin{remark}\label{remark_on_existence_distributional_gradient}
	We have noted in Remark \ref{remark_on_distributional gradient} that $\nabla u$ is not, in general, the gradient of $u$ in the usual sense used in the definition of Sobolev spaces. However, it can be shown that if $\nabla u\in \left( L^q_{loc}\left( \R^N\right) \right)^N$ for some $1\leq q \leq p$ then $u\in W_{loc}^{1,q}\left( \R^N\right) $ and $\nabla u$ is the usual gradient of $u$ (see Remark 2.10 of \cite{DMOP99}). In particular, when $p=N$ the definition of renormalized solution implies that $u\in W_{loc}^{1,1}\left( \R^N\right) $ and the gradient of $u$ coincides with the usual definition. 
\end{remark}

In the definition of renormalized solution we assumed $u$ is finite $a.e.$ in $\pd\R^N_+$. In the case $g\equiv 0$ this assumption could have been dropped. Moreover, the condition could also be removed by assuming $g$ is a function defined on the extended real line. However, we now show that whenever $g(u)\in L^1\left( \pd\R^N_+\right) $ then $u$ must be finite  $a.e.$ in $\pd\R^N_+$. Indeed, by our definition of trace, and in view of Proposition \ref{trace&extensionE} and remark \ref{renormalized_solutions_are_local}, it will be enough to show that local renormalized solutions of $-\lap_pu=\mu$ in $\R^N$  are finite $cap_{1,p,N}-q.e.$ in $\R^N$. We will obtain this as a consequence of the following local version of the estimates on level sets stated in Theorem \ref{DMMOP99_level_set_estimates}.

\begin{thm}\label{local_level_set_estimates}
	Let $u$ be a local renormalized solution of $-\lap_pu=\mu$ in $\O$, and let $\O'$ be such that  $\O'\subset \subset \O$. Then
	\be\label{estimate1}
	\int_{\left\lbrace\abs{u}<k\right\rbrace \cap \O'} \abs{\nabla u}^pdx\leq C\left(p,\O,\O',\mu,u\right)k  \ , \ \forall \ k>0 ,
	\ee
	and there exists $k_0(u,\O,\O',p)$ such that: if $p<N$ then for every $k>k_0$,
	\be	\label{estimate2}
	\abs{\left\lbrace\abs{u}>k\right\rbrace \cap \O'}\leq C(N,p,\O,\O',\mu,u)k^{-\frac{N(p-1)}{N-p}},
	\ee
	\be\label{estimate3}
	\abs{\left\lbrace\abs{\nabla u}>k\right\rbrace \cap \O'}\leq C(N,p,\O,\O',\mu,u)k^{-\frac{N(p-1)}{N-1}} ;
	\ee
	if $p=N$ then for every $k>k_0$,
	\be	\label{estimate4}
	\abs{\left\lbrace\abs{u}>k\right\rbrace \cap \O'}\leq C(r,N,p,\O,\O',\mu,u)k^{-r(p-1)},
	\ee
	for every $r>1$, and 
	\be\label{estimate5}
	\abs{\left\lbrace\abs{\nabla u}>k\right\rbrace \cap \O'}\leq C(s,N,p,\O,\O',\mu,u)k^{-s} ,
	\ee
	for every $s<N$.
\end{thm}
\begin{proof}
	Choose $\phi\in C_0^\infty\left( \O\right)$ such that $0\leq \phi\leq 1$, $\phi\equiv 1$ in $\O'$, and $\supp{\phi}\subset \O_0 \subset\subset \O$ for some $\O_0$. Then, testing against $\phi T_k(u)$ we obtain
	\begin{multline*}
	\int_{\O_0}\abs{\nabla T_k(u)}^p\phi dx + \int_{\O_0}T_k(u)\abs{\nabla u}^{p-2}\nabla u \cdot \nabla \phi dx = \\ \int_{\O_0} T_k(u)\phi d\mu_0 + \int_{\O_0} k\phi d\left( \mu_s^++\mu_s^-\right) 
	\end{multline*}
	and so
	\[
	\int_{\O'}\abs{\nabla T_k(u)}^pdx \leq  k\norm{\nabla u}_{L^{p-1}\left( \O_0\right) }\norm{\nabla \phi}_\infty + k\norm{\mu}_{\mathfrak{M}_b}= C(\O_0,p,\mu,u) k
	\]
	which is estimate \eqref{estimate1}.
	
	Next, we observe that since $u\in L^{s}\left( \O_0\right)$ for some $s>0$ Chebyshev's inequality gives
	\[
	\abs{\left\lbrace\abs{u}>k\right\rbrace \cap \O_0}\leq C(u,\O_0,p)k^{-s} .
	\]
	Hence, we can choose $k_0$ such that
	\[
	\abs{\left\lbrace  \abs{u}>\frac{k}{2}\right\rbrace \cap \O_0}\leq \frac{1}{4}\abs{\O'}
	\]
	for all $k\geq k_0$. Define $c_{k}=(T_k(u))_{\O'}$: the average of $T_k(u)$ in $\O'$. Then we estimate 
	\[
	\abs{c_{k}}\leq \frac{1}{\abs{\O'}}\left( \int_{\O'\cap \left\lbrace\abs{u}\leq k/2\right\rbrace }\abs{T_k(u)}dx + \int_{\O'\cap \left\lbrace\abs{u}> k/2\right\rbrace }\abs{T_k(u)}dx\right) 
	\leq \frac{k}{2} + \frac{k}{4} = \frac{3}{4}k 
	\]
	for all $k\geq k_0$. Then, if $p<N$, by Poincar\'e-Wirtinger's inequality, Sobolev inequality, and \eqref{estimate1}, we obtain 
	\[
	\norm{T_k(u) - c_{k}}_{L^q\left( \O'\right) }\leq C(N,p,\O_0,\mu,u)k^{\frac{1}{p}} ,
	\]
	where $q=\frac{Np}{N-p}$. Since for all $k\geq k_0$ we have the inclusions
	\[
	\left\lbrace\abs{u}\geq k\right\rbrace =\left\lbrace\abs{T_k(u)\geq k}\right\rbrace \subset \left\lbrace\abs{T_k(u) - c_{k}}\geq k-\abs{c_{k}}\right\rbrace  \subset \left\lbrace\abs{T_k(u) - c_{k}}\geq \frac{k}{4}\right\rbrace 
	\]
	we deduce
	\[
	\abs{\left\lbrace\abs{u}\geq k\right\rbrace \cap \O'}\leq \left( \frac{4\norm{T_k(u) - c_{k}}_{L^q\left( B_M\right) }}{k}\right)^q \leq C(N,p,\O_0,\mu,u)k^{q\left( \frac{1-p}{p}\right) }
	\]
	which is estimate \eqref{estimate2}. In the case $p=N$, the same procedure gives \eqref{estimate4}. The remaining estimates follow from the above ones just as in the proof of Theorem \ref{DMMOP99_level_set_estimates} in \cite{DMOP99}, using the results in \cite{BBGGPV95}. 
\end{proof}

Note that unlike the estimates in Theorem \ref{DMMOP99_level_set_estimates} the above estimates are \textit{not} uniform on $u$. However, they are enough for our purposes.

\begin{prop}\label{cap_finite}
	Let $u$ be a local renormalized solution of $-\lap_p u =\mu$ in $\R^N$. Then $u$ is finite $cap_{1,p,N}-q.e.$ in $\R^N$. In particular, if $v$ is a renormalized solution of \eqref{equation} in the sense of definition \ref{defn} then the trace of $v$, as defined in Remark \ref{remark_quasicontinuity}, is finite $cap_{1-\frac{1}{p},p,N-1}-q.e.$ in $\pd\R^N_+$.
\end{prop}
\begin{proof}
	As observed before, it is enough to show that $u$ is finite $cap_{1,p,N}-q.e.$ in $\R^N$. Fix $M \in \N$. By the previous theorem, with $\O=\R^N$ and $\O'=B_M(0)=:B_M$, we can find $k_0(u,M,p)$ such that for all $k\geq k_0$
	\[
	\abs{\left\lbrace\abs{u}\geq k\right\rbrace \cap B_M}\leq \frac{1}{4}\abs{B_M} .
	\]
	Then, we can proceed as in the proof of the previous theorem to obtain that  
	\[
	c_{2k,M}:=\frac{1}{\abs{B_M}}\int_{B_M}\abs{T_{2k}(u)}dx
	\]
	satisfies 
	\[
	c_{2k,M}\leq \frac{3}{2}k
	\]
	for any $k\geq k_0$. 
	Now consider the function $\phi=\frac{T_{2k}(u)-c_{2k,M}}{2k-c_{2k,M}}$. We have $\phi \in W^{1,p}\left( B_M\right) $ and by combining Poincar\'e-Wirtinger's inequality, estimate \eqref{estimate1}, and the above estimate we conclude
	\[
	\norm{\phi}_{W^{1,p}\left( B_M\right) }\leq \frac{C(p,N,M,u)}{\abs{2k - c_{2k,M}}}\norm{\nabla T_{2k}(u)}_{L^p\left( B_M\right) } \leq C(p,N,M,\mu,u)k^{\frac{1}{p}-1} .
	\]
	for any $k\geq k_0$. Further, we have $\phi= 1$ on the set $\left\lbrace u\geq 2k\right\rbrace \cap B_M$. Hence, by definition of $cap_{1,p,N}$ we obtain
	\[
	cap_{1,p,N}\left( \left\lbrace u\geq 2k\right\rbrace \cap B_M\right) \leq \norm{\phi}_{W^{1,p}\left( B_M\right) }^p\leq Ck^{1-p}
	\]
	for any $k\geq k_0$. Since $p>1$ we conclude that $cap_{1,p,N}\left( \left\lbrace u=+\infty\right\rbrace \cap B_M\right)=0$. In a similar way we can control the set where $u=-\infty$. Since $M\in \N$ is arbitrary, this concludes the proof.
\end{proof}

Note that to obtain the estimates of Theorem \ref{local_level_set_estimates} for a local renormalized solution $u$ in $\O$, it would have been enough to have $\abs{\nabla u}^{p-1}\in L^1_{loc}\left( \O\right)$ instead of condition $(2)$ of Definition \ref{definition_local_solution}. Similarly, instead of condition $(3)$ we only used $u\in L^s_{loc}\left( \O\right) $ for some $s>0$ as a step in obtaining the level set estimate $\abs{\left\lbrace\abs{u}>k\right\rbrace }\leq C k^{-s}$. As an interesting consequence of this, we have that conditions $(2)$ and $(3)$ in Definition \ref{definition_local_solution} could be weakened. We remark that this result has already been shown in Theorem 3.1 of \cite{B-V03}, although by a different method and with the stronger condition $\abs{u}^{q}\in L^{1}_{loc}\left( \O\right) $ for some $q>p-1$.

\begin{cor}
	Let $\O$ be any domain, $\mu\in \mathfrak{M}_b\left( \O\right)$, and let $u$ satisfy conditions $(1)$ and $(4)$ of Definition \ref{definition_local_solution}. If $u$ also satisfies
	\begin{enumerate}
		\item [(2')] $\abs{\nabla u}^{p-1}\in L^1_{loc}\left( \O\right) $,
		\item [(3')] for any $\O_0\subset\subset \O$ there exists $C>0$ and $\a>0$ such that 
		\[
		\abs{\left\lbrace\abs{u}>k\right\rbrace \cap \O_0}\leq C k^{-\a} ,
		\]
	\end{enumerate}
	then $u$ is a local renormalized solution of $-\lap_p u=\mu$ in $\O$.
\end{cor}
\begin{proof}
	Since we have the estimates of Theorem \ref{local_level_set_estimates}, we can show $(2)$ and $(3)$ of Definition \ref{definition_local_solution} following the ideas in Remark \ref{remark_on_L^p_estimate}. Indeed, thanks to \eqref{eq_estimate_L^p_from_Level_set}, we can write
	\[
	\int_{\O_0} \abs{u}^s dx\leq k_0^s \abs{\O_0} + s\int_{k_0}^\infty k^{s-1}\abs{\left\lbrace\abs{u}\geq k\right\rbrace }dk
	\]
	which is finite when $\O_0\subset\subset\O$ and $1< s\leq\frac{N}{N-p}$ ($s<\infty$ if $p=N$). Hence, we have $(3)$. Similarly, the estimates on $\nabla u$ show that $(2)$ holds.  
\end{proof}

We now show that renormalized solutions of \eqref{equation} in fact exists.

\begin{prop}
	Let $1<p\leq N$, and let 
	\[
	u=\begin{cases}
	\frac{p-1}{N-p}\left( \frac{2}{\sigma_N}\right)^\frac{1}{p-1}\abs{x}^{\frac{p-N}{p-1}} & \mbox{ if } p<N \\
		\left( \frac{2}{\sigma_N}\right)^\frac{1}{N-1}\ln(\abs{x}) & \mbox{ if } p=N
	\end{cases}
	\]
	where 
	$\sigma_N$ is the surface area of $\pd B_1$. Then $u$ is a renormalized solution to
	\[
	\begin{cases}
	-\lap_p u =0 & \mbox{ in }\R^N_+\\
	-\abs{\nabla u}^{p-2}u_\n = \d_0 & \mbox{ on } \pd\R^N_+ .
	\end{cases}
	\]
\end{prop}
\begin{proof}
	Let us first observe that $\d_0$ is positive and singular with respect to $cap_{1,p,N}$ since $cap_{1,p,N}\left( \left\lbrace0\right\rbrace \right) =0$. This can be proven, for example, by using the known relationships between capacity and Hausdorff measure (see \cite{AdamsHedberg}).
	
	We assume $p<N$ since the case $p=N$ is almost identical. We note that $u$ is finite $a.e.$, measurable, and clearly satisfies $T_k(u) \in W^{1,p}_{loc}\left( \R^N_+\right) $ and so the first requirement holds. For the second one observe that $\nabla u = -\left( \frac{2}{\sigma_N}\right) ^{\frac{1}{p-1}}\abs{x}^{\frac{2-p-N}{p-1}}x$ and so $\abs{\nabla u}^{q(p-1)}=\left( \frac{2}{\sigma_N}\right) ^q\abs{x}^{q(1-N)}$. If $q(N-1)<N$ then $-q(1-N)<N$ and so the singularity is integrable at the origin and $\abs{\nabla u}^{p-1}\in L^q_{loc}\left( \R^N_+\right) $. The third requirement is immediate. 
	
	Suppose now that $w\in W^{1,p}\left( \R^N_+\right) $ has compact support in $\overline{\R^N_+}$ and trace in $L^\infty\left( \pd\R^N_+\right) $. Let $k>0$ and suppose $w=w^{+\infty}$ $a.e.$ in the set $\left\lbrace x\in \R^N_+:u(x)>k\right\rbrace $ with $w^{+\infty}\in W^{1,r}\left( \R^N_+\right) $ and $r>N$. Note that since $r>N$ we have that $w^{+\infty}$ is continuous in $\overline{\R^N_+}$. As in the considerations following definition \ref{defn}, we see that $\abs{\nabla u}^{p-1}\abs{\nabla w}$ belongs to $L^1\left( \R^N_+\right) $. Hence, we can apply Lebesgue's Dominated Convergence Theorem to obtain
	\[
	\int_{\R^N_+}\abs{\nabla u}^{p-2}\nabla u\cdot \nabla wdx = \lim_{\e\downarrow0}\int_{\R^N_+}\abs{\nabla u}^{p-2}\nabla u(x',x_N+\e)\cdot \nabla w(x) dx \ .
	\]
	Since $\nabla u(x',x_N+\e)$ is smooth for every $\e>0$, vanishes as $\abs{x}\goesto \infty$, and $\lap_pu=0$ in $\R^N_+$ we obtain
	\begin{multline*}
	\int_{\R^N_+}\abs{\nabla u}^{p-2}\nabla u (x',x_N+\e)\cdot\nabla w(x)dx = \int_{\pd\R^N_+}\abs{\nabla u}^{p-2}u_\n(x',\e)w(x',0)dx' \\
	= -\frac{2\e}{\sigma_N}\int_{\pd\R^N_+}\frac{w(x',0)}{\abs{(x',0)-(0,\e)}^N}dx' .
	\end{multline*}
	Finally, it is well-known (cf. \cite{Evans}) that this last integral satisfies
	\[
	\lim_{\e\downarrow0} -\frac{2\e}{\sigma_N}\int_{\pd\R^N_+}\frac{w(x',0)}{\abs{(x',0)-(0,\e)}^N}dx'=-w^{+\infty}(0)
	\]
	since $w$ has bounded trace and is continuous in a neighborhood of the origin because $w=w^{+\infty}$ $cap_{1,p,N}-q.e.$ near the origin (see Remark \ref{remark_test_trace}).
\end{proof}

\begin{remark}
	The ideas above can be used to define renormalized solutions to Neumann problems in bounded domains. We do so now.
	
	Let $\O$ be a bounded extension domain, i.e., a domain such that there exists a linear bounded extension operator from $W^{1,p}\left( \O\right) $ into $W^{1,p}\left( \R^N\right)$. Assume $1<p\leq N$ and $\mu \in \mathfrak{M}_b\left( \R^N\right) $ is supported in $\pd\O$. Let $\mu=\mu_0 + \mu_s^+-\mu_s^-$ be the decomposition of $\mu$ with respect to $cap_{1,p,N}$. Then, a renormalized solution of 
	\[
	\begin{cases}
	-\lap_pu=0 & \mbox{ in } \O \\
	\abs{\nabla u}^{p-2}u_\n = \mu & \mbox{ on } \pd\O
	\end{cases}
	\]
	is a function $u$ defined in $\O$ such that
	
	\begin{enumerate}
		\item $u$ is measurable, finite $a.e.$, and $T_k(u)\in W^{1,p}\left( \O\right) $ for all $k>0$;
		\item $\abs{\nabla u}^{p-1}\in L^q\left( \O\right) $ for all $1\leq q<\frac{N}{N-1}$;
		\item there holds
	\end{enumerate}
\[
\int_{\O}\abs{\nabla u}^{p-2}\nabla u\cdot\nabla wdx  = \int_{\pd\O} w d\mu_0 + \int_{\pd\O}w^{+\infty}d \mu_s^+ - \int_{\pd\O}w^{-\infty}d\mu_s^-
\]
\begin{enumerate}[label={}]
\item for all $w\in W^{1,p}\left( \O\right) $ with trace in $L^\infty\left(\pd\O;d\mu_0\right) $, and satisfying the following condition: there exist $k>0$, $r>N$, and functions $w^{\pm\infty}\in W^{1,r}\left( \O\right)$ such that
\[
\begin{cases}
w=w^{+\infty} & a.e. \mbox{ in } \left\lbrace x\in \O \ : \ u>k\right\rbrace \\
w=w^{-\infty} & a.e. \mbox{ in } \left\lbrace x\in \O \ : \ u<-k\right\rbrace  .
\end{cases} 
\]
\end{enumerate}
	 Note that under the above conditions test functions have well defined traces on $\pd\O$. Indeed, by using that $\O$ is an extension domain, we can proceed as in Proposition \ref{qcTrace} to show that $w$ has a $cap_{1,p,N}-$ quasi-continuous representative which is unique $\mu_0-a.e.$. Hence, we let the trace of $w$ be the restriction to $\pd\O$ of this $cap_{1,p,N}-$ quasi-continuous representative. Similarly, $w^{\pm\infty}$ can be extended, uniquely, as continuous and bounded functions in $\overline{\O}$. 
	 
	 It can be shown, just as in the case of definition \ref{defn}, that all the integrals above are well defined and finite. Note that we have assumed that the trace of $w$ belongs to $L^\infty\left(\pd\O;d\mu_0\right) $. This has to be contrasted with definition \ref{defn} where, thanks to Proposition \ref{L^1mu_0}, we only assumed that the trace is in $L^\infty\left( \pd\R^N_+\right) $.
\end{remark}
	
\chapter{Local renormalized Solutions in $\R^N$}\label{Extension}

We now prove some preliminary results that will help us to obtain a renormalized solution to \eqref{equation} in the sense of definition \ref{defn}. We will mostly use ideas developed in \cite{DMOP99} for the case $g\equiv 0$. Note however that the theory developed there only applies to bounded domains and so it cannot be applied directly to our case. We circumvent this problem by working locally, that is, we first obtain a sequence of solutions on balls $B_m$ of increasing radii and then we consider the behavior of these solutions on any fixed ball $B_M$. 

As a corollary, we will prove the following theorem on the existence of local renormalized solutions in $\R^N$. 

\begin{thm}\label{existenceExtended}
	Let $\bar{\mu} \in \mathfrak{M}_b\left( \R^N\right)$ and $1<p\leq N$. Then there exists a local renormalized solution to 
	\[-\lap_pu=\bar{\mu} \mbox{ in } \R^N \ .\]
\end{thm}

\section{Preliminary convergence result}

Consider the following restrictions of a measure $\bar{\mu}\in \mathfrak{M}_b\left( \R^N\right) $
\[
\bar{\mu}_m(A):=\bar{\mu}(A\cap B_m)
\]
where $B_m$ is the ball centered at the origin of radius $m$. It is easy to see that 
\[
\left( \bar{\mu}_m\right) _0=\left( \bar{\mu}_0\right) _m \ , \ \left( \bar{\mu}_m\right)_s^\pm=\left( \bar{\mu}_s^\pm\right) _m \ . 
\]

For each $m\in \N$ we can use the results in \cite{DMOP99} to obtain a renormalized solution $u_m$ to the problem
\be\label{eqn u_m}
\begin{cases}
	-\lap_p u_m = \bar{\mu}_m & \mbox{ in } B_m \\
	u_m = 0 & \mbox{ on } \pd B_m .
\end{cases}
\ee
Here and in the sequel we identify the functions $u_m$ as functions defined on the whole space extending them by zero outside of $B_m$. Note that since $T_k(u_m)\in W^{1,p}_0\left( B_m\right) $ the extension satisfies $T_k(u_m)\in W^{1,p}\left( \R^N\right) $. Hence, by Remarks \ref{remark_quasicontinuity} and \ref{remark_quasicontinuity_DMOP99}, the extension of $u_m$ has a $cap_{1,p,N}-$ quasi-continuous representative in $\R^N$. Clearly, up to a set of zero capacity, this representative is the extension by zero of the $cap_{1,p,N}-$ quasi-continuous representative of $u_m$ given by Remark \ref{remark_quasicontinuity_DMOP99}.

In the following lemma we show that we can extract a point-wise convergent subsequence from $\left\lbrace u_m\right\rbrace _m$. The argument follows closely the ideas used in Section 5 of \cite{DMOP99}. 

\begin{lem}\label{general_existence_of_limit_functions}
	Let $1<p\leq N$. Let $\nu_m \in \mathfrak{M}_b\left( \R^N\right)$ be a sequence of measures such that $\abs{\nu_m}\left( B_m\right)\leq C_1<\infty$ for all $m\in\N$. Let $u_m$ be renormalized solutions to \eqref{eqn u_m} with data $\nu_m$, i.e., 
	\be\nonumber
	\begin{cases}
		-\lap_p u_m = \nu_m & \mbox{ in } B_m \\
		u_m = 0 & \mbox{ on } \pd B_m .
	\end{cases} 
	\ee
	Then there exists a function $u$ such that, up to a subsequence, $u_m\goesto u$ $a.e.$ in $\R^N$. Moreover :
	\begin{enumerate}
		\item $u$ is measurable and finite $cap_{1,p,N}-q.e.$, $T_k(u_m)\goesto T_k(u)$ weakly in $W^{1,p}\left( B_M\right)$ for any fixed $k>0$ and $M\in \N$, and $\nabla T_k(u_m)\goesto \nabla T_k(u)$ $a.e.$ in $\R^N$ for any $k>0$;
		\item $\nabla u_m \goesto \nabla u$ $a.e.$ and $\abs{\nabla u_m}^{p-2}\nabla u_m\goesto \abs{\nabla u}^{p-2}\nabla u$ strongly in $\left( L^q\left( B_M\right) \right) ^N$ for any $M\in \N$ and $1\leq q<\frac{N}{N-1}$;
		\item $\abs{u}^{p-1}\in L^q_{loc}\left( \R^N\right) $ for all $1<q<\frac{N}{N-p}$ ($1<q<\infty$ if $p=N$).
	\end{enumerate}
\end{lem}
\begin{proof}
	To begin we note that each $u_m$ satisfies the estimates stated in Theorem \ref{DMMOP99_level_set_estimates} uniformly in the sense that they hold with $\abs{\nu_m}\left( B_m\right)$ replaced by $C_1$. Now fix any $M\in \N$, $k\in \N$, and $\sigma>0$. Observe that $\left\lbrace x\in B_M : \abs{u_m-u_n}>\sigma \right\rbrace $ is contained in
	\begin{multline}\label{eq_containments}
	\left\lbrace x\in B_M : \abs{u_m}>k\right\rbrace \cup\left\lbrace x\in B_M : \abs{u_n}>k\right\rbrace \cup \\ \left\lbrace x\in B_M : \abs{T_k(u_m)-T_k(u_n)}>\sigma \right\rbrace  .
	\end{multline}
	Thanks to \eqref{estimate2_DMMOP99} and \eqref{estimate4_DMMOP99} the measure of the first two sets is arbitrarily small, independent of $m$ and $n$, provided $k$ is large enough. 
	
	Since for each fixed $k$ estimate \eqref{estimate1_DMMOP99} gives an uniform bound for $\norm{\nabla T_k(u_m)}_{L^p}$ we conclude that the sequence $\left\lbrace T_k(u_m)\right\rbrace _m$ is uniformly bounded in $W^{1,p}\left( B_M\right) $ for any fixed $k$ and $M$. Since the injection $W^{1,p}\left( B_M\right) \hookrightarrow L^p\left( B_M\right) $ is compact, this means that $\left\lbrace T_k(u_m)\right\rbrace _m$ has a subsequence that converges strongly in $L^p\left( B_M\right) $, and hence, that it is a Cauchy subsequence in measure in $B_M$. 
	
	Now take $k=1$ and apply the above argument in $B_M$ to obtain a subsequence $\left\lbrace u_{m,1}\right\rbrace _m\subset \left\lbrace u_m\right\rbrace _m$ such that $\left\lbrace T_1(u_{m,1})\right\rbrace _m$ is a Cauchy sequence in measure in $B_M$. Since $\left\lbrace u_{m,1}\right\rbrace _m$ has the same properties as $\left\lbrace u_m\right\rbrace _m$, we fix $k=2$ and apply again the argument above to obtain a subsequence $\left\lbrace u_{m,2}\right\rbrace _m\subset \left\lbrace u_{m,1}\right\rbrace _m$ such that $\left\lbrace T_2(u_{m,2})\right\rbrace _m$ is a Cauchy sequence in measure in $B_M$. Proceeding inductively, we see that we can define a diagonal sequence $\left\lbrace u_{m,m}\right\rbrace _m$. Going back to \eqref{eq_containments}, it easy to see that this sequence, which we relabel as $\left\lbrace u_m\right\rbrace _m$, is a Cauchy sequence in measure. Hence, passing to a subsequence, there exists a measurable and $a.e.$ finite function $v_M$ such that $u_m\goesto v_M$ $a.e.$ in $B_M$. Proceeding in a similar way, but now with respect to $M\in\N$, we can obtain a subsequence $\left\lbrace u_{m,m}\right\rbrace _m\subset \left\lbrace u_m\right\rbrace _m$, such that for every $M\in \N$ $u_m\goesto v_M$ $a.e.$ in $B_M$. Relabeling this subsequence as $\left\lbrace u_m\right\rbrace _m$, we see that there exists a measurable and $a.e.$ finite function $u$ such that
	\[
	u_m\goesto u \mbox{ $a.e.$ in } \R^N 
	\]
	satisfying $u=v_M$ $a.e.$ in $B_M$.
	
	We now consider the properties of the limit function. 
	Note that since $T_k(s)$ is continuous we have $T_k(u_m)\goesto T_k(u)$ $a.e.$ in $B_M$. 
	Estimate \eqref{estimate1_DMMOP99} implies that $\left\lbrace T_k(u_m)\right\rbrace _m$ is uniformly bounded in $W^{1,p}\left( B_M\right) $ for any fixed $k>0$. Thus, for any subsequence $\left\lbrace T_k(u_{m_j})\right\rbrace _j$ a further subsequence converges weakly in $W^{1,p}\left( B_M\right) $ to a limit function $v_k$. But $T_k(u_m)\goesto T_k(u)$ $a.e.$ in $B_M$, which implies (by the boundedness of the sequence) that $v_k=T_k(u)$. Therefore 
	\[
	T_k(u_m)\goesto T_k(u) \mbox{ weakly in } W^{1,p}\left( B_M\right)  \mbox{ for any fixed } k>0.
	\]
	In particular
	\[
	T_k(u) \in W^{1,p}\left( B_M\right) ,
	\]
	and thus for any $k>0$
	\[
	T_k(u)\in W^{1,p}_{loc}\left( \R^N\right)   .
	\]
	Let us make explicit that this allows us to define $\nabla u$ in the generalized sense described earlier. 
	Also, using \eqref{estimate1_DMMOP99} and Fatou's Lemma we further conclude that
	\be\label{estimate1_global_solution}
	\frac{1}{k}\int_{\left\lbrace n\leq\abs{u}<n+k\right\rbrace \cap B_M}\abs{\nabla T_{k+n}(u)}^pdx\leq C_1 .
	\ee
	
	Now we want to show that for any fixed $k>0$ and $M\in \N$, $\{\nabla T_k(u_m)\}_m$ is a Cauchy sequence in measure in $B_M$. For this we follow the approach in the proof of Theorem 4.3.8 in \cite{Veron16}. Fix any $M\in \N$, $k>0$, and $\eta,\sigma>0$, and let $m,n\geq M+1$. Choose any $\phi\in C_0^\infty\left( B_{M+1}\right) $ such that $\phi=1$ in $B_M$ and $0\leq \phi\leq 1$. For $\d>0$ we define
	\[
	D_\d :=\left\lbrace \abs{T_k(u_m)-T_k(u_n)}>\d\right\rbrace \ , \ E_\d:=D_\d^c \cap \left\lbrace \abs{\nabla T_k(u_m) -\nabla T_k(u_n)}>\sigma\right\rbrace 
	\]
	and observe that
	\[
	\left\lbrace \abs{\nabla T_k(u_m) -\nabla T_k(u_n)}>\sigma\right\rbrace \subset D_\d \cup E_\d.
	\]
	Let $w=\phi T_\d\left( T_k(u_m)-T_k(u_n)\right) $ and test against $w$ in the equations solved by the truncates $T_k(u_m)$ and $T_k(u_n)$ (see Remark \ref{remark_on_equivalent_definitions_DMOP99}) to find that
	\begin{multline*}
	\abs{\int_{B_{M+1}}\left( \abs{\nabla T_k(u_m)}^{p-2}\nabla T_k(u_m)-\abs{\nabla T_k(u_n)}^{p-2}\nabla T_k(u_n)\right) \cdot \nabla w dx}\leq \\
	 \d \left( 2\abs{\mu_0}\left( B_{M+1}\right) + \l_{k,m}^+\left( B_{M+1}\right) +\l_{k,m}^-\left( B_{M+1}\right) +\l_{k,n}^+\left( B_{M+1}\right) +\l_{k,n}^-\left( B_{M+1}\right) \right) 
	\end{multline*}
	for some measures $\l_{k,m}^\pm$ and $\l_{k,n}^\pm$ converging in the narrow topology of measures to $\left( \mu_s^\pm\right) _m$ and $\left( \mu_s^\pm\right) _n$, respectively, as $k\goesto \infty$. By testing against $T_k(u_m)$ in the equation solved by $T_k(u_m)$, and using estimate \eqref{estimate1_DMMOP99}, we obtain that $\l_{k,m}^\pm$ are bounded independently of $m$. Hence, the right hand side in the above inequality is bounded by $\d c_1$ where $c_1=c_1(k,C_1)$ is independent of $m$ and $n$. On the other hand, 
	\begin{multline*}
	\abs{\int_{B_{M+1}}T_\d\left( T_k(u_m)-T_k(u_n)\right)\abs{\nabla T_k(u_m)}^{p-2}\nabla T_k(u_m)\cdot \nabla \phi dx}\leq \\ \d \norm{\nabla T_k(u_m)}^{p-1}_{L^p\left( B_{M+1}\right) }\norm{\nabla \phi}_{L^p\left(B_{M+1}\right) }\leq \d c_2
	\end{multline*}
	where, again by \eqref{estimate1_DMMOP99}, $c_2=c_2(k,p,\phi,C_1)$ is independent of $m$. Then, using the structural inequality \eqref{coercitivity}, we can proceed as in the proof of Theorem \ref{symmetry} to show that
	\begin{multline*}
	\int_{B_{M}\cap D_\d^c}\abs{\nabla T_k(u_m)-\nabla T_k(u_n)}^pdx\leq \\ c_3\int_{B_{M}\cap D_\d^c}\!\!\!\!\!\!\left( \abs{\nabla T_k(u_m)}^{p-2}\nabla T_k(u_m)-\abs{\nabla T_k(u_n)}^{p-2}\nabla T_k(u_n)\right) \cdot \left( \nabla T_k(u_m)-\nabla T_k(u_n)\right) dx
	\end{multline*}
	where $c_3=c_3(k,p,C_1)$ is independent of $m$ and $n$. Hence, by combining all the above estimates we see that
	\[
	\abs{E_\d\cap B_M}\leq \frac{1}{\sigma^p}\int_{B_{M}\cap D_\d^c}\abs{\nabla T_k(u_m)-\nabla T_k(u_n)}^pdx\leq \d\left( \frac{c_3(c_1+c_2)}{\sigma^p}\right)  ,
	\]
	and so we can choose $\d>0$ independent of $m$ and $n$ such that $\abs{E_\d\cap B_M}<\eta$. Since $\{T_k(u_m)\}_m$ is a Cauchy sequence in measure in $B_M$, once $\d$ is fixed we obtain that $\abs{D_\d\cap B_M}<\eta$ if $m$ and $n$ are large enough. Hence, the desired result follows. Note that we also obtain that there exists a subsequence such that $\nabla T_k(u_{m_j})\goesto v_k$ $a.e.$ in $B_M$. Since $\nabla T_k(u_m)$ is uniformly bounded in $\left( L^p\left( B_M\right) \right)^N$ we conclude that in fact $v_k=\nabla T_k(u)$.

	Now, noticing that $\left\lbrace x\in B_M : \abs{\nabla u_m-\nabla u_n}>\sigma \right\rbrace $ is contained in
	\[
	\left\lbrace x\in B_M : \abs{u_m}>k\right\rbrace \cup\left\lbrace x\in B_M : \abs{u_n}>k\right\rbrace \cup \left\lbrace x\in B_M : \abs{\nabla T_k(u_m)-\nabla T_k(u_n)}>\sigma \right\rbrace  
	\]
	we proceed as before to obtain that, passing to a subsequence, $\nabla u_m$ converges $a.e.$ to a function $v$ in $B_M$. Note that for fixed $k>0$ we can choose a subsequence to obtain 
	\begin{multline*}
	v\chi_{\{\abs{u}<k\}}=\lim_{j\goesto\infty} \nabla u_{m_j} \chi_{\{\abs{u}<k\}} =\lim_{j\goesto\infty} \nabla u_{m_j} \chi_{\{\abs{u}<k\}}\chi_{\left\lbrace  \abs{u_{m_j}}<k\right\rbrace} \\= \lim_{j\goesto\infty} \nabla T_k(u_{m_j}) \chi_{\{\abs{u}<k\}} = \nabla T_k(u)
	\end{multline*}
	$a.e.$ in $B_M$. Therefore
	\[
	\nabla u_m \goesto \nabla u \mbox{ $a.e.$ in } \R^N . 
	\]
	It then follows that in fact, for any $k>0$, $\nabla T_k(u_m) \goesto \nabla T_k(u)$ $a.e.$ in $\R^N$.
	Moreover, the identity \eqref{eq_estimate_L^p_from_Level_set} and the uniform decay estimates of Theorem \ref{DMMOP99_level_set_estimates} imply that the family $\abs{\nabla u_m}^{p-2}\nabla u_m$ is uniformly integrable over $B_M$ (see also Step 1 of Section 5 of \cite{DMOP99}). Hence, by Vitali's Theorem it follows that
	\[
	\abs{\nabla u_m}^{p-2}\nabla u_m\goesto \abs{\nabla u}^{p-2}\nabla u \mbox{ in } \left( L^q\left( B_M\right)\right)^N \mbox{ for all }  1\leq q< \frac{N}{N-1} .
	\]
	In particular
	\[
	\abs{\nabla u}^{p-1} \in L^q_{loc}\left( \R^N\right) \mbox{ for all } 1\leq q <\frac{N}{N-1} . 
	\]
	In the same spirit one can show that 
	\[
	\abs{u}^{p-1} \in L^q_{loc}\left( \R^N\right) \mbox{ for all }  1 < q <\frac{N}{N-p}  
	\]
	when $p<N$, whereas
	\[
	\abs{u}^{p-1} \in L^q_{loc}\left( \R^N\right) \mbox{ for all }  1< q <\infty  
	\]
	when $p=N$ (see Remark \ref{remark_on_L^p_estimate}).
	
	To finish the proof we show that $u$ is finite $cap_{1,p,N}-q.e.$ in $B_M$ for all $M\in \N$, and thus in the whole $\R^N$. Fix $M \in \N$. By estimate \eqref{estimate2_DMMOP99} and \eqref{estimate4_DMMOP99} we can choose $k_0>0$ such that for all $k\geq k_0$ and for all $m\in \N$
	\[
	\abs{\left\lbrace\abs{u_m}\geq k\right\rbrace }\leq \frac{1}{4}\abs{B_M} .
	\]
	Thus, we can estimate
	\begin{multline*}
	\int_{B_M}\abs{T_{2k}(u_m)}dx = \int_{B_M\cap\left\lbrace\abs{u_m}<k\right\rbrace } \abs{T_{2k}(u_m)}dx + \int_{B_M\cap\left\lbrace\abs{u_m}\geq k\right\rbrace }\abs{T_{2k}(u_m)}dx \\
	\leq k\abs{B_M} + 2k\frac{1}{4}\abs{B_M}=\frac{3}{2}k\abs{B_M}
	\end{multline*}
	for any $k\geq k_0$. Let us define the following averages:
	\[
	c_{k,m,M}:=\frac{1}{\abs{B_M}}\int_{B_M}T_k(u_m) dx \ , \ c_{k,M}:=\frac{1}{\abs{B_M}}\int_{B_M}T_k(u) .
	\] 
	Note that by Lebesgue's Dominated Convergence Theorem we have
	\[
	c_{k,M}=\lim_{m\goesto\infty}c_{k,m,M}
	\]
	and by the above estimate we get
	\[
	\abs{c_{2k,M}}\leq \frac{3}{2}k
	\]
	for any $k\geq k_0$. Now, to finish, we can proceed as in the proof of Proposition \ref{cap_finite}, by considering the function $\phi=\frac{T_{2k}(u)-c_{2k,M}}{2k-c_{2k,M}}$.
\end{proof}

\section{Stability}

We now consider the problem of showing that the limit function $u$ defined in the previous lemma is a local renormalized solution of the desired equation. Since we will deal with nonlinear terms later, it will be useful to prove a more general result. Let us recall that if $\nu \in L^1\left(B_m\cap\pd\R^N_+\right) $ then, by Proposition \ref{trace&extensionE}, $\nu\mathcal{H} \in \mathfrak{M}_0\left(B_m\right) $ (see also Proposition \ref{decomposition}). We also remark that if a function $u$ satisfies $(1)$ in Lemma \ref{general_existence_of_limit_functions} then $u$ has a $cap_{1,p,N}-$ quasi-continuous representative, which we identify with $u$ (see Remark \ref{remark_quasicontinuity}). 


\begin{lem}\label{general_existence_global_solution}
	Let $\bar{\mu} \in \mathfrak{M}_b\left( \R^N\right)$ and assume $g_m$ and $g$ are measurable functions defined in $\pd\R^N_+$ such that $\norm{g_m}_{L^1\left( B_m\cap \pd\R^N_+ \right) } + \norm{g}_{L^1\left( \pd\R^N_+\right) } \leq C_1 <\infty$ for some positive constant $C_1$. Let $u_m$ be renormalized solutions to 
	\be\nonumber
	\begin{cases}
		-\lap_p u_m = \bar{\mu}_m - g_m\mathcal{H} & \mbox{ in } B_m \\
		u_m = 0 & \mbox{ on } \pd B_m
	\end{cases} 
	\ee
	where $\bar{\mu}_m$ is the restriction of $\bar{\mu}$ to $B_m$. Assume $u_m\goesto u$ $a.e.$ in $\R^N$, where $u$ is a function satisfying properties (1), (2), and (3) in Lemma \ref{general_existence_of_limit_functions}. Suppose also that  
	\be\label{condition_general_existence_global_solution}
	\lim_{m\goesto\infty} \int_{B_{M}\cap\pd\R^N_+}\phi_m g_m dx'= \int_{B_{M}\cap\pd\R^N_+} \phi g dx'
	\ee
	for any $M\in \N$ and any sequence $\left\lbrace\phi_m\right\rbrace _m$ converging to $\phi$ both $a.e.$ in $B_M$ and weakly in $W_0^{1,p}\left( B_M\right) $ and such that $\phi_m$ is uniformly bounded in $L^\infty\left( B_M\right) $. Then $u$ is a local renormalized solution of 
	\[-\lap_p u = \bar{\mu} -g\mathcal{H} \mbox{ in } \R^N \ .\]
	Moreover, $T_k(u_m)\goesto T_k(u)$ strongly in $W^{1,p}\left( B_M\right) $ for any fixed $k>0$ and $M\in \N$.
\end{lem}
\begin{proof}
	Since properties $(1)$, $(2)$, and $(3)$ of Lemma \ref{general_existence_of_limit_functions} hold, we have that $u$ solves the desired equation if we can prove the last property listed in Definition \ref{definition_local_solution}. We show this first, following the approach of \cite{M05}.

	First we note that by Theorem \ref{BGO96_thm} we have $\left( \bar{\mu}_m\right) _0=\left( \bar{\mu}_0\right) _m= f_m -\dv h_m$ in $\D'\left( B_m\right) $ for some $f_m\in L^1\left( B_m\right) $ and $h_m\in \left( L^{p'}\left( B_m\right) \right) ^N$. Note that this representation is also valid in $\D'\left( B_{m'}\right) $ for any $m'<m$ and so $\left. \left(f_m -\dv h_m\right) \right| _{B_{m'}}=\left( f_{m'}-\dv h_{m'}\right) $. Then, by Lemma 3.1 of \cite{M05} there exists a set $U\subset (0,\infty)$ with $U^c$ of zero measure such that each $u_m$ satisfies the following condition: for every $k\in U$ there exists two measures $\a_{m,k}^+$, $\a_{m,k}^- \in \mathfrak{M}_0\left(B_m\right) $ supported in $\left\lbrace u_m=k\right\rbrace $ and $\left\lbrace u_m=-k\right\rbrace $ respectively, such that up to a subsequence (possibly depending on $m$) $\a_{m,k}^\pm \goesto \left( \bar{\mu}_m\right) _s^\pm $, as $k \in U$ goes to infinity, in the weak-$\ast$ topology of $\mathfrak{M}_b\left( B_m\right)$, and the truncations $T_k(u_m)$ satisfy
\begin{multline}
 \int_{\left\lbrace\abs{u_m}<k\right\rbrace }\left( \abs{\nabla T_k(u_m)}^{p-2}\nabla T_k(u_m) - h_m\right) \cdot \nabla v dx= \\
 		\int_{\left\lbrace u_m=k\right\rbrace } v d\a_{m,k}^+ - \int_{\left\lbrace u_m=-k\right\rbrace } v d\a_{m,k}^- + \int_{\left\lbrace\abs{u_m}<k\right\rbrace }vf_mdx - \int_{\left\lbrace\abs{u_m}<k\right\rbrace \cap\pd\R^N_+}vg_mdx'
\end{multline}
for every $v\in W^{1,p}_0\left( B_m\right) \cap L^\infty\left( B_m\right) $. 

Let us consider the convergence, in $m$, of the above terms. Given $M\in \N$ let $E_M= \left\lbrace k\in \R_+ : \abs{\left\lbrace x\in B_M : \abs{u}=k\right\rbrace }>0\right\rbrace  $ and write $F_M=\left( E_M\right) ^c$. Since $\abs{B_M}<\infty$, $E_M$ is countable and thus of zero measure. Note that $\chi_{\left\lbrace\abs{u_m}<k\right\rbrace }\goesto  \chi_{\left\lbrace\abs{u}<k\right\rbrace }$ $a.e.$ in $B_M$ except possibly in $\left\lbrace x\in B_M : \abs{u}=k\right\rbrace $, thus $\chi_{\left\lbrace\abs{u_m}<k\right\rbrace }\goesto \chi_{\left\lbrace\abs{u}<k\right\rbrace }$ $a.e.$ in $B_M$ and weakly-$\ast$ in $L^\infty\left( B_M\right) $ for all $k \in F_M$. 

By hypothesis, we have that $\abs{\nabla u_m}^{p-2}\nabla u_m\goesto \abs{\nabla u}^{p-2}\nabla u$ strongly in $\left( L^1\left( B_M\right) \right) ^N$ for any $M\in \N$. It follows that 
\begin{multline*}
\int_{\left\lbrace\abs{u_m}<k\right\rbrace } \abs{\nabla T_k(u_m)}^{p-2}\nabla T_k(u_m) \cdot \nabla \phi dx \goesto \int_{\left\lbrace\abs{u}<k\right\rbrace }\abs{\nabla T_k(u)}^{p-2}\nabla T_k(u) \cdot \nabla\phi dx
\end{multline*}
for any $\phi\in C_0^\infty\left( B_M\right) $ and $k \in F_M$. Similarly, for any such $\phi$ and $k$ there holds
\[
\int_{\left\lbrace\abs{u_m}<k\right\rbrace } h_m \cdot \nabla \phi dx + \int_{\left\lbrace\abs{u_m}<k\right\rbrace }\phi f_m dx \goesto \int_{\left\lbrace\abs{u}<k\right\rbrace } h_M \cdot \nabla \phi dx + \int_{\left\lbrace\abs{u}<k\right\rbrace }\phi f_M dx.
\]
Note that since $g_m$ are uniformly bounded in $L^1\left( B_M\cap\pd\R^N_+\right)$ so are the functions $g_m\chi_{\left\lbrace\abs{u_m}<k\right\rbrace }$. Then, up to a subsequence depending on $k$, there exist a measure $\tau_k \in \mathfrak{M}_b\left( \R^N\right)$ such that 
\[
\int_{\left\lbrace\abs{u_m}<k\right\rbrace }\phi g_m dx' \goesto \int_{B_M}\phi d\tau_k
\]
for any $\phi\in C_0^\infty\left( B_M\right) $.

Now we turn our attention to the measures $\a_{m,k}^\pm$. Just as in the proof of Theorem 4.1 in \cite{M05} we can use the fact that $\bar{\mu}_m$ and $g_m$ and are uniformly bounded as measures to conclude that for every $m$
\[
\abs{\a_{m,k}^+}\left( B_m\right) + \abs{\a_{m,k}^-}\left( B_m\right) \leq C(\bar{\mu},C_1) 
\]
for any $k>0$ in some subset $V$ with $\abs{V^c}=0$, and where $C$ is independent of $k$ or $m$. Hence, for each $k\in V$, there exists nonnegative measures $\l_k^+$ and $\l_k^-$ defined in $\R^N$ such that, up to a subsequence, 
\[
\a_{m,k}^\pm\goesto \l_k^\pm \mbox{ weakly-$\ast$ in } \mathfrak{M}_b\left( \R^N\right) .
\]
In particular, given $\phi\in C_0^\infty\left( B_M\right) $ we can pass to a subsequence to conclude
\[
\int_{B_M}\phi d\a_{m,k}^+ - \int_{B_M}\phi d \a_{m,k}^- \goesto \int_{B_M}\phi d\l_k^+ - \int_{B_M}\phi d\l_k^-
\]
which implies, by the previous considerations, that for any $\phi\in C_0^\infty\left( B_M\right) $ and $k\in K_M=F_M\cap U \cap V$
\begin{multline*}
\int_{\left\lbrace\abs{u}<k\right\rbrace }\left( \abs{\nabla T_k(u)}^{p-2}\nabla T_k(u)-h_M\right) \cdot \nabla \phi dx = \\ \int_{\left\lbrace\abs{u}<k\right\rbrace }\phi f_Mdx - \int_{B_{M}}\phi d\tau_k  + \int_{B_M}\phi d \l_k^+ - \int_{B_M}\phi d\l_k^- .
\end{multline*}
Note that $\abs{\nabla T_k(u)}^{p-1} + \abs{h_M} \in L^{p'}\left( B_M\right) $ while $f_M \in \mathfrak{M}_0\left( B_M\right)\cap L^1\left( B_M\right)  $ and thus, by Theorem \ref{BGO96_thm}, $\left. -\tau_k +\l_k^+-\l_k^-\right| _{B_M}$ belongs to $\mathfrak{M}_0\left( B_M\right)$ and 
\begin{multline}\label{Test1}
\int_{\left\lbrace\abs{u}<k\right\rbrace \cap B_M}\left( \abs{\nabla T_k(u)}^{p-2}\nabla T_k(u)-h_M\right) \cdot \nabla \phi dx = \\ \int_{\left\lbrace\abs{u}<k\right\rbrace \cap B_M}\phi f_Mdx + \int_{B_M}\phi d\left( -\tau_k + \l_k^+ -\l_k^-\right)  
\end{multline}
for any $\phi \in W^{1,p}_0\left( B_M\right) \cap L^\infty\left( B_M\right) $ and $k\in K_M$.

Since $u$ is $cap_{1,p,N}-$ quasi-continuous $\left\lbrace x\in B_M : \abs{u}>k\right\rbrace $ is quasi-open, and thus there exists a sequence of functions $\w_n \in W^{1,p}\left( \R^N\right) $ such that $0\leq \w_n \leq \chi_{\left\lbrace\abs{u}>k\right\rbrace }$ and $\w_n \uparrow \chi_{\left\lbrace\abs{u}>k\right\rbrace }$ $cap_{1,p,N}-q.e.$ in $\R^N$ (see Chapter \ref{Capacities}).
For any $\phi \in C_0^\infty\left( B_M\right) $ we can put $\phi\w_n$ as test function in \eqref{Test1} and conclude
\[
\int_{B_M}\phi\w_n d \left(-\tau_k + \l_k^+-\l_k^-\right) =0
\]
for any $k\in K_M$. Since $\left. \left( -\tau_k + \l_k^+-\l_k^-\right) \right|_{B_M} \in \mathfrak{M}_0\left( \R^N\right) $ we can pass to the limit using Proposition \ref{LDC_mu_0} to conclude that for any $k\in K_M$
\[
\left. \left( -\tau_k + \l_k^+-\l_k^-\right) \right. _{\left\lbrace\abs{u}>k\right\rbrace \cap B_M} = 0 .
\]

As above, let now $\w_n$ denote a sequence in $W^{1,p}\left( \R^N\right) $ such that $0\leq \w_n \leq \chi_{\left\lbrace\abs{u}<k\right\rbrace }$ and $\w_n \uparrow \chi_{\left\lbrace\abs{u}<k\right\rbrace }$ $cap_{1,p,N}-q.e.$ in $\R^N$. Let $\phi \in C_0^\infty \left(B_M\right) $ and put $\phi \w_n$ as test function in \eqref{Test1} with both $k$ and $h>k$ in $K_M$ to conclude
\[
\int_{B_M} \phi\w_n d\left(-\tau_h + \l_h^+-\l_h^-\right) =  \int_{B_M}\phi\w_n d\left( -\tau_k + \l_k^+ - \l_k^-\right) ,
\]
and by passing to the limit
\[
\int_{\left\lbrace\abs{u}<k\right\rbrace \cap B_M}\phi d\left( -\tau_h + \l_h^+-\l_h^-\right) = \int_{\left\lbrace\abs{u}<k\right\rbrace \cap B_M} \phi d \left(-\tau_k + \l_k^+-\l_k^-\right) 
\]
which implies
\[
\left. \left( -\tau_h + \l_h^+-\l_h^-\right) \right| _{\left\lbrace\abs{u}<k\right\rbrace \cap B_M} = \left. \left( -\tau_k + \l_k^+-\l_k^-\right) \right| _{\left\lbrace\abs{u}<k\right\rbrace \cap B_M}
\]
for any $h>k$ in $K_M$. As in the proof of Theorem 4.1 in \cite{M05}, this allows us to define a measure $\nu_0 \in \mathfrak{M}_0\left( \R^N\right) $ with support in $B_M$ such that 
\[
\left. \nu_0\right| _{\left\lbrace\abs{u}<k\right\rbrace }= \left. \left( -\tau_h + \l_h^+-\l_h^-\right) \right| _{\left\lbrace\abs{u}<k\right\rbrace \cap B_M}
\]
for any $h\geq k$ in $K_M$. Hence, if we define
\[
\nu_k^+=\left. \left( -\tau_k + \l_k^+-\l_k^-\right) \right| _{\left\lbrace u=k\right\rbrace \cap B_M} \ , \ \nu_k^-=-\left. \left( -\tau_k + \l_k^+-\l_k^-\right) \right| _{\left\lbrace u=-k\right\rbrace \cap B_M}
\]
we can rewrite \eqref{Test1} as
\begin{multline}\label{Test2}
\int_{\left\lbrace\abs{u}<k\right\rbrace \cap B_M}\left( \abs{\nabla T_k(u)}^{p-2}\nabla T_k(u)-h_M\right) \cdot \nabla \phi dx = \\
\int_{\left\lbrace\abs{u}<k\right\rbrace \cap B_M}\phi f_Mdx  + \int_{B_M\cap\left\lbrace\abs{u}<k\right\rbrace }\phi d \nu_0 + \int_{\left\lbrace u=k\right\rbrace \cap B_M} \phi d \nu_k^+ - \int_{\left\lbrace u=-k\right\rbrace \cap B_M}\phi d \nu_k^-
\end{multline}
for any $\phi \in W^{1,p}_0\left( B_M\right) \cap L^\infty\left( B_M\right) $ and $k\in K_M$.

Let us now consider the measures $\nu_0$, $\nu_k^+$, and $\nu_k^-$. For any $\d>0$ let $\w_{\d,k}(s)$ be defined by
\be 
\w_{\d,k}(s)=
\begin{cases}
 0 & ,s<k-\d \\
 \frac{1}{\d}\left( s-k+\d\right) & , k-\d\leq s < k \\
 1 & , k\leq s
\end{cases}
\ee
and choose $k\in K_M$, $\phi \in C_0^\infty\left( B_M\right) $. Plugging $\phi \w_{\d,k}(u)$ as test function in \eqref{Test1} and passing to the limit as $\d\goesto 0$ we conclude
\begin{multline*}
\lim_{\d\goesto 0}\frac{1}{\d}\int_{\left\lbrace k-\d<u<k\right\rbrace \cap B_M}\left( \abs{\nabla T_k(u)}^{p-2} \nabla T_k(u) - h_M\right) \cdot \phi\nabla T_k(u) dx = \\ \int_{\left\lbrace u\geq k\right\rbrace \cap B_M}\phi d \left( -\tau_k +\l_k^+-\l_k^-\right)  = \int_{\left\lbrace u=k\right\rbrace \cap B_M}\phi d \nu_k^+
\end{multline*}
Following the argument in the proof of Theorem 4.1 of \cite{M05} we see that there exists a sequence of positive numbers $k \in K_M$ going to infinity such that $\nu_k^+ \goesto \nu^+$ weakly-$\ast$ in $\mathfrak{M}_b\left( B_M\right) $ as $k\goesto \infty$, for some nonnegative measure $\nu^+$. Choosing now
\[ 
\w_{\d,k}(s)=
\begin{cases}
1 & ,s<-k \\
\frac{1}{\d}\left( -s-k+\d\right) & , -k\leq s < -k + \d \\
0 & , -k+\d \leq s
\end{cases}
\]
we obtain 
\[
\lim_{\d\goesto 0}-\frac{1}{\d}\int_{\left\lbrace-k<u<-k+\d\right\rbrace \cap B_M}\!\!\!\!\!\!\!\!\!\left( \abs{\nabla T_k(u)}^{p-2} \nabla T_k(u) - h_M\right) \cdot \phi\nabla T_k(u) dx = -\int_{\left\lbrace u=-k\right\rbrace \cap B_M}\phi d \nu_k^-
\]
and similarly conclude that, up to a sequence $k \in K_M$, $\nu_k^- \goesto \nu^-$ weakly-$\ast$ in $\mathfrak{M}_b\left( B_M\right) $ for some nonnegative measure $\nu^-$. 

Next, let us note that by the very definition of $u_m$ we have
\[
\int_{B_M}\abs{\nabla u_m}^{p-2}\nabla u_m \cdot \nabla \phi dx = \int_{B_M}\phi d\bar{\mu} - \int_{B_{M} \cap \pd\R^N_+}\phi g_m dx'
\]
for any $\phi \in C_0^\infty\left( B_M\right) $, $m\geq M$, and thus taking limit
\[
\int_{B_M}\abs{\nabla u}^{p-2}\nabla u \cdot \nabla \phi dx = \int_{B_M}\phi d\bar{\mu} - \int_{B_{M} \cap \pd\R^N_+}\phi g dx' .
\]
(Note that we have used the assumptions on $g_m$ with $\phi_m=\phi$). On the other hand, for any such $\phi$ we can take a sequence $k\goesto \infty$, $k\in K_M$, in \eqref{Test2} to conclude
\[
\int_{B_M}\abs{\nabla u}^{p-2}\nabla u \cdot \nabla \phi dx= \int_{B_M}\phi d\bar{\mu}_0  + \int_{B_M}\phi d \nu_0 + \int_{B_M}\phi d\nu^+ - \int_{B_M}\phi d\nu^-
\]
where we have used that $\left( \bar{\mu}_0\right)_M= f_M-\dv h_M$ in the sense of distributions. Thus we get
\[
\int_{B_M}\phi d\bar{\mu}_s - \int_{B_{M} \cap \pd\R^N_+}\phi g dx' = \int_{B_M}\phi d \nu_0 + \int_{B_M}\phi d \nu^+ - \int_{B_M}\phi d \nu^-
\]
which implies $\bar{\mu}_s - g\mathcal{H}=\nu_0+\nu^+-\nu^-$ in $B_M$.

Consider now the function $\b_n(s)$ defined by
\[
\b_n(s)=
\begin{cases}
0 & , s<n\\
\frac{s-n}{n} & n\leq s <2n \\
1 & , 2n \leq s .
\end{cases}
\]
For any nonnegative $\phi\in C_0^\infty\left( B_M\right) $ we have $\phi \b_n(u) \in W_0^{1,p}\left( B_M\right) \cap L^\infty\left( B_M\right) $ and so by \eqref{Test2} we have
\begin{multline*}
\int_{\left\lbrace\abs{u}<k\right\rbrace \cap B_M} \b_n(u)\left( \abs{\nabla T_k(u)}^{p-2}\nabla T_k(u)-h_M\right) \cdot \nabla \phi dx  + \\ \frac{1}{n}\int_{\left\lbrace\abs{u}<k\right\rbrace \cap\left\lbrace n<u<2n\right\rbrace  \cap B_M}\phi\left(  \abs{\nabla T_k(u)}^{p-2}\nabla T_k(u)-h_M\right) \cdot\nabla u dx \\ 
=\int_{\left\lbrace\abs{u}<k\right\rbrace \cap B_M}\phi\b_n(u)f_Mdx + \int_{\left\lbrace\abs{u}<k\right\rbrace \cap B_M}\phi\b_n(u)d \nu_0 + \int_{\left\lbrace u=k\right\rbrace \cap B_M} \phi\b_n(u) d \nu_k^+ 
\end{multline*}
for $k \in K_M$. Using Lebesgue's Dominated Convergence Theorem, the fact that $\abs{\nabla u}^{p-1}\in L^1\left( B_M\right) $, $h_M\in \left( L^{p'}\left( B_M\right) \right) ^N$, the smoothness of $\phi$, the fact that $u$ is finite $cap_{1,p,N}-q.e.$, and that $\b_n(k)=1$ for all $k\geq 2n$ we may take $k\goesto\infty$, for some sequence of $k \in K_M$, to conclude
\begin{multline*}
\int_{B_M} \b_n(u)\abs{\nabla u}^{p-2}\nabla u\cdot\nabla \phi dx + \frac{1}{n}\int_{\left\lbrace n<u<2n\right\rbrace  \cap B_M} \phi \abs{\nabla u}^p dx\\ = \int_{B_M} \b_n(u)\phi d\left( \bar{\mu}_0 + \nu_0 \right)  + \int_{B_M} \phi d \nu^+ 
\end{multline*}
where we have used again Theorem \ref{BGO96_thm} to identify $\left( \bar{\mu}_0\right) _M=f_M-\dv h_M$ for functions in $W^{1,p}_0\left( B_M\right) \cap L^\infty\left( B_M\right)$. Since $\b_n(u)\chi_{B_M} \goesto 0$ weakly-$\ast$ in both $L^\infty\left( \R^N \right) $ and  $L^\infty\left( \R^N;d\left( \bar{\mu}_0 + \nu_0\right)  \right)$, thanks to Lebesgue's Dominated Convergence Theorem and Proposition \ref{LDC_mu_0}, we can take $n\goesto\infty$ and conclude 
\be\label{estimate_limit}
\lim_{n\goesto\infty}\frac{1}{n}\int_{\left\lbrace n<u<2n\right\rbrace  \cap B_M} \phi \abs{\nabla u}^pdx = \int_{B_M} \phi d \nu^+ .
\ee
On the other hand, if we go back to the definition of $u_m$ and put $w=\b_n(u_m)\phi$ as test function with $w^{+\infty}=\phi$, $w^{-\infty}=0$, for $k>2n$, we obtain
\begin{multline}\label{eq_identity_stability}
\frac{1}{n}\int_{\left\lbrace n<u_m<2n\right\rbrace  \cap B_M}\abs{\nabla u_m}^p\phi dx + \int_{B_M} \b_n(u_m)\abs{\nabla u_m}^{p-2}\nabla u_m\cdot \nabla \phi dx \\ = \int_{B_M}\phi \b_n(u_m)d\bar{\mu}_0-\int_{B_M\cap\pd\R^N_+}\phi \b_n(u_m)g_mdx' + \int_{B_M}\phi d\bar{\mu}_s^+ .
\end{multline}
Note that since $\b_n$ is continuous we have $\b_n(u_m)\goesto \b_n(u)$ $a.e.$ in $B_M$, and so we can pass to the limit in the second term above as $m\goesto \infty$. For the third term we use that $\phi \b_n(u_m)$ belongs to $W^{1,p}_0\left( B_M\right)\cap L^\infty\left( B_M\right) $ and Theorem \ref{BGO96_thm} to write
\begin{multline*}
\int_{B_M}\phi \b_n(u_m)d\bar{\mu}_0 = \\ \int_{B_M}\phi \b_n(u_m)f_Mdx + \int_{B_M}\b_n(u_m)\nabla \phi \cdot h_M dx + \frac{1}{n}\int_{B_M\cap\left\lbrace n<u_m<2n\right\rbrace }\phi \nabla T_{2n}(u_m)\cdot h_M dx.
\end{multline*}
Then, by Lebesgue's Dominated Convergence Theorem, and combining the fact that $T_k(u_m)\goesto T_k(u)$ weakly in $W^{1,p}\left( B_M\right) $ with Proposition \ref{prop_Egorov}, we see that we may also take limit as $m\goesto \infty$ above for almost every $n\in \R_+$. 
Similarly, by the continuity of $\b_n$ and Proposition \ref{prop_Egorov}, $\phi \b_n(u_m)\goesto \phi \b_n(u)$ weakly in $W^{1,p}_0\left( B_M\right) $ for almost every $n\in \R_+$. Hence, since $\norm{\phi \b_n(u_m)}_\infty$ is uniformly bounded in $m$, we can use condition \eqref{condition_general_existence_global_solution} to obtain
\[
-\int_{B_M\cap\pd\R^N_+}\phi \b_n(u_m)g_mdx'\goesto -\int_{B_M\cap\pd\R^N_+}\phi \b_n(u)gdx'
\]
as $m\goesto\infty$ for almost every $n\in \R_+$. 

Thus, since we can take $m\goesto \infty$ in the second, third, and fourth term in \eqref{eq_identity_stability}, we can use Fatou's Lemma to conclude
\begin{multline*}
\frac{1}{n}\int_{\left\lbrace n<u<2n\right\rbrace  \cap B_M}\abs{\nabla u}^p\phi dx \leq \liminf_{m\goesto\infty}\frac{1}{n}\int_{\left\lbrace n<u_m<2n\right\rbrace  \cap B_M}\abs{\nabla u_m}^p\phi dx \\
 \leq -\int_{B_M} \b_n(u)\abs{\nabla u }^{p-2}\nabla u \cdot \nabla \phi dx+ \int_{B_M}\b_n(u)d \bar{\mu}_0  -\int_{B_M\cap\pd\R^N_+}\phi \b_n(u)gdx' + \int_{B_M}\phi d \bar{\mu}_s^+ 
\end{multline*}
for almost every $n\in \R_+$.
Passing to the limit as $n\goesto \infty$ as before yields 
\[
\lim_{n\goesto\infty}\frac{1}{n}\int_{\left\lbrace n<u<2n\right\rbrace  \cap B_M}\abs{\nabla u}^p \phi dx \leq \int_{B_M} \phi d \bar{\mu}_s^+ ,
\]
and so comparing with \eqref{estimate_limit} we obtain
\[
\int_{B_M} \phi d\nu^+ \leq \int_{B_M}\phi d\bar{\mu}_s^+ ,
\]
which implies $\nu^+\leq \bar{\mu}_s^+$ in $B_M$. Similarly, one can conclude $\nu^-\leq \bar{\mu}_s^-$. This implies in particular that $\nu^+$ and $\nu^-$ are singular with respect to $cap_{1,p,N}$, and since $\bar{\mu}_s-g\mathcal{H}=\nu_0+\nu^+ -\nu^-$ we conclude that $\nu_0\equiv -g\mathcal{H}$. Recalling that $\bar{\mu}_s^+$ and $\bar{\mu}_s^-$ have disjoint support we further conclude $\nu^+=\bar{\mu}_s^+$ and $\nu^-=\bar{\mu}_s^-$ in $B_M$. In particular this allows us to rewrite \eqref{Test2} as 
\begin{multline}\label{Test3}
\int_{\left\lbrace\abs{u}<k\right\rbrace \cap B_M}\left( \abs{\nabla T_k(u)}^{p-2}\nabla T_k(u)-h_M\right) \cdot \nabla \phi dx = \\ \int_{\left\lbrace\abs{u}<k\right\rbrace \cap B_M}\phi f_Mdx - \int_{\left\lbrace\abs{u}<k\right\rbrace \cap B_M}\phi gdx'  + \int_{\left\lbrace u=k\right\rbrace \cap B_M} \phi d \nu_k^+ - \int_{\left\lbrace u=-k\right\rbrace \cap B_M}\phi d \nu_k^-
\end{multline}
for any $\phi\in W^{1,p}_0\left( B_M\right) \cap L^\infty\left( B_M\right)$ and $k\in K_M$. 

We are now ready to finish. Let $w\in W^{1,\infty}\left( \R\right) $ with $w'$ compactly supported, and let $\phi \in W^{1,r}\left( \R^N\right) $, for some $r>N$, be compactly supported and such that $w(u)\phi \in W^{1,p}\left( \R^N\right) $. We write $w(\pm\infty)=\lim_{s\goesto\pm\infty}w(s)$. Choosing $M \in \N$ large enough we can assume $w(u)\phi\in W^{1,p}_0\left( B_{M}\right) $ so that $w(u)\phi$ is a valid test function for \eqref{Test3} and thus
\begin{multline*}
\int_{\left\lbrace\abs{u}<k_j\right\rbrace \cap B_M}\left( \abs{\nabla u}^{p-2}\nabla u - h_M\right) \cdot \nabla \left( w(u)\phi\right)  dx =\\ \int_{\left\lbrace\abs{u}<k_j\right\rbrace \cap B_M} w(u)\phi f_Mdx -\int_{\left\lbrace\abs{u}<k_j\right\rbrace \cap B_M} w(u)\phi gdx'  \\ + \int_{\left\lbrace u=k_j\right\rbrace \cap B_M}w(u)\phi d \nu_{k_j}^+ - \int_{\left\lbrace u=-k_j\right\rbrace \cap B_M}w(u)\phi d \nu_{k_j}^- ,
\end{multline*}
where we have chosen $k_j\in K_M$ to be the sequence such that $\nu_{k_j}^\pm\goesto \nu^\pm$ weakly-$\ast$ as $k_j\goesto\infty$. 
Let $\tau$ be such that $w(s)$ is constant in $(-\tau,\tau)^c$. Then, if $k_j>\tau$, 
\begin{multline*}
\int_{\left\lbrace\abs{u}<k_j\right\rbrace \cap B_M}\left( \abs{\nabla u}^{p-2}\nabla u-h_M\right) \cdot \nabla \left( w(u)\phi\right)  dx = \\ \int_{\left\lbrace\abs{u}<\tau\right\rbrace \cap B_M}\left( \abs{\nabla u}^{p-2}\nabla u - h_M\right)  \cdot \nabla \left( w(u)\phi\right)  dx \\
+ w(+\infty)\int_{\left\lbrace\tau<u<k_j\right\rbrace \cap B_M}\left( \abs{\nabla u}^{p-2}\nabla u - h_M\right)  \cdot \nabla \phi dx \\ + w(-\infty)\int_{\left\lbrace-k_j<u<-\tau\right\rbrace \cap B_M}\left( \abs{\nabla u}^{p-2}\nabla u - h_M\right)  \cdot \nabla \phi dx . 
\end{multline*}
Since $r>N$ we have that $r'<\frac{N}{N-1}$. Hence, $\left( \abs{\nabla u}^{p-2}\nabla u - h_M\right)  \cdot \nabla \phi \in L^1\left( B_M\right) $ and since $u$ is finite $a.e.$ we take $k_j\goesto \infty$ in the last two terms above and obtain
\begin{multline*}
\int_{\left\lbrace\abs{u}<k_j\right\rbrace \cap B_M}\left( \abs{\nabla u}^{p-2}\nabla u - h_M\right) \cdot \nabla \left( w(u)\phi\right)  dx \goesto \\ \int_{B_M}\left( \abs{\nabla u}^{p-2}\nabla u - h_M\right) \cdot \nabla \left( w(u)\phi\right)  dx.
\end{multline*}
We know that $u$ is finite $cap_{1,p,N}-q.e.$ in $B_M$ and so $\chi_{\left\lbrace\abs{u}<k_j\right\rbrace }\goesto 1$ $cap_{1,p,N}-q.e.$ in $B_M$. It follows that
\begin{multline*}
\int_{\left\lbrace\abs{u}<k_j\right\rbrace \cap B_M} w(u)\phi f_Mdx -\int_{\left\lbrace\abs{u}<k_j\right\rbrace \cap B_M} w(u)\phi gdx'  \goesto \\ \int_{B_M} w(u)\phi f_Mdx -\int_{B_M\cap\pd\R^N_+} w(u)\phi gdx'
\end{multline*}
as $k_j\goesto\infty$. Recall that $\nu_{k_j}^\pm$ are concentrated in $\left\lbrace u=\pm k_j\right\rbrace \cap B_M$, respectively. Thus, assuming $k_j>\tau$, we use that $\phi \in C_0\left( B_M\right) $ to conclude
\[
\int_{B_M} w(u)\phi d\nu_{k_j}^{\pm} = w(\pm\infty)\int_{B_M} \phi d\nu_{k_j}^\pm \goesto w(\pm\infty)\int_{B_M} \phi d \nu^\pm = w(\pm\infty)\int_{B_M} \phi d\bar{\mu}_s^\pm 
\]
as $k_j\goesto \infty$. Putting together all the above we get
\begin{multline*}
\int_{B_M}\abs{\nabla u}^{p-2}\nabla u\cdot \nabla \left( w(u)\phi\right)  dx = \int_{B_M} w(u)\phi d\bar{\mu}_0 - \int_{B_M\cap\pd\R^N_+} w(u)\phi gdx'\\ + w(+\infty)\int_{B_M} \phi d\bar{\mu}_s^+ - w(-\infty)\int_{B_M} \phi d\bar{\mu}_s^- .
\end{multline*}
Hence, by the results in \cite{B-V03}, $u$ is a local renormalized solution of $-\lap_pu=\bar{\mu}- g\mathcal{H}$ in $\R^N$.

Now we show the strong convergence of the truncates. Fix $M\in \N$, $k>0$, and let $\phi \in C_0^\infty\left( B_M\right) $. By testing against $T_k(u_m)\phi$ in the definition of $u_m$ as renormalized solution, for any $m\geq M$ we have
\begin{multline*}
\int_{B_M}\phi \abs{\nabla T_k(u_m)}^pdx + \int_{B_{M}}T_k(u_m)\abs{\nabla u_m}^{p-2}\nabla u_m\cdot\nabla \phi dx =\\
-\int_{B_{M}\cap\pd\R^N_+}\phi T_k(u_m)g_mdx' + \int_{B_{M}}T_k(u_m)\phi d\bar{\mu}_0 + k\int_{B_{M}}\phi d\bar{\mu}_s^+ - k \int_{B_{M}}\phi d\bar{\mu}_s^-.
\end{multline*}
Similarly, 
\begin{multline*}
\int_{B_M}\phi \abs{\nabla T_k(u)}^pdx + \int_{B_{M}}T_k(u)\abs{\nabla u}^{p-2}\nabla u\cdot\nabla \phi dx =\\
-\int_{B_{M}\cap\pd\R^N_+}\phi T_k(u)gdx' + \int_{B_{M}}T_k(u)\phi d\bar{\mu}_0 + k\int_{B_{M}}\phi d\bar{\mu}_s^+ - k \int_{B_{M}}\phi d\bar{\mu}_s^-.
\end{multline*}
Comparing the above identities we have 
\begin{multline*}
\int_{B_M}\phi\abs{\nabla T_k(u)}^{p}dx -  \int_{B_M}\phi\abs{\nabla T_k(u_m)}^{p}dx= \\ \int_{B_M} \phi T_k(u)d\bar{\mu}_0- \int_{B_M} \phi T_k(u_m)d\bar{\mu}_0 + \int_{B_M\cap\pd\R^N_+} \phi T_k(u_m)g_mdx' - \int_{B_M\cap\pd\R^N_+}\phi T_k(u)gdx'  \\ -  \int_{B_{M}}T_k(u)\abs{\nabla u}^{p-2}\nabla u\cdot\nabla \phi dx + \int_{B_{M}}T_k(u_m)\abs{\nabla u_m}^{p-2}\nabla u_m\cdot\nabla \phi dx .
\end{multline*}
Writing again $\left( \bar{\mu}_0\right) _M= f_M -\dv h_M$ we use that $\phi\in C_0^\infty\left( B_M\right)$ and that $ T_k(u_m)\goesto T_k(u)$ weakly in $W^{1,p}\left( B_M\right)$ and weakly-$\ast$ in $L^\infty\left( B_M\right)$ to obtain
\[
\int_{B_M} \phi T_k(u_m)d\bar{\mu}_0\goesto \int_{B_M} \phi T_k(u)d\bar{\mu}_0 
\]
as $m\goesto \infty$. Note that by condition \eqref{condition_general_existence_global_solution}
\[
\int_{B_M\cap\pd\R^N_+} \phi T_k(u_m)g_mdx' \goesto \int_{B_M\cap\pd\R^N_+}\phi T_k(u)gdx'
\]
as $m\goesto\infty$. Moreover, since $\abs{\nabla u_m}^{p-2}\nabla u_m\goesto \abs{\nabla u }^{p-2}\nabla u$ strongly in $\left( L^q\left( B_M\right) \right) ^N$ for some $q>1$, while $T_k(u_m)\goesto T_k(u)$ strongly in $L^r\left( B_M\right) $ for any $1\leq r <\infty$, we get 
\[
\int_{B_{M}}T_k(u_m)\abs{\nabla u_m}^{p-2}\nabla u_m\cdot\nabla \phi dx \goesto \int_{B_{M}}T_k(u)\abs{\nabla u}^{p-2}\nabla u\cdot\nabla \phi dx
\]
as $m\goesto\infty$. Hence,  
\[
\lim_{m\goesto\infty}  \int_{B_M}\phi\abs{\nabla T_k(u_m)}^{p}dx = \int_{B_M}\phi\abs{\nabla T_k(u)}^{p}dx 
\]
for any $\phi \in C_0^\infty\left( B_M\right) $, which implies that
\[
\lim_{m\goesto\infty} \norm{\nabla T_k(u_m)}_{L^p\left( B_{M'}\right) }= \norm{\nabla T_k(u)}_{L^p\left( B_{M'}\right) }
\]
for any $M> M'\in \N$. 
Using the above, the inequality $\abs{\abs{a+b}-\abs{a}-\abs{b}}\leq 2 \abs{b}$ with $a=\abs{\nabla T_k(u_m)}^p-\abs{\nabla T_k(u)}^p$ and $b=\abs{\nabla T_k(u)}^p$, and the fact that $\nabla T_k(u_m)\goesto \nabla T_k(u)$ $a.e.$ in $B_{M'}$, we obtain that $\abs{\nabla T_k(u_m)}^p\goesto \abs{\nabla T_k(u)}^p$ strongly in $L^1\left( B_{M'}\right)$. Then, by Vitalli's Theorem, $\nabla T_k(u_m)\goesto \nabla T_k(u)$ strongly in $\left( L^p\left( B_{M'}\right)\right) ^N$, from which the claim follows.
\end{proof}

Proving Theorem \ref{existenceExtended} is now trivial:
\begin{proof}[Proof of Theorem \ref{existenceExtended}.]
	Let $\bar{\mu}_m$ be the restriction of $\bar{\mu}$ to $B_m$. Since $\abs{\bar{\mu}_m}\left( B_m\right) \leq \bar{\mu}\left( \R^N\right) <\infty$ we can apply Lemma \ref{general_existence_of_limit_functions} and Lemma \ref{general_existence_global_solution} with $g=g_m\equiv 0$.
\end{proof}

\chapter{Symmetric Solutions}\label{Symmetry}

\section{Symmetry}

In this section we show that any solution of the extended problem given by Theorem \ref{existenceExtended} must be symmetrical with respect to $\pd\R^N_+$ whenever the measure $\bar{\mu}$ is supported in $\pd\R^N_+$. This symmetry will allows us to recover a solution to the original problem, i.e., equation \eqref{equation}. 

\begin{thm}\label{symmetry}
	Let $\O$ be any bounded domain in $\R^N$ that is symmetric with respect to the hyperplane $\pd\R^N_+$. Let $\bar{\mu} \in \mathfrak{M}_b\left(\O\right)$ be supported in $\pd\R^N_+\cap \O$ and let $u$ be a renormalized solution to
	\[
	\begin{cases}
	-\lap_p u = \bar{\mu} & \mbox{ in } \O \\
	u = 0 & \mbox{ on } \pd \O .
	\end{cases}
	\]
	Then $u(x',x_N)=u(x',-x_N)$ $a.e.$ in $\O$.
\end{thm}
\begin{proof}
	In what follows we write $\O^+=\O\cap \R^N_+$ and for any $f$ defined in $\O$ we denote by $f^\ast$ its reflection with respect to $\pd\R^N_+$, i.e., $f^\ast(x',x_N)= f(x',-x_N)$. 
	
	Let us first show that $u^\ast$ is also a renormalized solution of the above problem. 
	Indeed, this is clear when we observe that if $\w$ is a test function with respect to $u^\ast$ then $w^\ast$ is a valid test function with respect to $u$. Hence we conclude
	\begin{multline*}
	\int_{\O}\abs{\nabla u^\ast}^{p-2}\nabla u^\ast\cdot\nabla w dx =\int_{\O}\abs{\nabla u}^{p-2}\nabla u\cdot\nabla w^\ast dx = \\ \int_{\O} w d \bar{\mu}_0 + \int_{\O}w^{+\infty}d \bar{\mu} _s^+ - \int_{\O}w^{-\infty}d\bar{\mu} _s^- 
	\end{multline*}
	as required, since $w^\ast=w$ on $\pd\R^N_+$ and $\O$ is invariant under $x_N\mapsto -x_N$.
	
	To continue, let us note that $T_k(u)-T_k(u^\ast)\in W^{1,p}_0\left( \O^+\right)$ for any $k>0$. Indeed, since $T_k(u) \in W_0^{1,p}\left( \O\right) $ we can choose a sequence $\phi_n\in C_0^\infty\left( \O\right) $ such that $\phi_n\goesto T_k(u_m)$ in $W^{1,p}\left( \O\right) $. Then $\phi_n-\phi_n^\ast \goesto T_k(u_m)-T_k(u_m^\ast)$ in $W^{1,p}\left( \O^+\right) $ and since $\phi_n-\phi_n^\ast \in C\left( \overline{\O^+}\right) \cap W^{1,p}\left( \O^+\right)$ vanishes in $\pd \O^+$ we conclude $\phi_n-\phi_n^\ast \in W^{1,p}_0\left( \O^+\right) $ and thus our claim follows.  
	
	By the equivalence of definitions of renormalized solutions (see Remark \ref{remark_on_equivalent_definitions_DMOP99}), we have that for every $k>0$ there exists two nonnegative measures $\l_{k}^+$, $\l_{k}^- \in \mathfrak{M}_0\left(\O\right) $ supported in $\left\lbrace u=k\right\rbrace $ and $\left\lbrace u=-k\right\rbrace $ respectively, such that $\l_{k}^\pm \goesto  \bar{\mu}_s^\pm$ as $k \goesto \infty$ in the narrow topology of measures, and the truncations $T_k(u)$ satisfy
	\begin{multline}\label{Testu_v2}
	\int_{\left\lbrace\abs{u}<k\right\rbrace }\abs{\nabla T_k(u)}^{p-2}\nabla T_k(u) \cdot \nabla v dx = \\ \int_{\left\lbrace u=k\right\rbrace } v d\l_{k}^+ - \int_{\left\lbrace u=-k\right\rbrace } v d\l_{k}^- + \int_{\left\lbrace\abs{u}<k\right\rbrace }vd\bar{\mu}_0
	\end{multline}
	for every $v\in W^{1,p}_0\left( \O\right) \cap L^\infty\left( \O\right) $. In particular
	\[
		\l_{k}^\pm\left( \O^+\right) \goesto \bar{\mu}_s^\pm\left( \O^+\right)=0 
	\]
	as $k\goesto \infty$ since $\bar{\mu} _s^\pm$ is supported in $\pd\R_+^N$. 
	
	We now extend $T_k(u)-T_k(u^\ast)$ by $0$ outside $\O^+$. Since $T_k(u)-T_k(u^\ast) \in W^{1,p}_0\left( \O\right) \cap L^\infty\left( \O\right) $ is a valid test function for \eqref{Testu_v2} which vanishes in $\pd\R^N_+\cap \O$ we get
	\begin{multline*}
	\int_{\left\lbrace\abs{u}<k\right\rbrace }\abs{\nabla T_k(u)}^{p-2}\nabla T_k(u)\cdot \nabla \left[  T_k(u)-T_k(u^\ast)\right]  dx  =\\  \int_{\left\lbrace u=k\right\rbrace }  \left[  T_k(u)-T_k(u^\ast)\right]   d\l_{k}^+ - \int_{\left\lbrace u=-k\right\rbrace }  \left[  T_k(u)-T_k(u^\ast)\right]  d\l_{k}^- .
	\end{multline*}
	Arguing in the same way for $u^\ast$ we obtain sequences $(\l_{k}^\pm)^\ast$ converging to $\bar{\mu} _s^\pm$ such that \eqref{Testu_v2} holds with $u^\ast$ in place of $u$, and so testing against $T_k(u)-T_k(u^\ast)$ and subtracting it from the previous equality we get
	\begin{multline}\label{difference}
	\int_{\O^+}\left[  \abs{\nabla T_k(u)}^{p-2}\nabla T_k(u) -\abs{\nabla T_k(u^\ast)}^{p-2}\nabla T_k(u^\ast)\right]  \cdot \nabla \left[  T_k(u)-T_k(u^\ast)\right]  dx= \\  \int_{\O^+} \left[  T_k(u)-T_k(u^\ast) \right] d\l_{k}^+ 
	- \int_{\O^+} \left[  T_k(u)-T_k(u^\ast) \right] d\l_{k}^- \\ - \int_{\O^+}  \left[ T_k(u)-T_k(u^\ast) \right] d(\l_{k}^+)^\ast 
	+ \int_{\O^+} \left[  T_k(u)-T_k(u^\ast)\right]  d(\l_{k}^-)^\ast .
	\end{multline}
	Using the well-known inequality
	\begin{multline}\label{coercitivity}
	\sum_{i=1}^N \left( \abs{z}^{p-2}z_i-\abs{\zeta}^{p-2}\zeta_i\right) (z_i-\zeta_i) \geq \\ \gamma
	\begin{cases}
	\left( \frac{1}{4}\right)^{p-1}\abs{z-\zeta}^p \ &, \ \mbox{ if } p\geq 2 \\
	\left( \frac{1}{4}\right)\abs{z-\zeta}^2\left(\abs{z}+\abs{\zeta}\right)^{p-2} \ &, \ \mbox{ if } p\leq 2
	\end{cases}
	\end{multline}
	for some $\gamma>0$, it follows immediately from \eqref{difference} that
	\begin{multline*}
	\int_{\O^+}\abs{\nabla T_k(u)- \nabla T_k(u^\ast)}^p dx \leq \\ C(p)k \left[\abs{\l_{k}^+}\left( \O^+\right) + \abs{\l_{k}^-}\left( \O^+\right) + \abs{(\l_{k}^+)^\ast}\left( \O^+\right)+\abs{(\l_{k}^-)^\ast}\left( \O^+\right)\right] 
	\end{multline*}
	when $p\geq 2$. When $1<p<2$ we use Holder's inequality first to get
	\begin{multline*}
	\int_{\O^+}\abs{\nabla T_k(u)-\nabla T_k(u^\ast)}^p dx= \\ \int_{\O^+}\frac{\abs{\nabla T_k(u)-\nabla T_k(u^\ast)}^p}{\left(\abs{\nabla T_k(u)} +\abs{\nabla T_k(u^\ast)}\right)^{\frac{p}{2}(2-p)}}\left(\abs{\nabla T_k(u)} +\abs{\nabla T_k(u^\ast)}\right)^{\frac{p}{2}(2-p)}dx
	\leq \\ \left\lbrace \int_{\O^+}\frac{\abs{\nabla T_k(u)-\nabla T_k(u^\ast)}^2}{\left(\abs{\nabla T_k(u)} +\abs{\nabla T_k(u^\ast)}\right)^{2-p}}dx\right\rbrace^{\frac{p}{2}}\left\lbrace \int_{\O^+}\left(\abs{\nabla T_k(u)} +\abs{\nabla T_k(u^\ast)}\right)^pdx\right\rbrace^{\frac{2-p}{2}} 
	\end{multline*}
	which then by \eqref{coercitivity}, \eqref{difference}, and \eqref{estimate1_DMMOP99} yields
	\begin{multline*}
	\int_{\O^+}\abs{\nabla T_k(u)-\nabla T_k(u^\ast)}^p dx\leq  \\  C(p) \left\lbrace \int_{\O^+}\!\!\!\left[  \abs{\nabla T_k(u)}^{p-2}\nabla T_k(u) -\abs{\nabla T_k(u^\ast)}^{p-2}\nabla T_k(u^\ast)\right]  \cdot \nabla \left[  T_k(u)-T_k(u^\ast)\right] dx\right\rbrace^{\frac{p}{2}}\\
	\times \left\lbrace k\abs{\bar{\mu}}\left( \O\right) \right\rbrace^{\frac{2-p}{2}} \\
	\leq  C(p) k\left[\abs{\l_{k}^+}\left( \O^+\right) + \abs{\l_{k}^-}\left( \O^+\right) + \abs{(\l_{k}^+)^\ast}\left( \O^+\right)+\abs{(\l_{k}^-)^\ast}\left( \O^+\right)\right]^{\frac{p}{2}}\left\lbrace \abs{\bar{\mu}}\left( \O\right) \right\rbrace^{\frac{2-p}{2}} .
	\end{multline*}
	Thus we see that for any $1<p\leq N$ there holds
	\[
	\frac{1}{k}\int_{\O^+}\abs{\nabla T_k(u)-\nabla T_k(u^\ast)}^pdx\goesto 0
	\]
	as $k\goesto\infty$. By symmetry, the same is true in $\O\cap\R^N_-=\O\cap \{(x',x_N)\in \R^N \ : \ x_N<0\}$. Hence we can apply the partial uniqueness result stated in Theorem $10.4$ of \cite{DMOP99} to conclude that $u=u^\ast$ $a.e.$ in $\O$. 
\end{proof}

\section{Existence from symmetry}

Now we are ready to prove an existence result for problem \eqref{equation} in the case $g\equiv 0$. We will state it as a corollary to the following theorem.

\begin{thm}\label{thm_on_existence_from_symmetry}
	Let $1<p\leq N$ and $\mu\in \mathfrak{M}_b\left( \pd\R^N_+\right) $. Suppose $u$ is a local renormalized solution of $-\lap_p u = 2\mu$ in $\R^N$ that is symmetric with respect to the hyperplane $\pd\R^N_+$. Then the restriction of $u$ to $\R^N_+$ is a renormalized solution of
	\[
	\begin{cases}
		-\lap_pu=0 & \mbox{ in } \R^N_+ \\
		\abs{\nabla u}^{p-2}u_\n = \mu & \mbox{ on } \pd\R^N_+ .
	\end{cases}
	\]
\end{thm}
\begin{proof}
	It is clear from the definition of local renormalized solution that the restriction of $u$ to $\R^N_+$ satisfies conditions $(1)$, $(2)$, and $(3)$ of Definition \ref{defn}. Hence, we only need to show that $(5)$ holds.
	
	Assume that $w\in W^{1,p}\left( \R^N_+\right) $ has compact support in $\overline{\R^N_+}$ and trace in $L^\infty\left( \pd\R^N_+\right) $ and there exist $k>0$, $r>N$, and functions $w^{\pm\infty}\in W^{1,r}\left( \R^N_+\right) $ such that
	\[
	\begin{cases}
	w=w^{+\infty} & \mbox{ $a.e.$ in } \left\lbrace x\in \R^N_+ \ : \ u>k\right\rbrace \\
	w=w^{-\infty} & \mbox{ $a.e.$ in } \left\lbrace x\in \R^N_+ \ : \ u<-k\right\rbrace  .
	\end{cases} 
	\]
	Choose $L$ such that $\abs{w}\leq L$ $a.e.$ in $\pd\R^N_+$ and $\abs{w^{\pm\infty}}\leq L$ in $\R^N_+$. Let us extend $w$ and $w^{\pm\infty}$ to $\R^N$ by even reflection, i.e., $w(x',x_N)=w(x',-x_N)$ for $x_N<0$ and similarly for $w^{\pm\infty}$. Note that since $u$ is symmetric with respect to $\pd\R^N_+$ we have
	\[
	\begin{cases}
	w=w^{+\infty} & \mbox{ $a.e.$ in } \left\lbrace x\in \R^N \ : \ u>k\right\rbrace \\
	w=w^{-\infty} & \mbox{ $a.e.$ in } \left\lbrace x\in \R^N \ : \ u<-k\right\rbrace  .
	\end{cases} 
	\]
	Next, we let 
	\[
	\Phi_\e=
	\begin{cases}
	0 & , x_N\leq 0 \\
	\frac{x_N}{\e} & , 0<x_N\leq \e \\
	1 & , \e\leq x_N
	\end{cases}
	\]
	and 
	\[
	\Psi_\e=
	\begin{cases}
	0 & , x_N\leq -\e \\
	\frac{x_N+\e}{\e} & , -\e<x_N\leq 0 \\
	1 & , 0\leq x_N .
	\end{cases} 
	\]
	Then for any $l>L$ we see that $T_l(w)\Psi_\e \in W^{1,p}\left( \R^N\right) \cap L^\infty\left( \R^N\right) $ is an adequate test function and thus we get
	\begin{multline*}
	\int_{\R^N}\Psi_\e\abs{\nabla u}^{p-2}\nabla u\cdot\nabla T_l(w)dx + \frac{1}{\e}\int_{\left\lbrace-\e\leq x_N \leq 0\right\rbrace }T_l(w)\abs{\nabla u}^{p-2}u_{x_N}dx = \\ 2\left\lbrace \int_{\pd\R_+^N} w d\mu_0 + \int_{\pd\R^N_+}w^{+\infty}d \mu_s^+ - \int_{\pd\R^N_+}w^{-\infty}d\mu_s^-\right\rbrace  .
	\end{multline*}
	By now taking $T_l(w)\Phi_\e$ as test function we get
	\[
	\int_{\R^N}\Phi_\e\abs{\nabla u}^{p-2}\nabla u\cdot\nabla T_l(w)dx + \frac{1}{\e}\int_{\left\lbrace0\leq x_N \leq \e\right\rbrace }T_l(w)\abs{\nabla u}^{p-2} u_{x_N}dx = 0 .
	\]
	By the symmetry of $u$ we have that $u_{x_N}(x',x_N)=-u_{x_N}(x',-x_N)$ and so
	\[
	\frac{1}{\e}\int_{\left\lbrace-\e\leq x_N \leq 0\right\rbrace }T_l(w)\abs{\nabla u}^{p-2}u_{x_N}dx= -\frac{1}{\e}\int_{\left\lbrace0\leq x_N \leq \e\right\rbrace }T_l(w)\abs{\nabla u}^{p-2} u_{x_N} dx.
	\]
	Adding up the previous equalities we conclude
	\begin{multline*}
	\int_{\R^N}\Psi_\e\abs{\nabla u}^{p-2}\nabla u\cdot\nabla T_l(w) dx+ 	\int_{\R^N}\Phi_\e\abs{\nabla u}^{p-2}\nabla u\cdot\nabla T_l(w)dx = \\ 2\left\lbrace \int_{\pd\R_+^N} w d\mu_0 + \int_{\pd\R^N_+}w^{+\infty}d \mu_s^+ - \int_{\pd\R^N_+}w^{-\infty}d\mu_s^- \right\rbrace 
	\end{multline*}
	and by Lebesgue's Dominated Convergence Theorem we let $\e\goesto 0$ to obtain
	\[
	\int_{\R^N_+}\abs{\nabla u}^{p-2}\nabla u\cdot\nabla T_l(w)dx = \int_{\pd\R_+^N} w d\mu_0 + \int_{\pd\R^N_+}w^{+\infty}d \mu_s^+ - \int_{\pd\R^N_+}w^{-\infty}d\mu_s^- .
	\]
	Writing 
	\begin{multline*}
	\int_{\R^N_+}\abs{\nabla u}^{p-2}\nabla u\cdot\nabla T_l(w)dx = \int_{\R^N_+\cap \left\lbrace\abs{u}\leq k\right\rbrace }\abs{\nabla T_k(u)}^{p-2}\nabla T_k(u)\cdot\nabla T_l(w)dx + \\ \int_{\R^N_+\cap \left\lbrace u>k\right\rbrace }\abs{\nabla u}^{p-2}\nabla u\cdot\nabla w^{+\infty}dx + \int_{\R^N_+\cap \left\lbrace u<-k\right\rbrace }\abs{\nabla u}^{p-2}\nabla u\cdot\nabla w^{-\infty}dx
	\end{multline*}
	we use the fact that $\nabla T_l(w)\goesto \nabla w$ weakly in $\left( L^{p}\left( \R^N_+\right)\right) ^N $ to take $l\goesto \infty$ above, and so conclude
	\[
	\int_{\R^N_+}\abs{\nabla u}^{p-2}\nabla u\cdot\nabla wdx = \int_{\pd\R_+^N} w d\mu_0 + \int_{\pd\R^N_+}w^{+\infty}d \mu_s^+ - \int_{\pd\R^N_+}w^{-\infty}d\mu_s^- 
	\]
	thus completing the proof of the theorem.
\end{proof}

\begin{thm}\label{existence_half_space}
	Let $1<p\leq N$ and $\mu\in \mathfrak{M}_b\left( \pd\R^N_+\right) $. Then there exists a renormalized solution to 
	\[
	\begin{cases}
	-\lap_pu=0 & \mbox{ in } \R^N_+ \\
	\abs{\nabla u}^{p-2}u_\n = \mu & \mbox{ on } \pd\R^N_+ .
	\end{cases}
	\]
\end{thm}
\begin{proof}
	Apply Theorem \ref{existenceExtended} to obtain a local renormalized solution to $-\lap_pu=2\mu$ in $\R^N$. By the construction of $u$, and in view of Theorem \ref{symmetry}, $u$ is symmetric with respect to $\pd\R^N_+$. Then the result follows from an application of the previous theorem.
\end{proof}

\chapter{Nonlinear problems with absorption}\label{absorption}

\section{The subcritical case}\label{subcritical}

We now consider the problem of finding renormalized solutions to problem \eqref{equation} with a nonlinear term $g(u)$. The fact that $g(s)$ is subcritical is expressed in the following assumption.

\begin{assumption}\label{assumption_subcritical}\mbox{ }
	\begin{enumerate}
		\item [(1)] $g:\R\goesto \R$ is a continuous function such that $g(s)s\geq0$.
		\item [(2)] Define $\tilde{g}:\R_+\goesto\R$ by $\tilde{g}(s)=\sup_{[-s,s]}\abs{g(t)}$. If $1<p<N$ we assume
		\[
		\int_1^\infty\tilde{g}(s)s^{-\frac{p(N-2)+1}{N-p}}ds<\infty.
		\]
		If $p=N$ we assume that there exists $\gamma>0$ such that
		\[
		\int_1^\infty\tilde{g}(s)e^{-\gamma N s}ds<\infty .
		\]
	\end{enumerate}
\end{assumption}

\begin{remark}\label{remark_definition_critical_exponent}
	In the special case when $g(s)=\abs{s}^{q-1}s$, $q\geq 0$, Assumption \ref{assumption_subcritical} holds whenever 
	\[
	q<
	\begin{cases}
	\frac{(N-1)(p-1)}{N-p} &, \mbox{ if } 1<p<N\\
	\infty &, \mbox{ if } p=N.
	\end{cases}
	\] 
	Hence, we say that $q_c:=(N-1)(p-1)/(N-p)$ is a \textit{critical exponent} for problem \eqref{equation}, and the problem is subcritical whenever $q<q_c$.
	
\end{remark}

We will use the tools developed in Chapters \ref{Extension} and \ref{Symmetry} to obtain a renormalized solution of \eqref{equation} as the limit of renormalized solutions to
\be\label{eq_intermediate_step}
\begin{cases}
-\lap_pu=\mu- g(u)\mathcal{H} & \mbox{ in } B_m \\
u=0 & \mbox{ on } \pd B_m.
\end{cases}
\ee
To find solutions of the above problem we use the theory developed in \cite{Veron16} for the equation
\be\label{renormalized_solution_B-VHV14}
\begin{cases}
-\lap_pu  + g(x,u)=\mu & \mbox{ in } \O \\
u= 0 & \mbox{ on } \pd\O
\end{cases}
\ee
in bounded domains. In order to pass from \eqref{renormalized_solution_B-VHV14} to \eqref{eq_intermediate_step} we apply the theory for problem \eqref{renormalized_solution_B-VHV14} to a sequence $g_n(x,u)$ obtained by multiplying $g(u)$ by an adequately chosen sequence $\zeta_n(x_N)$, and then show that the associated sequence of solutions converges to a solution of problem \eqref{eq_intermediate_step}. 

We define $\zeta \in C^\infty\left(\R \right)$ as
\[
\zeta(t)=\frac{1}{\pi}\left( \frac{1}{1+t^2}\right) .
\] 
Note that $\norm{\zeta}_{L^1}=1$. Then, for $n\in \N$ we define
\be\label{def_g_n}
\zeta_n(t)=n\zeta\left(nt\right) \ , \ g_n(x,s)=\zeta_n(x_N)g(s) .
\ee

We start by defining renormalized solutions to problems \eqref{eq_intermediate_step} (in a general bounded domain $\O$) and \eqref{renormalized_solution_B-VHV14}.

\begin{defn}
	Let $\O$ be a bounded domain, $\mu\in\mathfrak{M}_b\left( \O\right) $, and $g:\R\goesto \R$. Then a function $u$ defined in $\O$ is a \textit{renormalized solution} to problem \eqref{eq_intermediate_step} if $u$ is finite $a.e.$ in $\O \cap \pd\R^N_+$, $g(u)\in L^1\left( \O \cap \pd\R^N_+\right) $ and $u$ is a renormalized solution to problem \eqref{renormalized_solution_DMOP99} with datum $\mu- g(u)\mathcal{H}$ in the sense of Definition \ref{definition_DMOP99}.
	
	Similarly, if $g:\O\times \R\goesto \R$ then a function $u$ defined in $\O$ is a \textit{renormalized solution} to problem \eqref{renormalized_solution_B-VHV14} if $g(x,u)\in L^1\left( \O\right) $ and $u$ is a renormalized solution to problem \eqref{renormalized_solution_DMOP99} with datum $\mu- g(x,u)$ in the sense of Definition \ref{definition_DMOP99}.
\end{defn}

The following result is obtained in the proof of Theorem 5.1.2 in \cite{Veron16} by testing against $w_s=\tanh(sT_k(u))$, $s>0$, and taking $s\goesto \infty$:
\begin{prop}\label{Veron16_bound_on_g}
	Let $u$ be a renormalized solution to problem \eqref{renormalized_solution_B-VHV14}, where $g(x,\cdot)$ is continuous and satisfies $g(x,s)s\geq0$ for all $x\in \O$ and $s\in \R$. Then 
	\[
	\int_{\O}\abs{g(x,u)}dx \leq \abs{\mu}\left( \O\right)  .
	\]
\end{prop}

The next lemma collects some relationships between capacities and Lebesgue measure.

\begin{lem}\label{L_trace_characteristic}
	Let $1<p\leq N$. There exists constants $C_1(M,N,p)$, $C_2(M,N,p)$, $C_3(N,p)$, and $C_4(M,N,p)$ such that for all Borel sets $E\subset B_M \subset \R^N$ there holds 
	\begin{enumerate}
		\item $\abs{E\cap\pd\R^N_+} \leq C_1cap_{1-\frac{1}{p},p,N-1}\left( E\cap \pd\R^N_+\right)^{\frac{1}{p}}$,
		\item $\abs{E}\leq C_2 cap_{1,p,N}\left( E\right)^\frac{1}{p}$,
		\item $ cap_{1-\frac{1}{p},p,N-1}\left( E\cap \pd\R^N_+\right)\leq C_3 cap_{1,p,N}\left( E\right)$,
		\item $\abs{E\cap\left\lbrace x_N=t\right\rbrace } \leq C_4cap_{1,p,N}\left( E \right)^{\frac{1}{p}}$.
	\end{enumerate} 
\end{lem}
\begin{proof}
	The first two inequalities follow from (the proof of) Proposition 2.6.1 of \cite{AdamsHedberg}, while the third is just Proposition \ref{trace&extensionE}. The last one follows from $(1)$, $(3)$, and the translation invariance of both the Lebesgue measure and capacities.
\end{proof}

Now we obtain estimates similar to \eqref{estimate2_DMMOP99} and \eqref{estimate4_DMMOP99} but on hyperplanes. We note explicitly that any $cap_{1,p,N}-$ quasi-continuous function on $\R^N$ has a well defined $cap_{1-\frac{1}{p},p,N-1}-$ quasi-continuous trace in any hyperplane $\R^{N-1}\times\left\lbrace t\right\rbrace $, $t\in \R$ (see the above lemma and Remark \ref{remark_quasincontinous_trace}). 
\begin{lem}\label{L_levelset_Trace}
	Let $f$ be $cap_{1,p,N}-$ quasi-continuous in $\R^N$ and such that $T_k(f) \in W_0^{1,p}\left( B_M\right) $ satisfies 
	\[ \frac{1}{k}\int_{\left\lbrace\abs{f}<k\right\rbrace \cap B_M}\abs{\nabla T_k(f)}^p dx\leq C_1 . \]
	If $1<p<N$ then there exists a constant $C(N,p,B_M)$ such that for any $t\in\R$
	\[\abs{\left\lbrace x \in B_M \cap \left\lbrace x_N=t\right\rbrace  \ : \ \abs{f}>k\right\rbrace }\leq C(N,p,B_M)C_1^{\frac{N-1}{N-p}}k^{\frac{(N-1)(1-p)}{N-p}} . \]
	If $p=N$ then there exists constants $C(N,B_M)$ and $c(N)>0$ such that for any $t\in\R$
	\[\abs{\left\lbrace x \in B_M \cap \left\lbrace x_N=t\right\rbrace  \ : \ \abs{f}>k\right\rbrace }\leq C(N,B_M)e^{-c(N) k (C_1)^{\frac{1}{1-N}}} . \]
\end{lem}
\begin{proof}
	Suppose $1<p<N$. By Sobolev's embedding (see \cite{Triebel06}), trace inequality, and Poincare's inequality, we have for $q=\frac{p(N-1)}{N-p}$ that
	\begin{align*}
	\norm{T_k(f)}_{L^q\left( B_M \cap \left\lbrace x_N=t\right\rbrace \right)}&\leq C(N,p)\norm{T_k(f)}_{F^{p,p}_{1-\frac{1}{p}}\left( B_M \cap \left\lbrace x_N=t\right\rbrace \right)}\\ 
	&\leq C(N,p)\norm{T_k(f)}_{W^{1,p}\left( B_M\right) } \\
	&\leq C(N,p,B_M)(kC_1)^{\frac{1}{p}} .
	\end{align*}
	Since $\left\lbrace\abs{f}> k\right\rbrace =\left\lbrace\abs{T_k(f)}\geq k\right\rbrace $ we conclude 
	\[
	\abs{\left\lbrace x \in B_M \cap \left\lbrace x_N=t\right\rbrace  \ : \ \abs{f}>k\right\rbrace }\leq \left( \frac{\norm{T_k(f)}_q}{k}\right) ^q\leq C(N,p,B_M) C_1^{\frac{q}{p}}k^{\left( \frac{1}{p}-1\right) q}
	\]
	which finishes the proof for the case $p<N$. When $p=N$, the results in \cite{Cianchi08} show that 
	\[
	\int_{B_M\cap \left\lbrace x_N=t\right\rbrace }e^{c_1\left( \frac{\abs{T_k(f)}}{\norm{\nabla T_k(f)}_{L^N\left( B_M\right) }}\right) ^{\frac{N}{N-1}}} dx' \leq c_2(N,B_M)
	\]
	for some constants $c_1(N)$ and $c_2(N,B_M)$. Since $\abs{T_k(f)}=k$ whenever $\abs{f}>k$ and $\norm{\nabla T_k(f)}_N\leq (kC_1)^\frac{1}{N}$ we conclude
	\begin{multline*}
	\abs{\left\lbrace x \in B_M \cap \left\lbrace x_N=t\right\rbrace  \ : \ \abs{f}>k\right\rbrace } e^{c_1 k C_1^{\frac{1}{1-N}}} \leq \\ \int_{\left\lbrace x \in B_M \cap \left\lbrace x_N=t\right\rbrace  \ : \ \abs{f}>k\right\rbrace }e^{c_1\left( \frac{\abs{T_k(f)}}{\norm{\nabla T_k(f)}_{L^N\left( B_M\right) }}\right) ^{\frac{N}{N-1}}} dx' \leq c_2 
	\end{multline*}
	which gives the desired bound. 
\end{proof}
\begin{remark}\label{improvement}
	Let us note that, in exactly the same way, one can prove that for any such function $f$ there holds
	\[\abs{\left\lbrace x \in B_M : \ \abs{f}>k\right\rbrace }\leq C(N,B_M)e^{-c(N) k C_1^{\frac{1}{1-N}}} \]
	when $p=N$. The fact that $c(N)$ does not depend on $B_M$ will be important to us when proving Theorem \ref{existence_subcritical_p=N}.
\end{remark}

Next we prove a lemma that will allow us to obtain solutions to \eqref{eq_intermediate_step} from solutions to \eqref{renormalized_solution_B-VHV14} under very general conditions. 
\begin{lem}\label{general_passage_to_trace_nonlinearity}
	Fix $m>0$. Suppose $g$ satisfies part $(1)$ of Assumption \ref{assumption_subcritical}, let $\tilde{g}$  be defined as in part $(2)$ of Assumption \ref{assumption_subcritical}, and let $g_n$ be defined by \eqref{def_g_n}. Let $u_n\goesto u$ $a.e.$ in $B_m$, where $u, u_n$ are $cap_{1,p,N}-$ quasi-continuous in $\R^N$. Assume also that $T_k(u_n)\goesto T_k(u)$ weakly in $W_0^{1,p}\left( B_m\right) $ for any $k>0$. 
	Define 
	\[
	\d_t(n,h)=\int_{B_m\cap \left\lbrace\abs{u_n}\geq h\right\rbrace \cap \left\lbrace x_N=t\right\rbrace }\tilde{g}(\abs{u_n})(x',t)dx' \]
	and
	\[
	\d(h)=\int_{B_m\cap \left\lbrace\abs{u}\geq h\right\rbrace \cap \left\lbrace x_N=0\right\rbrace }\tilde{g}(\abs{u})(x',0)dx' .
	\] 
	If $\d_t(n,h)\goesto 0$ and $\d(h)\goesto 0$ as $h\goesto \infty$, uniformly in $t$ and $n$, then   
	\[
	\lim_{n\goesto\infty} \int_{B_m} \phi_{n} g_{n}(x,u_{n})dx= \int_{B_m\cap\pd\R^N_+} \phi g(u)dx'
	\]
	whenever $\left\lbrace\phi_n\right\rbrace _n$ is a bounded subset of $L^\infty\left( B_m\right)$ such that $\phi_n\goesto \phi$ both $a.e.$ in $B_m$ and weakly in $W_0^{1,p}\left( B_m\right)$.
	Here $\phi$ is identified with its $cap_{1,p,N}-$ quasi-continuous representative in $\R^N$.
\end{lem}
\begin{proof}
	We first note that $\phi\in L^\infty\left( B_m\right)$.	Now, for any $n$ and $k$ we write $ E_n^k =\left\lbrace\abs{u_n} < k\right\rbrace $ and $E^k= \left\lbrace\abs{u} <k\right\rbrace $. We note that 
	\[
	\int_{B_m} \abs{\phi_n }\abs{g_n(x,u_n) - g_n(x,T_k(u_n))}dx = \int_{\left( E_n^k\right)^c} \abs{\phi_n }\zeta_n(x_N)\abs{g(u_n) - g(T_k(u_n))}dx
	\]
	and since $\sup_{[-k,k]}\abs{g(s)}=\tilde{g}(k)$ we estimate
	\[
	\int_{\left(  E_n^k\right)^c} \abs{\phi_n }\zeta_n(x_N)\abs{g(u_n) - g(T_k(u_n))}dx \leq 2\norm{\phi_n}_\infty\int_{\left( E_n^k\right)^c}\zeta_n(x_N)\tilde{g}(\abs{u_n})dx
	\]
	where we have used that $\abs{T_k(u_n)}=k\leq \abs{u_n}$ in $\left( E_n^k\right)^c$. For all $n$ we can estimate 
	\begin{multline*}
	\int_{\left( E_n^k\right)^c}\zeta_n(x_N)\tilde{g}(\abs{u_n})dx\leq \int_\R \zeta_n(t)\left[ \int_{(E_n^k)^c\cap \left\lbrace x_N=t\right\rbrace }\tilde{g}(\abs{u_n})(x',t)dx'\right] dt  \\ =\int_\R\zeta_n(t)\d_t(n,k) dt \leq \norm{\zeta_n}_{1}\norm{\d_t(n,k)}_\infty ,
	\end{multline*}
	and since $\d_t(n,k)\goesto 0$ uniformly we conclude
	\[
		\int_{\left( E_n^k\right)^c}\zeta_n(x_N)\tilde{g}(\abs{u_n})dx\goesto 0 
	\]
	as $k\goesto \infty$ uniformly in $n$.
	In a similar way we can write $\Gamma = B_m\cap \pd\R^N_+$ and estimate
	\[
	\int_{\Gamma}\abs{\phi} \abs{g(u)-g(T_k(u))} dx' = \int_{\Gamma \cap \left( E^k\right)^c } \abs{\phi} \abs{g(u)-g(T_k(u))} \leq
	2\norm{\phi}_\infty \d(k) \goesto 0
	\]
	as $k\goesto \infty$. Thus, collecting the above estimates we have
	\begin{multline}\label{control_tail}
	\abs{\int_{B_m}g_n(x,u_n)\phi_n dx - \int_{\Gamma}g(u)\phi dx'} \\ \leq w_n(k)+ w(k)+ \abs{\int_{B_m}g_n(x,T_k(u_n))\phi_n dx - \int_{\Gamma}g(T_k(u))\phi dx'}
	\end{multline}
	for some functions $w(k)$ and $w_n(k)$ such that $w(k)\goesto 0$ and $w_n(k)\goesto 0$ as $k\goesto \infty$ uniformly in $n$. Note that we have used that $\norm{\phi_n}_\infty$ are uniformly bounded.  
	
	Fix now any $\e>0$ and let $g_0\in C^1\left(\R\right) $ be such that $\sup_{s\in [-k,k]}\abs{g_0(s)-g(s)}\leq \e$. Then 
	\[
	\abs{\int_{B_m}\phi_n\zeta_n(x_N)g(T_k(u_n)) dx- \int_{B_m} \phi_n\zeta_n(x_N) g_0(T_k(u_n))dx}\leq \norm{\phi_n}_\infty C(m,N)\e
	\]
	and
	\[
	\abs{\int_{B_m}\phi\zeta_n(x_N)g(T_k(u)) dx- \int_{B_m}\phi \zeta_n(x_N) g_0(T_k(u))dx}\leq \norm{\phi}_\infty C(m,N)\e .
	\]
	On the other hand, since $g_0\in C^1\left(\R\right)$ has bounded derivative in $[-k,k]$ we see that $g(T_k(u))$ and $g(T_k(u_n))$ belong to $W_0^{1,p}\left( B_m\right)\cap L^\infty \left( B_m\right)$. It is easy to show, by using density of $C^\infty\left( B_m\right)$ in $W^{1,p}\left( B_m\right)\cap L^{p'} \left( B_m\right)$, that
	\[
	\int_{B_{m}}\Psi_1\partial_N \Psi_2 dx= - \int_{B_{m}} \Psi_2\partial_N\Psi_1 dx
	\] 
	for any pair of functions $\Psi_1 \in W_0^{1,p}\left( B_m\right)\cap L^\infty \left( B_m\right)$ and $\Psi_2 \in W^{1,p}\left( B_m\right)\cap L^\infty \left( B_m\right)$. We also observe that 
	\[
	\zeta_n(t)=\frac{1}{\pi}\partial_t \arctan(nt) 
	\]  
	and let 
	\[
	\tau_n(x):=\frac{1}{\pi}\arctan(nx_N).
	\]
	Then, using that $\zeta_n$ is smooth, we can write
	\[
	\int_{B_m}\zeta_n(x_N)\left( \phi_n g_0(T_k(u_n))dx - \phi g_0(T_k(u))\right)  dx= -(A) - (B) ,
	\]
	where
	\[
	(A) = \int_{B_{m}} \tau_n\left[ (\partial_N\phi_n) g_0(T_k(u_n)) -(\partial_N \phi) g_0(T_k(u))\right]dx 
	\]
	and
	\[
	(B)= \int_{B_{m}} \tau_n \left[ \phi_ng_0'(T_k(u_n))\partial_N T_k(u_n) - \phi g_0'(T_k(u))\partial_N T_k(u)\right] dx.
	\]
	Note that $\tau_n (t)\goesto \frac{1}{2}\left( \frac{t}{\abs{t}}\right) =: \tau(t)$ and $g_0(T_k(u_n))\goesto g_0(T_k(u))$ strongly in $L^r\left( B_m\right)$ for any $1\leq r <\infty$ since $\tau_n$, $g_0(T_k(u_n))$, $\tau$, and $g_0(T_k(u))$ are uniformly bounded in $L^\infty\left( B_m\right) $. Similarly $\tau_ng_0(T_k(u_n))\goesto \tau g_0(T_k(u))$ strongly in $L^{p'}\left( B_m\right) $. Thus, since $(\partial_N \phi) g_0(T_k(u)) \in L^p\left( B_m\right) $ and $\partial_N\phi_n\goesto \partial_N\phi $ weakly in $L^p\left( B_m\right) $ we conclude
	\[
	(A) \goesto 0
	\]
	as $n\goesto\infty$ for any fixed $k>0$. 
	Recall that $\phi_n\goesto\phi$ $a.e.$ in $B_m$. Since $\phi_n$, $\phi$, $g_0'(T_k(u_n))$, and $g_0'(T_k(u))$ are uniformly bounded in $L^\infty\left( B_m\right) $, we conclude as above that $\tau_n\phi_ng_0'(T_k(u_n))\goesto \tau\phi g_0'(T_k(u))$ strongly in $L^{p'}\left( B_m\right) $. Since $\partial_NT_k(u_n)\goesto \partial_NT_k(u)$ weakly in $L^{p}\left( B_m\right) $ we conclude
	\[
	(B) \goesto 0
	\]
	as $n\goesto\infty$ for any fixed $k>0$. Collecting the above we can rewrite \eqref{control_tail} as
	\begin{multline}\label{control_tail_2}
	\abs{\int_{B_m}g_n(x,u_n)\phi_n dx - \int_{\Gamma}g(u)\phi dx'}  \leq w_n(k)+ w(k)+ w_{k,\e}(n) \\ + \e C(m,N)\left( \norm{\phi_n}_\infty+\norm{\phi}_\infty\right)  +
	 \abs{\int_{B_m}g_n(x,T_k(u))\phi dx - \int_{\Gamma}g(T_k(u))\phi dx'}
	\end{multline}
	for any $\e>0$, where $w_{k,\e}(n)$ is a function such that $w_{k,\e}(n)\goesto 0$ as $n\goesto\infty$ for any $k>0$ and $\e>0$ fixed.
	
	To continue we observe that since both $\phi$ and $u$ are $cap_{1,p,N}-$ quasi-continuous in $\R^N$, given $\e>0$ we can find a closed set $\O_0$ such that $u,\phi \in C\left( \O_0\right)$ and $cap_{1,p,N}\left( \O_0^c\right) <\e$. Then, $g(T_k(u))$ and $\phi$ are uniformly continuous and bounded in $\O=\O_0\cap \overline{B_{m}}$, and we can find $t_0$ small enough so that 
	\[
	\abs{\phi(x',x_N) g(T_k(u))(x',x_N)-\phi(x',0) g(T_k(u))(x',0)}<\e 
	\]
	for any $(x',x_N)\in \O\cap \left\lbrace\abs{x_N}\leq t_0\right\rbrace $.
	We can also assume $t_0$ is such that 
	\[
	\abs{\left( \left( \Gamma \times \R \right) \setminus \overline{B_{m}}\right)  \cap \left\lbrace x_N=t\right\rbrace }<\e 
	\]
	for all $\abs{t}\leq t_0$. Note that $\norm{\chi_{\left\lbrace\abs{t}>t_0\right\rbrace }\zeta_n}_{L^1}\goesto 0$ as $n\goesto\infty$ for any $t_0>0$. Then, we write
	\begin{multline*}
	\abs{\int_{B_m}\phi g_n(x,T_k(u))dx-\int_{\Gamma}\phi g(T_k(u))dx'}= \\ \abs{\int_{B_m}\phi g_n(x,T_k(u))dx-\int_{\Gamma\times \R}\zeta_n(x_N)\left( \phi g(T_k(u))\right) (x',0)dx} \\ \leq \int_{\Gamma\times \R}\zeta_n(x_N) \abs{\phi(x',x_N) g(T_k(u))(x',x_N)-\phi(x',0) g(T_k(u))(x',0)}dx,
	\end{multline*}
	and estimate
	\begin{multline*}
	\int_{\Gamma\times \R \cap \left( \O\cap \left\lbrace\abs{x_N}\leq t_0\right\rbrace \right)}\zeta_n(x_N) \abs{\phi(x',x_N) g(T_k(u))(x',x_N)-\phi(x',0) g(T_k(u))(x',0)}dx\leq\\ \e C(m,N)
	\end{multline*}
	and 
	\begin{multline*}
	\int_{\Gamma\times \R\cap\left( \left\lbrace\abs{x_N}> t_0\right\rbrace \right) }\zeta_n(x_N) \abs{\phi(x',x_N) g(T_k(u))(x',x_N)-\phi(x',0) g(T_k(u))(x',0)}dx\leq \\ \norm{\phi}_\infty \tilde{g}(k)\norm{\chi_{\left\lbrace\abs{x_N}>t_0\right\rbrace }\zeta_n}_{L^1}.
	\end{multline*}
	In view of Lemma \ref{L_trace_characteristic} we also have
	\begin{multline*}
	\int_{\Gamma\times \R\cap \left( \O^c\cap \left\lbrace\abs{x_N}\leq t_0\right\rbrace \right) }\zeta_n(x_N) \abs{\phi(x',x_N) g(T_k(u))(x',x_N)-\phi(x',0) g(T_k(u))(x',0)}dx\\ \leq \norm{\phi}_\infty \tilde{g}(k) \int_{-t_0}^{t_0} \zeta(t)\left[  \abs{\O_0^c\cap \left\lbrace x_N=t\right\rbrace } + \abs{\left( \left( \Gamma \times \R \right) \setminus \overline{B_{m}}\right)  \cap \left\lbrace x_N=t\right\rbrace }\right] dt \\
	\leq \left( \e^\frac{1}{p}C(m,N,p) + \e\right) \norm{\phi}_\infty \tilde{g}(k) .
	\end{multline*}
	Note that we have used that $\norm{\phi(x',0)}_\infty\leq \norm{\phi}_\infty$ (see Proposition \ref{L^infty_mu_0}).
	Considering the above estimates we obtain from \eqref{control_tail_2} that
	\begin{multline}\label{control_tail_3}
	\abs{\int_{B_m}g_n(x,u_n)\phi_n dx - \int_{\Gamma}g(u)\phi dx'}  \leq w_n(k)+ w(k)+ \tilde{w}_{k,\e}(n) \\ +  \e^\frac{1}{p} C\left( m,N,p, \tilde{g}(k),\norm{\phi}_\infty, \norm{\phi_n}_\infty\right) 
	\end{multline}
	for any $0<\e<1$, where $\tilde{w}_{k,\e}(n)$ is a function such that $\tilde{w}_{k,\e}(n)\goesto 0$ as $n\goesto\infty$ for any fixed $k>0$ and $\e>0$. Thus, taking $k$ large enough and then choosing $\e$ small enough, we see that for any $\d>0$ 
	\[
		\abs{\int_{B_m}g_n(x,u_n)\phi_n dx - \int_{\Gamma}g(u)\phi dx'}  \leq \d
	\]
	for all $n$ large enough. Hence, the result follows.
\end{proof}

When $g_n(x,u_n)\goesto g(u)$ in $L^1\left( B_m\right) $, renormalized solutions to $-\lap_p u_n + g_n(x,u_n)=\bar{\mu}$ converge to a solution of $-\lap_p u + g(u)=\bar{\mu}$ by the stability result of \cite{DMOP99} or \cite{M05}. Since in our case the convergence is not quite in $L^1$ we cannot simply use the same result. To pass to the limit, we use the following stability result.

\begin{lem}\label{general_stability_in_balls}
	Fix $m>0$. Let $u_n$ be renormalized solutions of
	\[
	\begin{cases}
	-\lap_pu_n + g_n(x,u_n)=\bar{\mu} & \mbox{ in } B_m \\
	u_n=0 & \mbox{ on } \pd B_m
	\end{cases}
	\]
	where $\bar{\mu}\in \mathfrak{M}_b\left( B_m\right)$, $g_n(x,s)$ is defined by \eqref{def_g_n}, and $g$ satisfies Assumption \ref{assumption_subcritical}. 
	Suppose $u_n\goesto u$, $\nabla T_k(u_n)\goesto \nabla T_k(u)$, and $\nabla u_n \goesto \nabla u$ $a.e.$ in $B_m$ where $u$ satisfies condition $(1)$ and $(2)$ of Definition \ref{definition_DMOP99} and is $cap_{1,p,N}-q.e.$ finite. Assume also that
	\[
	\abs{\nabla u_n}^{p-2}\nabla u_n\goesto \abs{\nabla u}^{p-2}\nabla u \mbox{ strongly in } \left( L^q\left( B_m\right) \right) ^N \mbox{ for any } 1\leq q<\frac{N}{N-1},
	\]
	\[
	T_k(u_n)\goesto T_k(u) \mbox{ weakly in } W_0^{1,p}\left( B_m\right) ,
	\]
	and
	\be\label{condition_stablity_balls}
	\lim_{n\goesto\infty} \int_{B_m} \phi_n g_n(x,u_n)dx= \int_{B_m\cap\pd\R^N_+} \phi g(u)dx' 
	\ee
	whenever $\left\lbrace\phi_n\right\rbrace _n$ is a bounded subset of $L^\infty\left( B_m\right)$ such that $\phi_n\goesto \phi$ both $a.e.$ in $B_m$ and weakly in $W_0^{1,p}\left( B_m\right)$. If $g(u)\in L^1\left( B_m\cap\pd\R^N_+\right)$ then $u$ is a renormalized solution of 
	\[
	\begin{cases}
	-\lap_pu +  g(u)\mathcal{H}=\bar{\mu} & \mbox{ in } B_m \\
	u=0 & \mbox{ on } \pd B_m .
	\end{cases}
	\]
	Moreover, $T_k(u_n)\goesto T_k(u)$ strongly in $W_0^{1,p}\left( B_m\right)$ for any $k>0$.
\end{lem}
\begin{proof}
	 Note that since properties $(1)$ and $(2)$ of Definition \ref{definition_DMOP99} hold, to prove that $u$ is a renormalized solution of the above equation it is enough to show that $(3)$ also holds. Since we essentially repeat, with a few modifications and simplifications, the argument used to pass to the limit in the proof of Theorem 4.1 in \cite{M05} we only point out the main ideas (see also the proof of Lemma \ref{general_existence_global_solution}). 
	 
	 Before we begin, note that by choosing $\phi_n=\phi$ we have
	 \[
	 \lim_{n\goesto\infty} \int_{B_m} \phi g_n(x,u_n)dx= \int_{B_m\cap\pd\R^N_+} \phi g(u)dx'
	 \]
	 for all $\phi \in W_0^{1,p}\left( B_m\right) \cap L^\infty\left( B_m\right) $.
	 
	 Now, by Theorem \ref{BGO96_thm} $\bar{\mu}_0= f -\dv h$ in $\D'\left( B_m\right) $ for some $f\in L^1\left( B_m\right) $ and $h\in \left( L^{p'}\left( B_m\right) \right) ^N$. Using Lemma 3.1 of \cite{M05} we obtain a set $U\subset (0,\infty)$ with $\abs{U^c}=0$ such that each $u_n$ satisfies the following: for every $k\in U$ there exists two measures $\a_{n,k}^+$, $\a_{n,k}^- \in \mathfrak{M}_0\left(B_m\right) $ supported in $\left\lbrace u_n=k\right\rbrace $ and $\left\lbrace u_n=-k\right\rbrace $ respectively, such that up to a subsequence (possibly depending on $n$) $\a_{n,k}^\pm \goesto \bar{\mu} _s^\pm $, as $k \in U$ goes to infinity, in the weak-$\ast$ topology of $\mathfrak{M}_b\left( B_m\right)$, and the truncations $T_k(u_n)$ satisfy
	 \begin{multline}
	 \int_{\left\lbrace\abs{u_n}<k\right\rbrace }\left( \abs{\nabla T_k(u_n)}^{p-2}\nabla T_k(u_n) - h\right) \cdot \nabla v dx= \\
	 \int_{\left\lbrace\abs{u_n}<k\right\rbrace }vfdx - \int_{\left\lbrace\abs{u_n}<k\right\rbrace }vg_n(x,u_n)dx + \int_{\left\lbrace u_n=k\right\rbrace } v d\a_{n,k}^+ - \int_{\left\lbrace u_n=-k\right\rbrace } v d\a_{n,k}^- 
	 \end{multline}
	 for every $v\in W^{1,p}_0\left( B_m\right) \cap L^\infty\left( B_m\right) $. 
	 
	 We consider the convergence of the above terms as $n\goesto \infty$. Let 
	 \[E= \left\lbrace k\in \R_+ : \abs{\left\lbrace x\in B_m : \abs{u}=k\right\rbrace }>0\right\rbrace  \] 
	 and write $F=\left( E\right) ^c$. Since $\abs{B_m}<\infty$, $E$ is countable and thus of null measure. Note that $\chi_{\left\lbrace\abs{u_n}<k\right\rbrace }\goesto  \chi_{\left\lbrace\abs{u}<k\right\rbrace }$ $a.e.$ in $B_m$ except possibly in $\left\lbrace x\in B_m : \abs{u}=k\right\rbrace $, thus $\chi_{\left\lbrace\abs{u_n}<k\right\rbrace }\goesto \chi_{\left\lbrace\abs{u}<k\right\rbrace }$ $a.e.$ in $B_m$ and weakly-$\ast$ in $L^\infty\left( B_m\right) $ for all $k \in F$. 
	 Then, as in Lemma \ref{general_existence_global_solution}, we can show that 
	 \begin{multline*}
	 \int_{\left\lbrace\abs{u_n}<k\right\rbrace } \abs{\nabla T_k(u_n)}^{p-2}\nabla T_k(u_n) \cdot \nabla \phi dx \goesto \int_{\left\lbrace\abs{u}<k\right\rbrace }\abs{\nabla T_k(u)}^{p-2}\nabla T_k(u) \cdot \nabla\phi dx
	 \end{multline*}
	 and
	 \[
	 \int_{\left\lbrace\abs{u_n}<k\right\rbrace } h \cdot \nabla \phi dx + \int_{\left\lbrace\abs{u_n}<k\right\rbrace }\phi f dx \goesto \int_{\left\lbrace\abs{u}<k\right\rbrace } h \cdot \nabla \phi dx + \int_{\left\lbrace\abs{u}<k\right\rbrace }\phi f dx
	 \]
	 for any $\phi\in C_0^\infty\left( B_m\right) $ and $k \in F$.
	 
	 By Proposition \ref{Veron16_bound_on_g} we see that for each $k>0$ 
	 \[
	 \int_{\left\lbrace\abs{u_n}<k\right\rbrace }\abs{g_n(x,u_n)}dx\leq \abs{\bar{\mu}}\left( B_m\right) .
	 \]
	 Hence, for each $k$ there exists a measure $\tau_k\in \mathfrak{M}_b\left( B_m\right) $ such that, up to a subsequence possibly depending on $k$, 
	 \[
	 \int_{\left\lbrace\abs{u_n}<k\right\rbrace }\phi g_n(x,u_n)dx\goesto \int_{B_{m}}\phi d\tau_k
	 \]
	 for any $\phi\in C_0^\infty\left( B_m\right) $.
	 Similarly, following the argument in the proof of Theorem 4.1 in \cite{M05}, we can conclude that for every $n$
	 \[
	 \abs{\a_{n,k}^+}\left( B_m\right) + \abs{\a_{n,k}^-}\left( B_m\right) \leq C(N,p,\bar{\mu}) 
	 \]
	 for any $k>0$ in some subset $V$ with $\abs{V^c}=0$. It follows that for each $k\in V$ there exists nonnegative measures $\l_k^+$ and $\l_k^-$ such that, up to a subsequence, 
	 \[
	 \a_{n,k}^\pm\goesto \l_k^\pm \mbox{ weakly-$\ast$ in } \mathfrak{M}_b\left( B_m\right). 
	 \]
	 In particular, given $\phi\in C_0^\infty\left( B_m\right) $ we can pass to a subsequence to conclude
	 \[
	 \int_{B_m}\phi d\a_{n,k}^+ - \int_{B_m}\phi d \a_{n,k}^- \goesto \int_{B_m}\phi d\l_k^+ - \int_{B_m}\phi d\l_k^-.
	 \]
	 Then, collecting all the above, we get that for any $\phi\in C_0^\infty\left( B_m\right) $ and $k\in K:=F\cap U \cap V$
	 \begin{multline*}
	 \int_{\left\lbrace\abs{u}<k\right\rbrace }\left( \abs{\nabla T_k(u)}^{p-2}\nabla T_k(u)-h\right) \cdot \nabla \phi dx = \\ \int_{\left\lbrace\abs{u}<k\right\rbrace }\phi f dx - \int_{B_m}\phi \tau_k  + \int_{B_m}\phi d \l_k^+ - \int_{B_m}\phi d\l_k^- .
	 \end{multline*}
	 From the above we see that $-\tau_k+ \l_k^+-\l_k^-$ belongs to $W^{-1,p'}\left( B_m\right) + L^1\left( B_m\right) $ and so
	 \begin{multline}\label{L_Test1}
	 \int_{\left\lbrace\abs{u}<k\right\rbrace }\left( \abs{\nabla T_k(u)}^{p-2}\nabla T_k(u)-h\right) \cdot \nabla \phi dx = \\ \int_{\left\lbrace\abs{u}<k\right\rbrace }\phi f dx + \int_{B_m}\phi d\left(  -\tau_k +\l_k^+-\l_k^-\right)  
	 \end{multline}
	 for any $\phi \in W^{1,p}_0\left( B_m\right) \cap L^\infty\left( B_m\right) $ and $k\in K$.
	 
	 Since $u$ is $cap_{1,p,N}-q.e.$ finite, it follows that $\left\lbrace x\in B_m : \abs{u}>k\right\rbrace $ is quasi-open, and we can find a sequence $\w_n \in W^{1,p}\left( \R^N\right) $ such that $0\leq \w_n \leq \chi_{\left\lbrace\abs{u}>k\right\rbrace }$ and $\w_n \uparrow \chi_{\left\lbrace\abs{u}>k\right\rbrace }$ $cap_{1,p,N}-q.e.$ in $\R^N$.
	 For any $\phi \in C_0^\infty\left( B_m\right) $ we can put $\phi\w_n$ as test function in \eqref{L_Test1} and conclude
	 \[
	 \int_{B_m}\phi\w_n d \left(-\tau_k+ \l_k^+-\l_k^-\right) =0
	 \]
	 for any $k\in K$. Taking $n\goesto\infty$ we conclude that for any $k\in K$
	 \[
	 \left. \left( -\tau_k+ \l_k^+-\l_k^-\right) \right|  _{\left\lbrace\abs{u}>k\right\rbrace } = 0 .
	 \]
	 
	 Letting now $\w_n$ be a sequence in $W^{1,p}\left( \R^N\right) $ such that $0\leq \w_n \leq \chi_{\left\lbrace\abs{u}<k\right\rbrace }$ and $\w_n \uparrow \chi_{\left\lbrace\abs{u}<k\right\rbrace }$ $cap_{1,p,N}-q.e.$, we put $\phi \w_n$ as test function in \eqref{L_Test1} with both $k$ and $h>k$ in $K$ and take $n\goesto\infty$ to conclude, as in the proof of Theorem 4.1 in \cite{M05}, that there exists a measure $\nu_0 \in \mathfrak{M}_0\left( B_m\right) $ such that 
	 \[
	 \left. \nu_0\right| _{\left\lbrace\abs{u}<k\right\rbrace }= \left. \left( -\tau_h + \l_h^+-\l_h^-\right) \right| _{\left\lbrace\abs{u}<k\right\rbrace }
	 \]
	 for any $h\geq k$ in $K$. Hence, defining
	 \[
	 \nu_k^+=\left. \left(-\tau_k + \l_k^+-\l_k^-\right) \right| _{\left\lbrace u=k\right\rbrace } \ , \ \nu_k^-=-\left. \left(-\tau_k + \l_k^+-\l_k^-\right) \right| _{\left\lbrace u=-k\right\rbrace }
	 \]
	 we can rewrite \eqref{L_Test1} as
	 \begin{multline}\label{L_Test2}
	 \int_{\left\lbrace\abs{u}<k\right\rbrace }\left( \abs{\nabla T_k(u)}^{p-2}\nabla T_k(u)-h\right) \cdot \nabla \phi dx = \\
	 \int_{\left\lbrace\abs{u}<k\right\rbrace }\phi f dx + \int_{\left\lbrace\abs{u}<k\right\rbrace }\phi d \nu_0 + \int_{\left\lbrace u=k\right\rbrace } \phi d \nu_k^+ - \int_{\left\lbrace u=-k\right\rbrace }\phi d \nu_k^-
	 \end{multline}
	 for any $\phi \in W^{1,p}_0\left( B_m\right) \cap L^\infty\left( B_m\right) $ and $k\in K$.

	 Proceeding as in the proof of Lemma \ref{general_existence_global_solution}, we can use \eqref{condition_stablity_balls} to show that $\nu_0\equiv -g(u)\mathcal{H}$, and that, up to a sequence $k_j \in K$ going to infinity, $\nu_{k_j}^\pm \goesto \bar{\mu}_s^\pm$ weakly-$\ast$ in $\mathfrak{M}_b\left( B_m\right) $.	 
	 In particular, we can rewrite \eqref{L_Test2} as 
	 \begin{multline}\label{L_Test3}
	 \int_{\left\lbrace\abs{u}<k\right\rbrace }\left( \abs{\nabla T_k(u)}^{p-2}\nabla T_k(u)-h\right) \cdot \nabla \phi dx = \\ \int_{\left\lbrace\abs{u}<k\right\rbrace }\phi fdx - \int_{\left\lbrace\abs{u}<k\right\rbrace \cap\pd\R^N_+}\phi g(u)dx'  + \int_{\left\lbrace u=k\right\rbrace } \phi d \nu_k^+ - \int_{\left\lbrace u=-k\right\rbrace }\phi d \nu_k^-
	 \end{multline}
	 for any $\phi\in W^{1,p}_0\left( B_m\right) \cap L^\infty\left( B_m\right)$ and $k\in K$. At this point, the fact that $u$ is a renormalized solution of the desired equation follows exactly as in the proof of Theorem 4.1 of \cite{M05}. 
	 
	 To finish the proof of our lemma, we observe that on one hand we have 
	 \[
	 \int_{B_m}\abs{\nabla T_k(u)}^{p}dx = \int_{B_m} T_k(u)d\bar{\mu}_0 - \int_{B_m\cap\pd\R^N_+} T_k(u)g(u)dx' + \int_{B_m} kd\bar{\mu}_s^+ + \int_{B_m} kd\bar{\mu}_s^-,
	 \]
	 while on the other 
	 \[
	 \int_{B_m}\abs{\nabla T_k(u_n)}^{p}dx = \int_{B_m} T_k(u_n)d\bar{\mu}_0 - \int_{B_m} T_k(u_n)g_n(x,u_n)dx + \int_{B_m} kd\bar{\mu}_s^+ + \int_{B_m} kd\bar{\mu}_s^-.
	 \]
	 Comparing the above identities we have 
	 \begin{multline*}
	 \int_{B_m}\abs{\nabla T_k(u)}^{p}dx -  \int_{B_m}\abs{\nabla T_k(u_n)}^{p}dx= \\ \int_{B_m} T_k(u)d\bar{\mu}_0- \int_{B_m} T_k(u_n)d\bar{\mu}_0 + \int_{B_m} T_k(u_n)g_n(x,u_n)dx - \int_{B_m\cap\pd\R^N_+} T_k(u)g(u)dx'.
	 \end{multline*}
	 Writing again $\bar{\mu}_0= f -\dv h$ for some $f\in L^1\left( B_m\right) $ and $h\in \left( L^{p'}\left( B_m\right) \right) ^N$, we use that $T_k(u_n)\goesto T_k(u)$ weakly in $W_0^{1,p}\left( B_m\right)$ and weakly-$\ast$ in $L^\infty\left( B_m\right)$ to obtain
	 \[
	 \int_{B_m} T_k(u)d\bar{\mu}_0\goesto \int_{B_m} T_k(u_n)d\bar{\mu}_0 
	 \]
	 as $n\goesto \infty$. Similarly, by \eqref{condition_stablity_balls}
	 \[
	 \int_{B_m} T_k(u_n)g_n(x,u_n)dx \goesto \int_{B_m\cap\pd\R^N_+}T_k(u)g(u)dx'
	 \]
	 as $n\goesto\infty$, and so  
	 \[
	  \norm{\nabla T_k(u_n)}_{L^p}\goesto \norm{\nabla T_k(u)}_{L^p}
	 \]
	 as $n\goesto \infty$. As in the proof of Lemma \ref{general_existence_global_solution}, we conclude from the above that $\abs{\nabla T_k(u_n)}^p\goesto \abs{\nabla T_k(u)}^p$ strongly in $L^1\left( B_{m}\right)$, and then, by Vitalli's Theorem, that $\nabla T_k(u_n)\goesto \nabla T_k(u)$ strongly in $\left( L^p\left( B_{m}\right)\right) ^N$. Hence the claim follows.
\end{proof}

The following lemma, similar to Lemma \ref{general_passage_to_trace_nonlinearity}, gives a useful sufficient condition under which \eqref{condition_general_existence_global_solution} holds. It will be used to obtain stability of local solutions throughout the sequel.

\begin{lem}\label{general_passage_to_trace_nonlinearity_R^N}
	Let $u_m$ and $u$ be $cap_{1,p,N}-$ quasi-continuous functions such that $u_m\goesto u$ $a.e.$ in $\R^N$ and $T_k(u_m)\goesto T_k(u)$ weakly in $W^{1,p}\left( B_M\right) $ for any fixed $k>0$ and $M\in \N$. Suppose $g$ satisfies part $(1)$ of Assumption \ref{assumption_subcritical} and let $\tilde{g}$ be defined as in part $(2)$ of Assumption \ref{assumption_subcritical}. For any fixed $M\in \N$ we define 
	\[
	\rho_m(h)=\int_{B_M\cap \left\lbrace\abs{u_m}\geq h\right\rbrace \cap \pd\R^N_+}\tilde{g}(\abs{u_m})(x')dx' \]
	and
	\[
	\rho(h)=\int_{B_M\cap \left\lbrace\abs{u}\geq h\right\rbrace \cap \pd\R^N_+}\tilde{g}(\abs{u})(x')dx' .
	\] 
	If $\rho_m(h)\goesto 0$ and $\rho(h)\goesto 0$ as $h\goesto \infty$, uniformly in $m$, then
	\[
	\lim_{m\goesto\infty} \int_{B_{M}\cap\pd\R^N_+}\phi_m g(u_m) dx'= \int_{B_{M}\cap\pd\R^N_+} \phi g(u) dx'
	\]
	for any sequence $\left\lbrace\phi_m\right\rbrace _m$ converging to $\phi$ both $a.e.$ in $B_M$ and weakly in $W_0^{1,p}\left( B_M\right) $ and such that $\phi_m$ is uniformly bounded in $L^\infty\left( B_M\right) $.
\end{lem}
\begin{proof}
	We follow very closely the ideas in the proof of Lemma \ref{general_passage_to_trace_nonlinearity} and so we omit some details. 
	
	As in Lemma \ref{general_passage_to_trace_nonlinearity}, we use the assumptions on $\rho_m$ and $\rho$ to obtain that
	\begin{multline*}
	\abs{\int_{\Gamma_M} \phi_m g(u_m)dx' - \int_{\Gamma_M} \phi g(u)dx'}\leq \w_m(k) + \\\abs{\int_{\Gamma_M} \phi_m g(T_k(u_m))dx' - \int_{\Gamma_M} \phi g(T_k(u))dx'}
	\end{multline*}
	for some $\w_m(k)$ such that $\w_m(k)\goesto 0$ as $k\goesto \infty$ uniformly on $m$, and where $\Gamma_M= B_M\cap\pd\R^N_+$. For any $n$ we let $\zeta_n$ be the functions defined in \eqref{def_g_n}. Hence, 
	\begin{multline*}
	\int_{\Gamma_M} \phi_m g(T_k(u_m))dx' - \int_{\Gamma_M} \phi g(T_k(u))dx'= \\ \int_{\Gamma_M\times \R} \zeta_n(x_N)\left( \phi_m g(T_k(u_m))-\phi g(T_k(u))\right)(x',0) dx.
	\end{multline*}
	To continue we observe that since $\phi_m$, $\phi$, $g(T_k(u_m))$ and $g(T_k(u))$ are $cap_{1,p,N}-$ quasi-continuous in $\R^N$, given $\e>0$ we can find a closed set $\O_0$ such that all of them belong to $C\left( \O_0\right)$ and $cap_{1,p,N}\left( \O_0^c\right) <\e$. Then, all of them are uniformly continuous in $\O=\O_0\cap \overline{B_{M}}$, and since they are also uniformly bounded, for any fixed $m$ and $\e>0$ we can find $t_0$ small enough so that 
	\[
	\abs{\phi_m(x',x_N) g(T_k(u_m))(x',x_N)-\phi_m(x',0) g(T_k(u_m))(x',0)}<\e 
	\]
	and
	\[
	\abs{\phi(x',x_N) g(T_k(u))(x',x_N)-\phi(x',0) g(T_k(u))(x',0)}<\e 
	\]
	for any $(x',x_N)\in \O\cap \left\lbrace\abs{x_N}\leq t_0\right\rbrace $. We also assume $t_0$ is such that 
	\[
	\abs{\left( \left( \Gamma_M \times \R \right) \setminus \overline{B_{M}}\right)  \cap \left\lbrace x_N=t\right\rbrace }<\e 
	\]
	for all $\abs{t}\leq t_0$. Then, by decomposing $\Gamma_M\times \R $ as
	\begin{multline*}
	\left[ \Gamma_M\times \R \cap \left( \O\cap \left\lbrace\abs{x_N}\leq t_0\right\rbrace \right)\right] \cup \left[ \Gamma_M\times \R\cap\left( \left\lbrace\abs{x_N}> t_0\right\rbrace \right)\right] \\ \cup \left[ \Gamma_M\times \R\cap \left( \O^c\cap \left\lbrace\abs{x_N}\leq t_0\right\rbrace \right)\right] ,
	\end{multline*}
	we see that 
	\begin{multline*}
	\abs{\int_{\Gamma_M\times \R} \zeta_n(x_N)\left( \phi_m g(T_k(u_m))-\phi g(T_k(u))\right)(x',0) dx}\leq  \tilde{w}_{k,m,\e}(n) \\  +  \e^\frac{1}{p} C\left( M,N,p, \tilde{g}(k),\norm{\phi}_\infty, \norm{\phi_m}_\infty\right)  \\ + 	\abs{\int_{B_M} \zeta_n(x_N)\left( \phi_m g(T_k(u_m))-\phi g(T_k(u))\right)(x',x_N) dx},
	\end{multline*}
	for any $\e>0$, where $\tilde{w}_{k,m,\e}(n)$ is a function such that $\tilde{w}_{k,m,\e}(n)\goesto 0$ as $n\goesto\infty$ for any fixed $\e>0$, $m\in \N$ and $k>0$. Since all the functions involved are uniformly bounded, we may approximate $g$ with a $g_0 \in C^1\left( \left[ -k,k\right] \right)$ such that $\sup_{[-k,k]}\abs{g-g_0}<\e$, and upon integrating by parts, obtain 
	\begin{multline*}
	\int_{B_M} \zeta_n(x_N)\left( \phi_m g_0(T_k(u_m))-\phi g_0(T_k(u))\right)(x',x_N)dx= \\ -\int_{B_M}\tau_n(x_N) \left( \left( \partial_N\phi_m\right)  g_0(T_k(u_m))-\left( \partial_N\phi\right)  g_0(T_k(u))\right)dx \\
	-\int_{B_M}\tau_n(x_N)\left( \phi_m \partial_N\left(  g_0(T_k(u_m))\right)  -\phi\partial_N \left( g_0(T_k(u))\right) \right)dx.
	\end{multline*}
	Let us consider the first integral in the right hand side of the above identity. Recall that $\tau_n(t)\goesto \tau(t)=\frac{1}{2}\frac{t}{\abs{t}}$ strongly in $L^s\left( \R\right) $ for any $1\leq s<\infty$. Since $\partial_N\phi_m$ is uniformly bounded in $L^p\left( B_M\right) $ and $T_k(u_m)$ is uniformly bounded in $L^\infty\left( B_M\right) $ we obtain that 
	\begin{multline*}
	\int_{B_{M}}\tau_n(x_N) \left( \left( \partial_N\phi_m\right)  g_0(T_k(u_m))-\left( \partial_N\phi\right)  g_0(T_k(u))\right)dx \\
	\goesto \int_{B_{M}}\tau(x_N) \left( \left( \partial_N\phi_m\right)  g_0(T_k(u_m))-\left( \partial_N\phi\right)  g_0(T_k(u))\right)dx
	\end{multline*}
	as $n\goesto\infty$ uniformly in $m$. On the other hand, as in Lemma \ref{general_passage_to_trace_nonlinearity}, 
	\[
	\int_{B_{M}}\tau(x_N) \left( \left( \partial_N\phi_m\right)  g_0(T_k(u_m))-\left( \partial_N\phi\right)  g_0(T_k(u))\right)dx \goesto 0
	\]
	as $m\goesto\infty$. A similar reasoning applies to the second integral. Hence, we may write
	\[
	\int_{B_M} \zeta_n(x_N)\left( \phi_m g_0(T_k(u_m))-\phi g_0(T_k(u))\right)(x',x_N)dx= w_{k,m,\e}(n) + w_{k,\e}(m)
	\]
	for some functions $w_{k,m,\e}(n)$ and $w_{k,\e}(m)$ such that $w_{m,k,\e}(n)\goesto 0$ as $n\goesto\infty$, uniformly on $m$, and $w_{k,\e}(m)\goesto 0$ as $m\goesto\infty$, for any fixed $k>0$ and $1>\e>0$. Thus, collecting all the above estimates, we conclude that
	\begin{multline*}
	\abs{\int_{\Gamma_M} \phi_m g(u_m)dx' - \int_{\Gamma_M} \phi g(u)dx'}\leq \w_m(k) +\tilde{w}_{k,m,\e}(n) \\ +  \e^\frac{1}{p} C\left( M,N,p, \tilde{g}(k),\norm{\phi}_\infty, \norm{\phi_m}_\infty\right) + w_{k,m,\e}(n) + w_{k,\e}(m).
	\end{multline*}
	Hence, we obtain
	\[
	\lim_{m\goesto\infty}\int_{\Gamma_M} \phi_m g(u_m)dx'= \int_{\Gamma_M} \phi g(u)dx'
	\]
	as desired.
\end{proof}

We are now ready to prove the existence of renormalized solutions to \eqref{equation} in the subcritical case. We treat the cases $1<p<N$ and $p=N$ separately for clarity of exposition.  

\begin{thm}\label{existence_subcritical_p<N}
	Let $1<p< N$ and $\mu\in \mathfrak{M}_b\left( \pd\R^N_+\right) $. Suppose $g(s)$ satisfies Assumption \ref{assumption_subcritical}. Then there exists a renormalized solution of
	\[
	\begin{cases}
	-\lap_pu=0 & \mbox{ in } \R^N_+ \\
	\abs{\nabla u}^{p-2}u_\n + g(u) = \mu & \mbox{ on } \pd\R^N_+ .
	\end{cases}
	\]
\end{thm}
\begin{proof}
	Let $g_n(x,s)$ be defined by \eqref{def_g_n} and fix $m\in \N$. By Theorem 5.1.2 of \cite{Veron16} for any $n \in \N$ there exists a renormalized solution of
	 \[
	 \begin{cases}
	 -\lap_pu_m^n + 2g_n(x,u_m^n)=2\mu_m & \mbox{ in } B_m \\
	 u_m^n=0 & \mbox{ on } \pd B_m ,
	 \end{cases}
	 \]
	where $\mu_m$ is the restriction of $\mu$ to $B_m$.
	By Proposition \ref{Veron16_bound_on_g} we have
	\be\label{estimate_g_n}
	\int_{B_m}\abs{g_n(x,u_m^n)} dx \leq \abs{\mu_m}\left( B_m\right)  .
	\ee
	By writing
	\[
	\tilde{\mu}_m^n = \mu_m-g_n(x,u_m^n)
	\]
	we see that $u_m^n$ is a renormalized solution to $-\lap_p u_m^n = 2\tilde{\mu}_m^n$ in $B_m$. Since 
	\[\abs{\tilde{\mu}_m^n}\left( B_m\right) \leq 2\abs{\mu}\left( \R^N\right) <\infty \]
	we can apply Theorem \ref{DMOP99_convergence_of_u_n_to_limit} to obtain that, passing to subsequences, $u_m^n\goesto u_m$ $a.e.$ in $B_m$ as $n\goesto \infty$ for suitable behaved functions $u_m$. Note that each $u_m^n$, and also $u_m$, have $cap_{1,p,N}-$ quasi-continuous representatives that are finite $cap_{1,p,N}-q.e.$ in $B_m$ (see Remark \ref{remark_quasicontinuity_DMOP99} and \ref{remark_quasicontinuity_DMOP99_limit}) which implies that they have well defined $cap_{1-\frac{1}{p},p,N-1}-$ quasi-continuous traces. By the same theorem, since each $u_m^n$ satisfies the estimate 
	\[
	\int_{B_m} \abs{\nabla T_k (u_m^n)} dx \leq 2k \abs{\tilde{\mu}_m^n}\left( B_m\right) \leq 4k\abs{\mu}\left( \R^N\right)
	\]
	so does the functions $u_m$. 
	
	Now we consider the convergence of $g_n(x,u_m^n)$ for fixed $m\in \N$. We fix $\abs{t}<m$ and for any $n$ and $h$ we write $ E_n^h =\left\lbrace\abs{u_m^n} < h\right\rbrace $ and $E^h= \left\lbrace\abs{u_m} <h\right\rbrace $.	Define $\sigma(s)= \abs{\left\lbrace x \in B_m \cap \left\lbrace x_N=t\right\rbrace  \ : \ \abs{u_m^n}>s\right\rbrace  } $. Proceeding as in Remark \ref{remark_on_L^p_estimate} we see that
	\begin{align*}
	\int_{B_m\cap(E_n^h)^c\cap \left\lbrace x_N=t\right\rbrace }\abs{g(u_m^n)}(x',t)dx' &\leq \int_0^\infty \abs{B_m\cap(E_n^h)^c\cap \left\lbrace x_N=t\right\rbrace \cap (E_n^s)^c}d\tilde{g}(s) \\ &= \sigma(h)\tilde{g}(h) + \int_h^\infty \sigma(s)d\tilde{g}(s).
	\end{align*}
	By Lemma \ref{L_levelset_Trace} 
	\begin{multline*}
	\sigma(h)\tilde{g}(h) + \int_h^\infty \sigma(s)d\tilde{g}(s) \leq \\ C(N,p,m)\norm{\mu_m}_{\mathfrak{M}_b}^{\frac{N-1}{N-p}} \left[  h^{-\frac{(N-1)(p-1)}{N-p}}\tilde{g}(h) + \int_h^\infty s^{-\frac{(N-1)(p-1)}{N-p}} d\tilde{g}(s)\right] 
	\end{multline*}
	while integration by parts gives
	\begin{multline*}
	h^{-\frac{(N-1)(p-1)}{N-p}}\tilde{g}(h) + \int_h^\infty s^{-\frac{(N-1)(p-1)}{N-p}} d\tilde{g}(s) = \\ \lim_{s\goesto\infty}\left[ s^{-\frac{(N-1)(p-1)}{N-p}}\tilde{g}(s) \right] + \frac{(N-1)(p-1)}{N-p} \int_h^\infty \tilde{g}(s)s^{-\frac{p(N-2)+1}{N-p}} ds.
	\end{multline*}
	Note that by Assumption \ref{assumption_subcritical}
	\[
	\lim_{s\goesto\infty}s^{-\frac{(N-1)(p-1)}{N-p}}\tilde{g}(s)\leq\frac{(N-1)(p-1)}{N-p}\lim_{s\goesto\infty}\int_s^\infty \tilde{g}(t)t^ {-\frac{p(N-2)+1}{N-p}} dt = 0
	\]
	and so we obtain
	\begin{align*}
	\int_{B_m\cap(E_n^h)^c\cap \left\lbrace x_N=t\right\rbrace }\abs{g(u_m^n)}(x',t)dx' & \leq C(N,p,m)\norm{\mu_m}_{\mathfrak{M}_b}^{\frac{N-1}{N-p}} \int_h^\infty\tilde{g}(s) s^{-\frac{p(N-2)+1}{N-p}} ds  \\
	 &\goesto 0 
	\end{align*}
 	as $h\goesto \infty$, uniformly in $t$ and $n$. Using the same argument we can show 
	\begin{align*}
	\int_{\pd\R^N_+\cap B_m\cap (E^h)^c} \abs{g(u_m)}dx' &\leq  C(N,m,p)\norm{\mu_m}_{\mathfrak{M}_b}^{\frac{N-1}{N-p}} \int_h^\infty\tilde{g}(s) s^{-\frac{p(N-2)+1}{N-p}} ds \\
	& \goesto 0
	\end{align*}
	as $h\goesto \infty$. Note also that
	\[
	\int_{B_{m}\cap\pd\R^N_+}\abs{g(u_m)}dx'\leq \tilde{g}(1)\abs{B_m\cap\pd\R^N_+} + C(N,p,m,\mu_m)\int_1^\infty \tilde{g}(s)s^{-\frac{p(N-2)+1}{N-p}}ds < \infty
	\]
	so in particular $g(u_m)\in L^1\left( B_m\cap\pd\R^N_+\right) $. 
	
	By the above considerations, we see that we can combine Lemma \ref{general_passage_to_trace_nonlinearity} and Lemma \ref{general_stability_in_balls} to obtain that $u_m$ is a renormalized solution of 
	\[
	\begin{cases}
	-\lap_pu_m =2\mu_m - 2g(u_m)\mathcal{H} & \mbox{ in } B_m \\
	u_m=0 & \mbox{ on } \pd B_m .
	\end{cases}
	\] 
	Moreover, since we have \eqref{estimate_g_n} and 
	\[
	\lim_{n\goesto\infty}\int_{B_m}\phi g_n(x,u_m^n)dx= \int_{B_m\cap\pd\R^N_+} \phi  g(u_m)dx'
	\]
	for any $\phi\in C_0^\infty\left( B_m\right) $ we get 
	\[
	\int_{B_m\cap \pd\R^N_+}\abs{g(u_m)}dx'\leq \abs{\mu_m}\left( B_m\right) \leq \abs{\mu}\left( \R^N_+\right)<\infty .
	\]
	Thus, we can apply Lemma \ref{general_existence_of_limit_functions} with data $2\mu_m - 2g(u_m)\mathcal{H}$ to obtain a suitable limit function $u$ such that $u_m\goesto u$ $a.e.$ in $\R^N$. Note that the above estimate says that $\norm{g(u_m)}_{L^1\left( B_m\cap\pd\R^N_+\right) }$ is uniformly bounded.
	
	Now we obtain estimates on the level sets of $u_m(x',0)$ and $u(x',0)$. Fix any $M>0$. Since $u_m$ satisfies estimate \eqref{estimate2_DMMOP99} and $\norm{\mu_m-g(u_m)\mathcal{H}}_{\mathfrak{M}_b}$ is uniformly bounded, we can find $k_0(M,N,p,\mu)$ independent of $m$ such that 
	\[
	 \abs{\left\lbrace  \abs{u_m}>\frac{k}{2}\right\rbrace }\leq \frac{1}{4}\abs{B_M}
	\]
	for all $k\geq k_0$. Let $c_{k,m,M}=(T_k(u_m))_M$ be the average of $T_k(u_m)$ in $B_M$. Then we estimate 
	\begin{multline*}
	\abs{c_{k,m,M}}\leq \frac{1}{\abs{B_M}}\left( \int_{B_M\cap \left\lbrace\abs{u_m}\leq k/2\right\rbrace }\abs{T_k(u_m)} dx+ \int_{B_M\cap \left\lbrace\abs{u_m}> k/2\right\rbrace }\abs{T_k(u_m)}dx\right) 
	\leq \\ \frac{k}{2} + \frac{k}{4} = \frac{3}{4}k 
	\end{multline*}
	for all $k\geq k_0$. As in the proof of Lemma \ref{L_levelset_Trace}, replacing Poincar\'e's inequality with Poincar\'e-Wirtinger's inequality, we obtain 
	\[
	\norm{T_k(u_m) - c_{k,m,M}}_{L^q\left( B_M\cap \pd\R^N_+\right) }\leq C(N,p,B_M)\left( k\norm{\mu}_{\mathfrak{M}_b}\right) ^{\frac{1}{p}} 
	\]
	with $q=\frac{p(N-1)}{N-p}$. Since for all $k\geq k_0$ we have the inclusions
	\begin{multline*}
	\left\lbrace \abs{u_m}\geq k\right\rbrace =\left\lbrace \abs{T_k(u_m)}\geq k\right\rbrace \subset \left\lbrace \abs{T_k(u_m) - c_{k,m,M}}\geq k-\abs{c_{k,m,M}}\right\rbrace  \\ \subset \left\lbrace \abs{T_k(u_m) - c_{k,m,M}}\geq \frac{k}{4}\right\rbrace , 
	\end{multline*}
	we similarly deduce
	\[
	\abs{\left\lbrace\abs{u_m}\geq k\right\rbrace \cap B_M\cap \pd\R^N_+}\leq C(N,p,B_M,\norm{\mu}_{\mathfrak{M}_b})k^{-\frac{(p-1)(N-1)}{N-p}}.
	\]
	In a similar way, by Fatou's Lemma, $u$ satisfies estimate \eqref{estimate1_DMMOP99} in $B_M$, while if $c_{k,M}=(T_k(u))_M$ is the average of $T_k(u)$ in $B_M$ then, by Lebesgue's Dominated Convergence Theorem, $\abs{c_{k,M}}\leq \frac{3}{4}k$ (see also the proof of Lemma \ref{general_existence_of_limit_functions}). Thus, we also have 
	\[
	\abs{\left\lbrace\abs{u}\geq k\right\rbrace \cap B_M\cap \pd\R^N_+}\leq C(N,p,B_M,\norm{\mu}_{\mathfrak{M}_b})k^{-\frac{(p-1)(N-1)}{N-p}}.
	\]
	
	Now we are ready to finish. As above, using the assumptions on $\tilde{g}$ we have
	\[
	\int_{\left\lbrace\abs{u_m}\geq k \right\rbrace \cap \pd\R^N_+\cap B_M}\tilde{g}(\abs{u_m})dx'\goesto 0
	\]
	as $k\goesto \infty$, uniformly on $m$. Similarly
	\[
	\int_{\left\lbrace\abs{u}\geq k \right\rbrace \cap \pd\R^N_+\cap B_M}\tilde{g}(\abs{u})dx'\goesto 0
	\]
	as $k\goesto \infty$. Then, we can apply Lemma \ref{general_passage_to_trace_nonlinearity_R^N} to obtain condition \eqref{condition_general_existence_global_solution} (with $g_m=g(u_m)$ and $g=g(u)$). Note that from the estimate $\norm{g(u_m)}_{L^1\left( B_M\cap\pd\R^N_+\right) }\leq \norm{\mu}_{\mathfrak{M}_b}$ and \eqref{condition_general_existence_global_solution} we conclude that $\norm{g(u)}_{L^1\left(\pd\R^N_+\right) }\leq \norm{\mu}_{\mathfrak{M}_b}$. Then, Lemma \ref{general_existence_global_solution} implies that $u$ is a local renormalized solution of 
	\[-\lap_p u = 2\mu - 2 g(u)\mathcal{H} \mbox{ in } \R^N . \]
	Since the measures $2\mu_m-2 g(u_m)\mathcal{H}$ are supported in $\pd\R^N_+$, we apply Theorem \ref{symmetry} to obtain that $u_m$, and thus $u$, are symmetric with respect to $\pd\R^N_+$. Hence, by Theorem \ref{thm_on_existence_from_symmetry} the restriction of $u$ to $\R^N_+$ is a solution to the problem.
\end{proof}

Now we consider the case $p=N$.

\begin{thm}\label{existence_subcritical_p=N}
	Let $p= N$ and $\mu\in \mathfrak{M}_b\left( \pd\R^N_+\right) $. Suppose $g(s)$ satisfies Assumption \ref{assumption_subcritical} with some $\gamma>0$. There exists a constant $C(N)$ such that if $\norm{\mu}_{\mathfrak{M}_b}\leq C(N)\gamma^{1-N}$ then there exists a renormalized solution of
	\[
	\begin{cases}
	-\lap_Nu=0 & \mbox{ in } \R^N_+ \\
	\abs{\nabla u}^{N-2}u_\n + g(u) = \mu & \mbox{ on } \pd\R^N_+ .
	\end{cases}
	\]
\end{thm} 
\begin{proof}
	We repeat the ideas used in the proof of Theorem \ref{existence_subcritical_p<N}, so we only point out the main differences. 
	
	As before, the first step is obtaining solutions $u_m^n$ to 
	\[
	\begin{cases}
	-\lap_Nu_m^n + 2g_n(x,u_m^n)=2\mu_m & \mbox{ in } B_m \\
	u_m^n=0 & \mbox{ on } \pd B_m .
	\end{cases}
	\]
	This could be achieved by using Theorem 5.1.2 of \cite{Veron16}, which guarantees the existence of a solution to
	\be\label{eq_veron_auxiliary}
	\begin{cases}
	-\lap_Nu + g_n(x,u)=\mu& \mbox{ in } B_m \\
	u=0 & \mbox{ on } \pd B_m
	\end{cases}
	\ee
	provided that $\norm{\mu}_{\mathfrak{M}_b}$ is bounded by $\left( \frac{C_0}{N\gamma}\right) ^{N-1}$, where $C_0=C_0(B_m)$ is a constant that may depend on the domain $B_m$. Note that, since we intend to take $m$ to infinity, this could be troublesome for us. Indeed, if $C_0(B_m)$ happens to vanish as $m\goesto \infty$ then requiring that the bound holds for all $m$ would lead to the conclusion that $\mu\equiv0$. Let us see that we can work around this problem.
	
	A look at the proof of Theorem 5.1.2 of \cite{Veron16} shows that the constant $C_0$ is exactly the constant in the estimate 
	\[\abs{\left\lbrace x \in B_m : \ \abs{u_\e}>k\right\rbrace }\leq C(N,B_m)e^{-C_0 k \norm{\mu}_{\mathfrak{M}_b}^{\frac{1}{1-N}}} \]
	which holds for solutions $u_\e$ to problem \eqref{eq_veron_auxiliary} with $\mu$ replaced by a regularized $\mu_\e$. Since any such solution satisfies $T_k(u_\e)\in W^{1,N}_0\left( B_m\right) $ and $\norm{\nabla T_k(u_\e)}_N\leq \left( Ck\norm{\mu}_{\mathfrak{M}_b}\right) ^\frac{1}{N}$ we have, as noted in Remark \ref{improvement}, that we can use the results of \cite{Cianchi08} to replace the above estimate with
	\[\abs{\left\lbrace x \in B_m : \ \abs{u_\e}>k\right\rbrace }\leq C(N,B_m)e^{-c_1(N) k \norm{\mu}_{\mathfrak{M}_b}^{\frac{1}{1-N}}} \]
	for some $c_1(N)$ independent of $B_m$. Hence, by applying the same argument as in \cite{Veron16}, but with the above estimate, it is easy to see that in fact a solution $u_m^n$ exists if we assume $\norm{2\mu_m}_{\mathfrak{M}_b}\leq \left( \frac{c_1}{N\gamma}\right) ^{N-1}$.
	
	Next, as in the case $1<p<N$, we obtain a limit function $u_m$ which we claim is a renormalized solution of
	\[
	\begin{cases}
	-\lap_Nu_m + 2 g(u_m)\mathcal{H}=2\mu_m & \mbox{ in } B_m \\
	u_m=0 & \mbox{ on } \pd B_m .
	\end{cases} 
	\]
	The proof of this claim is as before: by using that
	\[
	\int_h^\infty \tilde{g}(s)e^{-\gamma N s}ds \goesto 0 
	\]
	as $h\goesto\infty$, we can apply Lemmas \ref{L_levelset_Trace}, \ref{general_passage_to_trace_nonlinearity} and \ref{general_stability_in_balls} to show that $u_m$ solves the above equation.
	
	In the final step, we similarly obtain a limit function $u$ which yields a renormalized solution to 
	\[
	\begin{cases}
	-\lap_Nu = 0 \mbox{ in } \R^N_+ \\
	\abs{\nabla u}^{N-2}u_\n+g(u)=\mu & \mbox{ on } \pd \R^N_+
	\end{cases} 
	\]
	provided we can show \eqref{condition_general_existence_global_solution} holds. By Lemma \ref{general_passage_to_trace_nonlinearity_R^N}, it is enough to obtain 
	\[
	\int_{\left\lbrace\abs{u_m}\geq k \right\rbrace \cap \pd\R^N_+\cap B_M}\tilde{g}(\abs{u_m})dx'\goesto 0
	\]
	and
	\[
	\int_{\left\lbrace\abs{u}\geq k \right\rbrace \cap \pd\R^N_+\cap B_M}\tilde{g}(\abs{u})dx'\goesto 0
	\]
	as $k\goesto \infty$, uniformly on $m$.
	
	As in the case $1<p<N$ we want to estimate the averages of the solutions and proceed as in the proof of Lemma \ref{L_levelset_Trace}. Let us first observe that by the results of \cite{DHM00} the solutions $u_m$ belong the Lorentz-Sobolev space $W^1L^{N,\infty}\left( B_m\right) $, i.e., the space of weakly differentiable functions in $B_m$ such that (the absolute value of) their derivative belongs to the Lorentz space $L^{N,\infty}\left( B_m\right)$ (note that in the case $p=N$ renormalized solutions have well defined $L^1_{loc}$ derivatives; see Remark \ref{remark_on_existence_distributional_gradient}). 
	Moreover, one has 
	\[
	\norm{\nabla u_ m}_{L^{N,\infty}\left( B_m\right) }\leq C(N) \norm{2\mu_m}_{\mathfrak{M}_b}^\frac{1}{N-1} \leq C(N) \norm{2\mu}_{\mathfrak{M}_b}^\frac{1}{N-1} 
	\] 
	and so 
	\[
	\norm{u_m}_{BMO\left( B_M\right) }\leq  C(N,B_M)\norm{2\mu}_{\mathfrak{M}_b}^\frac{1}{N-1} 
	\] 
	for any $m\geq M$ (see also \cite{CP98}). Here $BMO\left( B_M\right) $ is the space of $L^N\left( B_M\right) $ functions of bounded mean oscillation (see \cite{DHM00} for a definition of $BMO$). On the other hand, by Theorem 2.5 of \cite{Cianchi08}, we can assert
	\[
	\int_{B_M\cap \pd\R^N_+}e^{C_1\frac{\abs{u_m-(u_m)_M}}{\norm{\nabla u_ m}_{L^{N,\infty}\left( B_m\right)}}}dx' \leq C(B_M,N)
	\]
	for some constant $C_1(N)$ and where $(u_m)_M$ is the average of $u_m$ in $B_M$. Hence, just as in Lemma \ref{L_levelset_Trace}, we obtain
	\begin{multline*}
	\abs{\left\lbrace x \in B_M \cap \pd\R^N_+ \ : \ \abs{u_m-(u_m)_M}>k\right\rbrace }\leq C(N,B_M)e^{-C_1 k \norm{\nabla u_ m}_{L^{N,\infty}\left( B_M\right) }^{-1}} \\ \leq C(N,B_M)e^{-c_2 k \norm{2\mu}_{\mathfrak{M}_b}^{\frac{1}{1-N}}} 
	\end{multline*}
	with $c_2=c_2(N)$. Since
	\[
	\left\lbrace\abs{u_m}\geq k\right\rbrace  \subset \left\lbrace\abs{u_m-(u_m)_M}\geq k-\abs{(u_m)_M}\right\rbrace ,
	\]
	we then get
	\begin{multline*}
	\abs{\left\lbrace x \in B_M \cap \pd\R^N_+ \ : \ \abs{u_m}>k\right\rbrace }  \leq C(N,B_M)e^{-c_2 (k-\abs{(u_m)_M}) \norm{2\mu}_{\mathfrak{M}_b}^{\frac{1}{1-N}}} \\ \leq  C(N,B_M,\mu)e^{-c_2 k\norm{2\mu}_{\mathfrak{M}_b}^{\frac{1}{1-N}}}
	\end{multline*}
	where we have used that $\abs{(u_m)_M}\leq C(N)\norm{u_m}_{BMO\left( B_M\right) }$. Thus we have
	\[
	\int_{B_M\cap\left\lbrace\abs{u_m}\geq k\right\rbrace \cap\pd\R^N_+}\tilde{g}(\abs{u_m})dx'\leq C(N,B_M,\mu)\int_k^\infty \tilde{g}(s)e^{-c_2 s\norm{2\mu}_{\mathfrak{M}_b}^{\frac{1}{1-N}}} ds
	\]
	which vanishes as $k\goesto \infty$ provided $\norm{2\mu}_{\mathfrak{M}_b}\leq \left( \frac{c_2}{N\gamma}\right) ^{N-1}$ (note that the above inequality can be obtained by the same argument as the one used in the case $1<p<N$).
	
	We now obtain the same estimate for $u$. We note that $\norm{\nabla u_m}_{L^{N,\infty}\left( B_M\right) }\leq C$ implies, by definition, that
	\[
	\abs{\left\lbrace\abs{\nabla u_m}>\l\right\rbrace }\leq C\l^{-N}
	\]
	for all $\l>0$ (see \cite{T98}). Since $\nabla u_m\goesto \nabla u$ $a.e.$ in $B_M$, by Fatou's Lemma we obtain that
	\[
	\abs{\left\lbrace\abs{\nabla u}>\l\right\rbrace }\leq C\l^{-N}
	\]
	for all $\l>0$ such that $\abs{\left\lbrace\abs{\nabla u}=\l\right\rbrace }=0$ and so, in particular, for $a.e.$ $\l>0$. Hence, by density, the bound can be seen to hold for all $\l>0$ and, again by definition, we obtain $\norm{\nabla u}_{L^{N,\infty}\left( B_M\right) }\leq C$. Thus, all the above computations remain true for $u$ and the desired estimate holds. 
	 
	Hence, we obtain a solution to the problem provided $\norm{\mu}_{\mathfrak{M}_b}\leq\frac{1}{2} \left( \frac{c}{N\gamma}\right) ^{N-1}$ where
	\[
	c=c(N)=\min\left( c_1(N),c_2(N)\right) .
	\]
\end{proof}

\begin{remark}\label{remark_on _convergence_of_trace}
	Note that we have proven stability of solutions without using any type of convergence of $u_m^n$ to $u_m$ (or of $u_m$ to $u$) in $\pd\R^N_+$. On the other hand, it is rather natural to expect that $g(u_m^n)\goesto g(u_m) $ strongly in $L^1\left( B_m\cap\pd\R^N_+\right) $ and $g(u_m)\goesto g(u)$ strongly in $L^1\left( \pd\R^N_+\cap B_M\right) $ for any $M\in \N$. Let us see that in fact we can assume convergence in $\pd\R^N_+$.
	
	Indeed, by Lemma \ref{general_stability_in_balls} we have $T_k(u_m^n)\goesto T_k(u_m)$ strongly in $W_0^{1,p}\left( B_m\right) $. Then, by Proposition 2.3.8 of \cite{AdamsHedberg}, up to a subsequence, $T_k(u_m^n)\goesto T_k(u_m)$ $cap_{1,p,N}-q.e.$ in $B_m$. By taking $k\in \N$ we may extract a diagonal subsequence $\{u_m^{n_j}\}_j$ from $\{u_m^{n}\}_n$ such that $T_k(u_m^{n_j})\goesto T_k(u_m)$ $cap_{1,p,N}-q.e.$ in $B_m$ as $j\goesto\infty$ for any $k\in \N$. We relabel this subsequence as $\{u_m^{n}\}_n$. Then, since $u_m$ is $cap_{1,p,N}-q.e.$ finite, we conclude that $u_m^n\goesto u_m$ $cap_{1,p,N}-q.e.$ in $B_m$. Hence, we may assume $u_m^n\goesto u_m$ $a.e.$ in $B_m\cap \pd\R^N_+$. Moreover, using the same estimates as in the above proofs, it is easy to show using Vitali's Theorem that this implies $g(u_m^n)\goesto g(u_m)$ strongly in $L^1\left( B_m\cap \pd\R^N_+\right) $. Similarly, one can use the strong convergence of $T_k(u_m)\goesto T_k(u)$ in $W^{1,p}\left( B_M\right) $ guaranteed by Lemma \ref{general_existence_global_solution} to show that, up to a subsequence, $u_m\goesto u$ $a.e.$ in $\pd\R^N_+$ and $g(u_m)\goesto g(u)$ strongly in $L^1\left( \pd\R^N_+\cap B_M\right) $ for any $M\in \N$.
\end{remark}

\section{The supercritical case}\label{supercritical}

We now obtain renormalized solutions to equation \eqref{equation} when $1<p<N$ and the absorption term $g(s)$ does not necessarily satisfy the growth estimates of Assumption \ref{assumption_subcritical}. In this case we can only guarantee existence of solutions if $\mu$ belongs to a subset of $\mathfrak{M}_b\left( \pd\R^N_+\right) $ which, in general, is strictly smaller. 

Throughout this section we will assume that
\[
g:\R\goesto \R \mbox{ is a continuous nondecreasing odd function.}
\]
Note that $g$ satisfies part $(1)$ of Assumption \ref{assumption_subcritical}, and that if $\tilde{g}$ is the function defined in part $(2)$ of Assumption \ref{assumption_subcritical} then $\tilde{g}(s)=g(s)$ ($s\geq 0$). Note also that $g(\abs{s})=\abs{g(s)}$ for any $s\in \R$.

As in the subcritical case, our starting point is the existence results for the problem 
\be \label{equation_B-VHV}
\begin{cases}
-\lap_pu + g(x,u)=\mu  & \mbox{ in } \O \\
u= 0 & \mbox{ on } \pd\O .
\end{cases}
\ee
We will use the existence results obtained in \cite{B-VHV14} (see also \cite{Veron16}), which rest mainly on the study of the Wolff potential of the measure $\mu$. As it turns out, the estimates involved are well suited to study trace problems such as ours. 

We begin by defining the $R$-truncated Wolff potential of a nonnegative measure $\mu\in \mathfrak{M}_b\left( \R^N\right) $ by
\[
W_{\a,s,N}^R\left[ \mu\right] (x)=\int_0^R\left( \frac{\mu\left( B_t(x)\right) }{t^{N-\a s}}\right) ^\frac{1}{s-1}\frac{dt}{t}
\]
where $\a>0$, $1<s<\a^{-1}N$, $0<R\leq \infty$, and $B_t(x)$ is the $N$-dimensional ball of radius $t$ centered at $x\in \R^N$. If $R=\infty$ we just write $W_{\a,s,N}\left[ \mu\right] $.

It follows immediately from the definition that if $\mu$ is supported in $\pd\R^N_+$ then 
\begin{align}\label{identity_trace_potential}
W_{1-\frac{1}{p},p,N-1}^R\left[ \mu\right] (x')&=\int_0^R\left( \frac{\mu\left( B_t(x')\times \left\lbrace x_N=0\right\rbrace \right) }{t^{N-1-(p-1)}}\right) ^\frac{1}{p-1}\frac{dt}{t}\\ \nonumber
&=\int_0^R\left( \frac{\mu\left( B_t(x',0)\cap \pd\R^N_+\right) }{t^{N-p}}\right) ^\frac{1}{p-1}\frac{dt}{t}\\ \nonumber
&=W_{1,p,N}^R\left[ \mu\right] (x',0)
\end{align}
for any $1<p<N$, $x'\in \R^{N-1}$. 
Moreover, we clearly have $B_t(x',x_N)\cap\pd\R^N_+\subset B_t(x',0)\cap\pd\R^N_+$ and so
\be\label{bound_trace_potential}
W_{1,p,N}^R\left[ \mu\right] (x',x_N)\leq W_{1-\frac{1}{p},p,N-1}^R\left[ \mu\right] (x')
\ee
for any $(x',x_N)\in \R^N$. 

\begin{remark}\label{remark_on_wolff_potential_estimate}
	Let us record the following important relationship between Wolff potential and $p-$ superharmonic functions: if $u$ is a nonnegative $p-$ superharmonic function in $\O$, $1<p\leq N$, and $-\lap_pu=\mu$ in $\O$ then there exists positive constants $c_1$, $c_2$, $c_3$, depending only on $p$ and $N$, such that for any $x\in\O$ and $B_{3r}(x)\subset\subset \O$ there holds
	\[
	c_1 W_{1,p,N}^r\left[\mu\right] (x)\leq u(x) \leq c_2\inf_{B_r(x)}u + c_3  W_{1,p,N}^r\left[\mu\right] (x)
	\]
	(see \cite{HKM}, \cite{PV08} or \cite{Veron16}). 
\end{remark}

The following existence result is Theorem 4.1 of \cite{B-VHV14}.
\begin{thm}\label{B-VHV14}
	Let $\O$ be a bounded domain and let $g:\O\times \R\goesto \R$  be a Caratheodory function such that $s\mapsto g(x,s)$ is nondecreasing and odd for $a.e.$ $x\in \O$. Then there exists a constant $c=c(N,p)$ such that the following is true: if $\mu_i \in \mathfrak{M}_b\left( \O\right) $, $i=1,2$, are nonnegative and there exists nondecreasing sequences $\left\lbrace\mu_{i,n}\right\rbrace _n$ of nonnegative measures in $\mathfrak{M}_b\left( \O\right) $ with compact support in $\O$ converging to $\mu_i$ weakly and such that $g\left( cW_{1,p,N}^{2\diam{\O}}\left[ \mu_{i,n}\right] \right) \in L^1\left( \O\right) $ then there exists a renormalized solution of
 		\[
		\begin{cases}
		-\lap_pu + g(x,u)=\mu_1-\mu_2  & \mbox{ in } \O \\
		u= 0 & \mbox{ on } \pd\O .
		\end{cases}
		\]
	Moreover,
	\be\label{eq_estimate_B-VHV14}-cW_{1,p,N}^{2\diam{\O}}\left[ \mu_2\right](x)\leq u(x) \leq cW_{1,p,N}^{2\diam{\O}}\left[ \mu_1\right](x) \mbox{ $a.e.$ in } \O . \ee
\end{thm}

Our first goal is to improve estimate \eqref{eq_estimate_B-VHV14} from $a.e.$ in $\O$ to $cap_{1,p,N}-q.e.$ in $\O$.

\begin{lem}\label{trace_potential_estimate}
	Let $\O$ be a bounded domain and let $g:\O\times \R\goesto \R$ be a Caratheodory function such that $s\mapsto g(x,s)$ is nondecreasing and $sg(x,s)\geq 0$ for $a.e.$ $x\in \O$ and all $s\in \R$. Let $\mu \in \mathfrak{M}_b\left( \O\right)$ and let $u$ be a renormalized solution to
	\[
	\begin{cases}
	-\lap_pu + g(x,u)=\mu  & \mbox{ in } \O \\
	u= 0 & \mbox{ on } \pd\O.
	\end{cases}
	\]
	Then
	\[-cW_{1,p,N}^{2\diam{\O}}\left[ \mu^-\right](x)\leq u(x) \leq cW_{1,p,N}^{2\diam{\O}}\left[ \mu^+\right](x)  \]
	$cap_{1,p,N}-q.e.$ in $\O$ with $c=c(N,p)$ as in the statement of Theorem \ref{B-VHV14}.
\end{lem}
\begin{proof}
	We know that for every $k\geq 0$ the functions $T_k(u)=u_k$ are renormalized solutions to 
	\[
	\begin{cases}
	-\lap_p u_k + g(x,u_k)\chi_{\left\lbrace\abs{u}<k\right\rbrace }=\mu_0\chi_{\left\lbrace\abs{u}<k\right\rbrace } + \l_k^+ - \l_k^-  & \mbox{ in } \O \\
	u= 0 & \mbox{ on } \pd\O
	\end{cases}
	\]
	for some nonnegative measures $\l_k^\pm\in \mathfrak{M}_b\left( \O\right) $ that converge to $\mu_s^\pm$ in the narrow topology of measures. 
	Let $v_k$ be a renormalized solution to  
	\[
	\begin{cases}
	-\lap_p v_k=\mu_0^+\chi_{\left\lbrace\abs{u}<k\right\rbrace } + \l_k^+  & \mbox{ in } \O \\
	v_k= 0 & \mbox{ on } \pd\O .
	\end{cases}
	\]
	Since $\mu_0^+\chi_{\left\lbrace\abs{u}<k\right\rbrace } + \l_k^+$ is nonnegative we have $v_k\geq 0$ (see Remark 6.5 of \cite{PV08}), and so $g(x,v_k)\geq 0$ $a.e.$ in $\O$. Since $u_k$ is bounded we have $g(x,v_k)\chi_{\left\lbrace u_k>v_k\right\rbrace }\in L^1\left( \O\right) $ and so we can use that $g(x,s)$ is $a.e.$ nondecreasing on $s$ and that all the measures involved are in $\mathfrak{M}_0\left( \O\right)$ to obtain, by an easy adaptation of the proof of Lemma 6.8 of \cite{PV08}, that $u_k\leq v_k$ $a.e.$ in $\O$. By Theorem 3.4 of \cite{DMOP99}, passing to a subsequence we have $v_k\goesto v$ $a.e.$ in $\O$ where $v$ is a renormalized solution to 
	\[
	\begin{cases}
	-\lap_p v=\mu_0^+ + \mu_s^+=\mu^+  & \mbox{ in } \O \\
	v= 0 & \mbox{ on } \pd\O .
	\end{cases}
	\]
	In particular, since $u$ is $a.e.$ finite, $u\leq v$ $a.e.$ in $\O$.
	
	Since $\mu^+$ is nonnegative, by Theorem 2.1 of \cite{PV08} $v$ coincides $a.e.$ in $\O$ with a $p-$ superharmonic function $\tilde{v}$ satisfying 
	\[\tilde{v}(x)\leq cW_{1,p,N}^{2\diam{\O}}\left[ \mu^+\right](x) \]
	in $\O$ where $c=c(N,p)$ is the same constant as in Theorem \ref{B-VHV14} (see the proof of Theorem 3.8 in \cite{B-VHV14}).  
	Moreover, by Theorem 10.9 of \cite{HKM} $\tilde{v}$ is $cap_{1,p,N}-$ quasi-continuous in $\O$. Considering the $cap_{1,p,N}-$ quasi-continuous representative of $u$, and since $u\leq \tilde{v}$ $a.e.$ in $\O$, we can conclude $u\leq \tilde{v}$ $cap_{1,p,N}-q.e.$ in $\O$ (see Remark \ref{corolary_to_T_6.1.4}).
	Hence, 
	\[
	u\leq cW_{1,p,N}^{2\diam{\O}}\left[ \mu^+\right](x) 
	\]
	$cap_{1,p,N}-q.e.$ in $\O$. The lower estimate can be obtained similarly.
\end{proof}

\begin{remark}\label{remark_on_cap_equality_with_psuperharmonic}
	We note that in the second part of the above proof we have used that the $p-$ superharmonic representative of a nonnegative renormalized solutions $u$, mentioned in Remark \ref{remark_on_p_superharmonic_representative}, is a $cap_{1,p,N}-$ quasi-continuous representative of $u$. We will use this fact in the sequel. 
	
	Let us also mention the following: if $u\leq v$ $a.e.$ in $\O$, where $u$ and $v$ are $p-$ superharmonic, then $u\leq v$ everywhere in $\O$. Indeed, this follows from applying Corollary 7.23 of \cite{HKM} to the $p-$ superharmonic function $\min(u,v)$.
\end{remark}

The above estimate is sufficient to obtain local solutions to \eqref{equation}. To obtain global solutions we need to compare solutions defined in  nondecreasing sequences of domains. The following lemma asserts that for a nonnegative measure it is possible to obtain nondecreasing solutions defined on nondecreasing domains.

\begin{lem}\label{comparison_pple}
	Let $\O$ and $\O'$ be bounded domains such that $\O\subset \O'$. Let $\mu \in \mathfrak{M}_b\left( \O\right) $ be nonnegative, compactly supported in $\O$, and assume $g\left( cW_{1,p,N}^{2\diam{\O'}}\left[ \mu\right]\right) \in L^1\left( \O'\right) $ where $g(x,s)$ and $c=c(N,p)$ are as in Theorem \ref{B-VHV14}. Then there exists renormalized solutions $u$ and $v$ to \\
	\be\label{comparison_pple_eq_u}
	\begin{cases}
	-\lap_p u + g(x,u)=\mu  & \mbox{ in } \O \\
	u= 0 & \mbox{ on } \pd\O
	\end{cases}
	\ee
	and
	\be\label{comparison_pple_eq_v}
	\begin{cases}
	-\lap_p v + g(x,v)=\mu & \mbox{ in } \O' \\
	v= 0 & \mbox{ on } \pd\O' ,
	\end{cases}
	\ee
	respectively, such that $u\leq v$ $a.e.$ in $\O$. 
\end{lem}  
\begin{proof}
	Suppose first that $g$ is bounded. Then Lemma 4.2 of \cite{B-VHV14} shows that the desired solutions $u$ and $v$ exists and that they can be defined as the $a.e.$ limit of sequences $\left\lbrace u_n\right\rbrace _n$ and $\left\lbrace v_n\right\rbrace _n$ of weak solutions to \eqref{comparison_pple_eq_u} and \eqref{comparison_pple_eq_v}, respectively, with data $\mu_n$ converging to $\mu$ in a weak sense. Since the solutions are nonnegative, $u_n \in W_0^{1,p}\left( \O\right) $, and $v_n\in W^{1,p}\left( \O\right) $, the maximum principle shows that $u_n\leq v_n$ $a.e.$ in $\O$. Hence, $u\leq v$ $a.e.$ in $\O$.
	If $g$ is not bounded then one can proceed as in the proof of Lemma 4.3 of \cite{B-VHV14} and consider the truncations $T_n(g)$. The fact that $\norm{g\left(c W_{1,p,N}^{2\diam{\O}}\left[ \mu\right]\right) }_{L^1\left( \O\right)} \leq \norm{g\left(c W_{1,p,N}^{2\diam{\O'}}\left[ \mu\right]\right) }_{L^1\left( \O'\right)} $ shows that one can pass to the limit as $n\goesto \infty$ to obtain solutions that conserve the desired property. 
\end{proof}

Next we use Lemma \ref{trace_potential_estimate} to show that we can obtain solutions with absorption term $ g(u)\mathcal{H}$ from solutions to problem \eqref{equation_B-VHV}.

\begin{lem}\label{passage_to_trace}
	Let $g$ be a continuous nondecreasing odd function, and let $c=c(N,p)$ be the constant in Theorem \ref{B-VHV14}. Let $\mu \in \mathfrak{M}_b\left( \pd\R^N_+ \right)$ be such that $g\circ cW_{1-\frac{1}{p},p,N-1}^{4m}\left[ \mu^\pm\right]$ is in $L^1\left( \pd\R^N_+\cap B_m\right) $. Let $g_n(x,s)$ be defined as in \eqref{def_g_n} and let $u_n$ be renormalized solutions to
	\[
	\begin{cases}
	-\lap_pu_n + g_n(x,u_n)=\mu  & \mbox{ in } B_m \\
	u_n=0 & \mbox{ on } \pd B_m .
	\end{cases}
	\]
	Then there exists a function $u$ and a subsequence of $\left\lbrace u_n\right\rbrace _n$, which we relabel as $\left\lbrace u_n\right\rbrace _n$, such that $u_n\goesto u$ $a.e.$ in $B_m$ and $u$ is a renormalized solution to 
	\be
	\begin{cases}
	-\lap_pu +  g(u)\mathcal{H}=\mu  & \mbox{ in } B_m \\
	u=0 & \mbox{ on } \pd B_m
	\end{cases}
	\ee
	that satisfies
	\be
	\label{trace_control}-cW_{1-\frac{1}{p},p,N-1}^{4m}\left[ \mu^-\right](x')\leq u(x',x_N) \leq cW_{1-\frac{1}{p},p,N-1}^{4m}\left[ \mu^+\right](x')  
	\ee
	$cap_{1,p,N}-$q.e in $B_m$.
\end{lem}
\begin{proof}
	By Lemma \ref{trace_potential_estimate} the functions $u_n$ satisfy
	\[-cW_{1,p,N}^{4m}\left[ \mu^-\right](x)\leq u_n(x) \leq cW_{1,p,N}^{4m}\left[ \mu^+\right](x) \ \mbox{ $cap_{1,p,N}-q.e.$ in $B_m$ }  \]
	and so, by \eqref{bound_trace_potential}, they satisfy estimate \eqref{trace_control}. 
	
	Since the measure $\mu$ is bounded independent of $n$, Proposition \ref{Veron16_bound_on_g} implies that the same is true for the measures $\tilde{\mu}_n = \mu-g_n(x,u_n)$. Hence, we can apply Theorem \ref{DMOP99_convergence_of_u_n_to_limit} and obtain that, up to a subsequence there exists, a suitable behaved function $u$ defined in $B_m$ such that $u_n\goesto u$ $a.e.$ in $B_m$ as $n\goesto \infty$. 
	
 	Using than $u_n$ satisfies \eqref{trace_control}, which holds $a.e.$ in the intersection of $B_m$ with any hyperplane, and that $g$ is nondecreasing we conclude
	\begin{multline*}
	\int_{\left\lbrace u_n\geq h\right\rbrace \cap \left\lbrace x_N=t\right\rbrace }g\left( u_n\right) (x',t)dx' \\ \leq \int_{\left\lbrace cW_{1-\frac{1}{p},p,N-1}^{4m}\left[\mu^+\right]\geq h\right\rbrace \cap \pd\R^N_+\cap B_m}g\left( cW_{1-\frac{1}{p},p,N-1}^{4m}\left[ \mu^+\right]\right) (x')dx'	\goesto 0 
	\end{multline*}
	as $h\goesto\infty$ and similarly
	\begin{multline*}
	\int_{\left\lbrace u_n\leq -h\right\rbrace \cap \left\lbrace x_N=t\right\rbrace }-g\left( u_n\right) (x',t)dx' \\ \leq \int_{\left\lbrace cW_{1-\frac{1}{p},p,N-1}^{4m}\left[ \mu^-\right]\geq h\right\rbrace \cap \pd\R^N_+ \cap B_m}g\left( cW_{1-\frac{1}{p},p,N-1}^{4m}\left[\mu^-\right]\right) (x')dx' \goesto 0 
	\end{multline*}
	as $h\goesto\infty$, since $g\left( cW_{1-\frac{1}{p},p,N-1}^{4m}\left[\mu^\pm\right]\right) \in L^1\left( B_m\cap \pd\R^N_+\right) $. 
	
	The above estimates can be seen to hold also for the limit function $u$.  Indeed, it is enough to show that $u$ also satisfies \eqref{trace_control}. Looking at the proof of Lemma \ref{trace_potential_estimate}, we see that we obtain the right hand side of estimate \eqref{trace_control} for $u_n$ from the inequality $u_n\leq v$ for some particular renormalized solution $v$. Using that $u$ is the $a.e.$ limit of the $u_n$ we get $u\leq v$ $a.e.$ in $B_m$. Then, considering $cap_{1,p,N}-$ quasi-continuous representatives, we conclude $u\leq v$ $cap_{1,p,N}-q.e.$ in $B_m$ (see Remark \ref{corolary_to_T_6.1.4}), and so, proceeding as in the proof of the lemma, we obtain that the right hand side of estimate \eqref{trace_control} also holds for $u$. The left hand side estimate follows in the same way, and so \eqref{trace_control} holds for $u$. With this estimate we obtain 
	\[\int_{\left\lbrace\abs{u}\geq h\right\rbrace \cap \pd\R^N_+}\abs{g\left( u\right)} dx'\goesto 0\]
	as $h\goesto \infty$ and $g(u)\in L^1\left( B_m\cap\pd\R^N_+\right) $. Hence we apply Lemma \ref{general_passage_to_trace_nonlinearity} together with Lemma \ref{general_stability_in_balls} to finish the proof.	
\end{proof}

We are now ready to show the following trace version of Theorem \ref{B-VHV14}. Recall that we assume $1<p<N$. 

\begin{thm}\label{B-VHV14_trace}
	Let $g$ be a continuous nondecreasing odd function, and let $c_1=2^{\frac{1}{p-1}}c(N,p)$ where $c(N,p)$ is the constant in Theorem \ref{B-VHV14}. Assume $\mu_i \in \mathfrak{M}_b\left( \pd\R^N_+\right) $, $i=1,2$, are nonnegative and for every $m\in \N$ there exists nondecreasing sequences $\left\lbrace\mu^m_{i,k}\right\rbrace _k$ of nonnegative measures in $\mathfrak{M}_b\left( \pd\R^N_+\right) $ with compact support in $B_m\cap\pd\R^N_+$ converging to $\mu_i^m=\left. \mu_i\right|_{B_m\cap\pd\R^N_+}$ weakly-$\ast$ in $\mathfrak{M}_b\left( B_m\cap\pd\R^N_+\right) $ such that $g\circ c_1W_{1-\frac{1}{p},p,N-1}^{4(m+1)}\left[ \mu^m_{i,k}\right]$ is in $L^1\left( \pd\R^N_+\cap B_m\right) $ and $\mu^m_{i,k}\leq \mu^{m+1}_{i,k}$ for each $k\in \N$. Then there exists a renormalized solution of
	\[
	\begin{cases}
	-\lap_pu=0  & \mbox{ in } \R^N_+ \\
	\abs{\nabla u}^{p-2}u_\n + g(u)= \mu_1 - \mu_2 & \mbox{ on } \pd\R^N_+ .
	\end{cases}
	\]
	Moreover,
	\be\label{eq_control_solution}
	-c_1W_{1-\frac{1}{p},p,N-1}\left[ \mu_2\right](x')\leq u(x',x_N) \leq c_1W_{1-\frac{1}{p},p,N-1}\left[ \mu_1\right](x')  
	\ee
	$cap_{1,p,N}-q.e.$ in $\overline{\R^N_+}$.
\end{thm}
\begin{proof}
	Let $g_n(x,s)$ be defined as in \eqref{def_g_n}. It is easy to see that $g_n(x,s)=\zeta_n(x_N)g(s)$ satisfies the assumptions of Theorem \ref{B-VHV14}. Note that $\mu^m_{i,k}$ has compact support in $B_m$. Since 
	\[W_{1-\frac{1}{p},p,N-1}^{4(m+1)}\left[ \mu^m_{i,k}\right](x')\geq W_{1,p,N}^{4m}\left[ \mu^m_{i,k}\right] (x',x_N)\]
	and $g$ is nondecreasing we obtain
	\[
	\norm{g_n\left( cW_{1,p,N}^{4m}\left[ 2\mu^m_{i,k}\right] \right)}_{L^1\left( B_m\right) } \leq C(n,m)\norm{g\left(c_1W_{1-\frac{1}{p},p,N-1}^{4m}\left[ \mu^m_{i,k}\right] \right) }_{L^1\left( B_m\cap\pd\R^N_+\right) }<\infty.
	\]
	Hence, we may apply Lemma 4.3 of \cite{B-VHV14} to obtain renormalized solutions $u_m^{n,k}$ and $u_{m,i}^{n,k}$, with $i=1,2$, of 
	\be\label{general_eqn_g_n}
	\begin{cases}
	-\lap_p u + 2g_n(x,u)=2 \mu & \mbox{ in } B_m \\
	u= 0 & \mbox{ on } \pd B_m
	\end{cases}
	\ee
	with data $\mu=\mu^m_{1,k}-\mu^m_{2,k}$, and $\mu=\mu^m_{i,k}$, respectively, satisfying  
	\[-u_{m,2}^{n,k}\leq u_m^{n,k}\leq u_{m,1}^{n,k} \]
	$a.e.$ in $B_m$. Let us remark that by the same lemma we can assume $u_{m,i}^{n,k}\leq u_{m,i}^{n,k+1}$ $a.e.$ in $B_m$. Note also that the functions $u_{m,i}^{n,k}$ are nonnegative (proceed as in Remark 6.5 of \cite{PV08}, testing against $T_k(\min(u_{m,i}^{n,k},0))$ and using the hypothesis on $g$). 
	
	For any fixed $n$, $k$, and $m$, the measures $\mu^m_{i,k}$ satisfy all the necessary conditions to guarantee, again by Lemma 4.3 of \cite{B-VHV14}, the existence of renormalized solutions $w_{m,i}^{n,k}$, $i=1,2$, to problem \eqref{general_eqn_g_n} with data $\mu=\mu^m_{i,k}$ in $B_{m+1}$. Since $\mu^m_{i,k}\leq \mu^{m+1}_{i,k}$ we can combine the results of Lemma 4.3 of \cite{B-VHV14} with Lemma \ref{comparison_pple} above to further assume $u_{m,i}^{n,k}\leq w_{m,i}^{n,k} \leq u_{m+1,i}^{n,k}$ $a.e.$ in $B_m$. That is, we may assume the solutions $u_{m,i}^{n,k}$ are nondecreasing in $m$. 
	
	Now, applying Lemma \ref{passage_to_trace} we take $n\goesto \infty$ to obtain renormalized solutions $u_m^k$, and $u_{m,i}^k$ to 
	\be\label{general_eqn}
	\begin{cases}
		-\lap_p u + 2 g(u)\mathcal{H}=2 \mu & \mbox{ in } B_m \\
		u= 0 & \mbox{ on } \pd B_m
	\end{cases}
	\ee
	with data $\mu=\mu^m_{1,k}-\mu^m_{2,k}$, and $\mu=\mu^m_{i,k}$, respectively. By Lemma \ref{general_stability_in_balls} (which is used in the proof of Lemma \ref{passage_to_trace}), we have $T_h(u_{m,i}^{n,k})\goesto T_h(u_{m,i}^k)$ and $T_h(u_{m}^{n,k})\goesto T_h(u_m^k)$ strongly in $W_0^{1,p}\left( B_m\right) $ for any $h>0$. Since renormalized solutions are $cap_{1,p,N}-q.e.$ finite, we can use Proposition 2.3.8 of \cite{AdamsHedberg} to obtain, passing to a diagonal subsequence, that $u_{m,i}^{n,k}\goesto u_{m,i}^k$ and $u_m^{n,k}\goesto u_m^k$ $cap_{1,p,N}-q.e.$ in $B_m$ (see Remark \ref{remark_on _convergence_of_trace}). Hence, we can assume that
	\begin{multline}\label{LDC_control_2}
	-cW_{1-\frac{1}{p},p,N-1}^{4m}\left[ 2\mu^m_{2,k}\right](x')\leq 
	-u_{m,2}^{k}(x',t)\leq u_m^k(x',t) \\ \leq u_{m,1}^k(x',t) 
	\leq cW_{1-\frac{1}{p},p,N-1}^{4m}\left[ 2\mu^m_{1,k}\right](x')
	\end{multline}
	$cap_{1,p,N}-q.e.$ in $\R^N$. Here we have used again that we are considering $cap_{1,p,N}-$ quasi-continuous representatives, so that we can extend the inequalities from $a.e.$ in $B_m$ to $cap_{1,p,N}-q.e.$ in $B_m$ (see Remark \ref{corolary_to_T_6.1.4}). Note that we also can assume
	\[0\leq u_{m,i}^k(x)\leq u_{m,i}^{k+1}(x) \mbox{ and } u_{m,i}^k(x)\leq u_{m+1,i}^k(x)\]
	$cap_{1,p,N}-q.e.$ in $B_m$, and so in particular $a.e.$ in $B_m\cap \pd\R^N_+$. 
	
	Next we fix $m\in \N$. Since the measures $\mu^m_{i,k}$ are uniformly bounded in norm by $ \mu_i$ we obtain by Proposition \ref{Veron16_bound_on_g} that
	\[
	\norm{g(u_{m,i}^{n,k})}_{L^1\left( B_m\right) } \leq \norm{ \mu_i}_{\mathfrak{M}_b}.
	\]
	By Lemma \ref{general_passage_to_trace_nonlinearity} we have
	\[
	\lim_{n\goesto\infty}\int_{B_m}\phi g_n(x,u_{m,i}^{n,k})dx= \int_{B_m\cap\pd\R^N_+} \phi  g(u_{m,i}^k)dx'
	\]
	for any $\phi\in C_0^\infty\left( B_m\right) $. Thus,  
	\[
	\norm{g(u_{m,i}^{k})}_{L^1\left( B_m\cap \pd\R^N_+\right) } \leq \norm{ \mu_i}_{\mathfrak{M}_b}
	\]
	and by \eqref{LDC_control_2}
	\[
	\norm{g(u_{m}^{k})}_{L^1\left( B_m\cap \pd\R^N_+\right) } \leq \norm{ \mu_1}_{\mathfrak{M}_b} + \norm{ \mu_2}_{\mathfrak{M}_b}.
	\]
	
	With the above estimates we can apply Theorem \ref{DMOP99_convergence_of_u_n_to_limit} to obtain the existence of subsequences such that $u_m^k\goesto u_m$ and $u_{m,i}^k\goesto u_{m,i}$ $a.e.$ in $B_m$ as $k\goesto \infty$ for some suitable behaved functions $u_m$ and $u_{m,i}$. 
	Note that there is no loss of generality in assuming $u^k_{m,i}$ coincides with its $p-$ superharmonic representative mentioned in Remark \ref{remark_on_cap_equality_with_psuperharmonic}, so that in particular we can assume $u_{m,i}^k$ are nondecreasing in $k$ everywhere in $B_m$. Then, Lemma 7.3 of \cite{HKM} shows that $\sup_k u_{m,i}^k$ is a $p-$ superharmonic function in $B_m$ and so, by Theorem 10.9 of \cite{HKM}, also $cap_{1,p,N}-$ quasi-continuous in $B_m$. Hence, $\sup_k u_{m,i}^k$ is a $cap_{1,p,N}-$ quasi-continuous representative of $u_{m,i}$, and we can assume $u_{m,i}^k\goesto u_{m,i}$ $cap_{1,p,N}-q.e.$ in $B_m$. Thus, considering $cap_{1,p,N}-$ quasi-continuous representatives, we conclude from \eqref{LDC_control_2} that
	\begin{multline}\label{LDC_control_3}
	-cW_{1-\frac{1}{p},p,N-1}^{4m}\left[ 2\mu^m_{2}\right](x')\leq -u_{m,2}(x',t)\leq u_m(x',t) \\ \leq u_{m,1}(x',t) \leq cW_{1-\frac{1}{p},p,N-1}^{4m}\left[ 2\mu^m_{1}\right](x')
	\end{multline}
	$cap_{1,p,N}-q.e.$ in $\R^N$.
	
	Since we have the estimate 
	\[\norm{g(u_{m,i}^k)}_{L^1\left(B_m\cap\pd\R^N_+\right) }\leq \norm{\mu_i}_{\mathfrak{M}_b}<\infty\]
	and $u_{m,i}^k$ are nondecreasing in $k$ and nonnegative, we obtain by Monotone Convergence that $g(u_{m,i}^k)\goesto g(u_{m,i})\in L^1\left(B_m\cap\pd\R^N_+\right)$ and moreover
	\[
	\norm{g(u_{m,i})}_{L^1\left(B_m\cap\pd\R^N_+\right) }\leq \norm{\mu_i}_{\mathfrak{M}_b} . 
	\]
	Then, by slightly modifying the arguments leading to Corollary 3.5 of \cite{B-VHV14} we obtain that $u_{m,i}$ is a renormalized solution to \eqref{general_eqn} with data $\mu=\mu^m_i$. Indeed, to obtain the same stability result we only need to consider the terms $ g(u_{m,i}^k)\mathcal{H}$ and $ g(u_{m,i})\mathcal{H}$, since the focus of the corollary is the handling of the measures $\mu_{m,i}^k$ in order to apply the stability result of \cite{DMOP99}. But, replacing this stability result by the one in \cite{M05}, we see by the proof of Lemma \ref{general_stability_in_balls} (or Lemma \ref{general_existence_global_solution}) that we can prove stability provided
	\[
	\lim_{k\goesto\infty}\int_{B_{m} \cap \pd\R^N_+} \phi_kg(u_{m,i}^k)dx'= \int_{B_{m} \cap \pd\R^N_+}\phi g(u_{m,i})dx'
	\]
	for any $\left\lbrace\phi_k\right\rbrace _k$ converging to $\phi$ both $a.e.$ in $B_m$ and weakly in $W_0^{1,p}\left( B_m\right) $ and such that $\phi_k$ is uniformly bounded in $L^\infty\left( B_m\right) $. By Lemma \ref{general_passage_to_trace_nonlinearity_R^N}, it is enough to show that 
	\[
	\int_{\left\lbrace\abs{u_{m,i}^k}\geq h \right\rbrace \cap \pd\R^N_+\cap B_m}\abs{g(u_{m,i}^k)}dx' + \int_{\left\lbrace\abs{u_{m,i}}\geq h \right\rbrace \cap \pd\R^N_+\cap B_m}\abs{g(u_{m,i})}dx'  \goesto 0
	\]
	as $h\goesto\infty$ uniformly in $k$. But this is clearly true since $u_{m,i}^k$ are nonnegative and $g(u_{m,i}^k)\uparrow g(u_{m,i}) \in L^1\left(B_m\cap\pd\R^N_+\right)$.
	
	Similarly, we can show that $u_m$ is a renormalized solution to \eqref{general_eqn} with data $\mu=\mu^m_1-\mu^m_2$ provided we show that 
	\[
	\int_{\left\lbrace\abs{u_m}\geq h \right\rbrace \cap \pd\R^N_+\cap B_m}\abs{g(u_m)}dx'\goesto 0
	\]
	and
	\[
	\int_{\left\lbrace \abs{u_m^k}\geq h \right\rbrace \cap \pd\R^N_+\cap B_m}\abs{g(u_m^k)}dx'\goesto 0
	\]
	as $h\goesto \infty$, uniformly in $k$. Now, by the monotonicity of $g$ and the fact that $g$ is odd, we conclude from estimate \eqref{LDC_control_2} that
	\[\abs{g(u_m^k)}\leq g(u_{m,1}^k) + g(u_{m,2}^k)\leq g(u_{m,1}) + g(u_{m,2}) \]
	while from \eqref{LDC_control_3} we have
	\[\abs{g(u_m)}\leq g(u_{m,1}) + g(u_{m,2}) .\]
	Moreover 
	\[\left( \left\lbrace\abs{u_m}\geq h \right\rbrace\cup\left\lbrace\abs{u_m^k}\geq h \right\rbrace\right)  \subset \left\lbrace u_{m,1}\geq h\right\rbrace  \cup \left\lbrace u_{m_2}\geq h\right\rbrace  \]
	outside a set of zero measure in $B_m\cap\pd\R^N_+$. Hence, since $g(u_{m,1}) + g(u_{m,2})\in L^1\left( B_m\cap\pd\R^N_+\right) $, the desired estimates hold. 
	
	To finish we can proceed exactly as in the proof of Theorem \ref{existence_subcritical_p<N}. Indeed, putting $\mu_m=\mu^m_1-\mu^m_{2}$ and using the uniform boundedness of $\norm{g(u_m)}_{L^1\left(B_m\cap\pd\R^N_+\right)}$ we can apply Lemma \ref{general_existence_of_limit_functions} to obtain a suitable behaved function $u$ as the limit of the $u_m$. Note that we can also take the limit of the $u_{m,i}$ to obtain suitable functions $u_i$. 
	
	As we argued above, using that $u_{m,i}$ are nondecreasing in $m$ and passing to $cap_{1,p,N}-$ quasi-continuous representatives, we obtain from \eqref{LDC_control_3} that
	\begin{multline}\label{LDC_control_4}
	-cW_{1-\frac{1}{p},p,N-1}\left[ 2\mu_{2}\right](x')\leq -u_{2}(x',t)\leq u(x',t) \\ \leq u_{1}(x',t) \leq cW_{1-\frac{1}{p},p,N-1}\left[2 \mu_{1}\right](x')
	\end{multline}
	$cap_{1,p,N}-q.e.$ in $\R^N$. 
	
	Next, we want to show that for any given $M\in \N$
	\[
	\lim_{m\goesto\infty}\int_{B_{M} \cap \pd\R^N_+} \phi_mg(u_m)dx'= \int_{B_{M} \cap \pd\R^N_+}\phi g(u)dx'
	\]
	for any $\left\lbrace\phi_m\right\rbrace _m$ converging to $\phi$ both $a.e.$ in $B_M$ and weakly in $W_0^{1,p}\left( B_M\right) $ and such that $\phi_m$ are uniformly bounded in $L^\infty\left( B_M\right)$. By Lemma \ref{general_passage_to_trace_nonlinearity_R^N}, it is enough to show that 
	\[
	\int_{\left\lbrace\abs{u}\geq h \right\rbrace \cap \pd\R^N_+\cap B_M}\abs{g(u)}dx'\goesto 0
	\]
	and
	\[
	\int_{\left\lbrace\abs{u_m}\geq h \right\rbrace \cap \pd\R^N_+\cap B_M}\abs{g(u_m)}dx'\goesto 0
	\]
	as $h\goesto \infty$, uniformly on $m$. From \eqref{LDC_control_3}, \eqref{LDC_control_4}, and the hypothesis on $g$ we conclude
	\[\abs{g(u_m)}\leq g(u_{m,1}) + g(u_{m,2})\leq g(u_{1}) + g(u_{2})  \mbox{ and } \abs{g(u)}\leq g(u_{1}) + g(u_{2}) \]
	in $B_M\cap\pd\R^N_+$. By the uniform bounds 
	\[
	\norm{g(u_{m,i})}_{L^1\left(B_m\cap\pd\R^N_+\right) }\leq \norm{\mu_i}_{\mathfrak{M}_b} 
	\]
	we use that $u_{m,i}$ are nondecreasing in $m$ to conclude that 
	\[
	\norm{g(u_{i})}_{L^1\left(\pd\R^N_+\right) }\leq \norm{\mu_i}_{\mathfrak{M}_b} .
	\]
	Then, the desired estimates follow and $g(u)\in L^1\left( \pd\R^N_+\right) $. Thus, by Lemma \ref{general_existence_global_solution} we obtain that $-\lap_p u +2 g(u)\mathcal{H}=2\mu$ in $\R^N$. Applying Theorem \ref{symmetry} and Theorem \ref{thm_on_existence_from_symmetry} we obtain that the restriction of $u$ to $\R^N_+$ is a solution of the desired problem satisfying \eqref{LDC_control_4} (which gives \eqref{eq_control_solution}).
\end{proof}

\begin{remark}
		As in Remark \ref{remark_on _convergence_of_trace}, we observe that it can be shown that, passing to subsequences if necessary, $g(u_m^{n,k})\goesto g(u_m^k)$ and $g(u_m^k)\goesto g(u_m)$ strongly in $L^1\left( B_m\cap\pd\R^N_+\right) $, and $g(u_m)\goesto g(u)$ strongly in $L^1\left( \pd\R^N_+\cap B_M\right)$ for any $M\in \N$.
\end{remark}

Theorem \ref{B-VHV14_trace} can be used to obtain existence of renormalized solutions when $g$ satisfies more explicit conditions. For example, we have the following application to the case when $g(s)$ is dominated by a power function. 

\begin{thm}\label{cor1_supercritical}
	Assume $1<p<N$ and let $g:\R\goesto \R$ be a continuous nondecreasing odd function such that
	\[\abs{g(s)}\leq C\abs{s}^{q} \mbox{ for all } \abs{s}\geq \abs{s_0}\]
	for some $C>0$, $q>p-1$, and $s_0\in \R$. If $\mu \in \mathfrak{M}_b\left( \pd\R^N_+\right)$ is absolutely continuous with respect to $cap_{p-1, \frac{q}{q-p+1},N-1}$ then there exists a renormalized solution to \eqref{equation} with datum $\mu$. 
\end{thm}
\begin{proof}
 Since $\mu$ is absolutely continuous with respect to $cap_{p-1, \frac{q}{q-p+1},N-1}$ so are $\mu_1=\mu^+$, $\mu_2=\mu^-$, and $\mu_i^m=\left. \mu_i\right| _{B_m}$, $i=1,2$. Then for every $m$ we can apply Theorem 2.6 of \cite{B-VHV14} in dimension $N-1$ with $s_1=s_2=q$, $\a=1-\frac{1}{p}$, and $R=4(m+1)$ to obtain nondecreasing sequences $\left\lbrace\mu^m_{i,k}\right\rbrace _k$ of nonnegative measures in $\mathfrak{M}_b\left( \pd\R^N_+\right) $ with compact support in $B_m\cap\pd\R^N_+$ converging to $\mu_i^m$ weakly-$\ast$ in $\mathfrak{M}_b\left( B_m\cap \pd\R^N_+\right)$ and such that $W_{1-\frac{1}{p},p,N-1}^{4(m+1)}\left[ \mu^m_{i,k}\right] \in L^q\left( \R^{N-1}\right)$. It follows immediately that $g\left( c_1W_{1-\frac{1}{p},p,N-1}^{4(m+1)}\left[ \mu^m_{i,k}\right] \right)$ belongs to $L^1\left( \pd\R^N_+\cap B_m\right) $. 
 
 To apply Theorem \ref{B-VHV14_trace} it only remains to show that we can assume $\mu^m_{i,k}\leq \mu^{m+1}_{i,k}$ for each $k\in \N$. For this we note that the approximating sequences $\mu^m_{i,k}$ given by Theorem 2.6 of \cite{B-VHV14}, which are defined in the proof of Theorem 2.5 of \cite{B-VHV14}, can be taken equal to $\sup\left\lbrace \sigma_1,\cdots, \sigma_k\right\rbrace $ for some $\sigma_i$ that approximate, and are bounded by, $\mu^m_i\phi_k$, where $\phi_k$ is a smooth function supported in a neighborhood of $B_{m-\frac{1}{k}}$. Since $\mu^m_i$ coincides with $\mu^{m-1}_i$ in $B_{m-1}$ one can check directly that by redefining $\mu^m_{i,k}$ as $\sup\left\lbrace \mu^m_{i,k}, \mu^{m-1}_{i,k}\right\rbrace $ one obtains approximating sequences with the same properties listed above and that moreover satisfy the desired condition $\mu^m_{i,k}\leq \mu^{m+1}_{i,k}$.
\end{proof}

\begin{remark}
It must be noted that Theorem \ref{cor1_supercritical} agrees with Theorem \ref{existence_subcritical_p<N} in the sense that if $q<q_c$ (see Remark \ref{remark_definition_critical_exponent}) then $cap_{p-1, \frac{q}{q-p+1},N-1}\left( \left\lbrace0\right\rbrace \right) >0$ (see Proposition 2.6.1 of \cite{AdamsHedberg}), and so any bounded Radon measure is admissible according to Theorem \ref{cor1_supercritical}. 

On the other hand, Theorem \ref{existence_subcritical_p<N} gives that any bounded Radon measure is admissible for a wider range of nonlinearities than Theorem \ref{cor1_supercritical}. For example  
\[
g(s)=\frac{\abs{s}^{q_c-1}s}{\left( \ln \left( \abs{s}+C\right) \right)^{1+\e} }
\]
is subcritical if $\e>0$ (and $C$ is chosen large enough) since it satisfies Assumption \ref{assumption_subcritical}, but there is no $q<q_c$ such that $\abs{g(s)}\leq \abs{s}^{q}$ for large values of $s$. Hence, in this case, Theorem \ref{cor1_supercritical} can only be applied with $q\geq q_c$, and so it no longer guarantees that every bounded Radon measure is admissible since, for example, the Dirac measure $\d_0$ is singular with respect to  $cap_{p-1, \frac{q}{q-p+1},N-1}$ precisely when $q\geq q_c$. 

Let us also mention that, since $p-1>0$, the $(N-1)$-Lebesgue measure is absolutely continuous with respect to $cap_{p-1, \frac{q}{q-p+1},N-1}$ (see \cite{AdamsHedberg}), and so every measure in $L^1\left( \pd\R^N_+\right) $ is admissible according to Theorem \ref{cor1_supercritical} (even when $q\geq q_c$). 
\end{remark}

To obtain similar conditions for other nonlinearities we need to introduce some terminology. First we define the Bessel-Lorentz capacities, which can be viewed as a generalization of the Bessel capacities. 

For $1\leq s_1 <\infty$ and $1<s_2\leq \infty$ we denote by $L^{s_1,s_2}\left( \R^N\right) $ the standard Lorentz space (see for example \cite{T98}). Then for $\a>0$ one can define the Lorentz-Bessel capacities 
\[ cap_{\a,s_1,s_2,N}\left( E\right) = \inf\left\lbrace  \norm{f}_{L^{s_1,s_2}\left(\R^N \right) } \ : \ f\geq 0, \ \G_\a \ast f\geq 1 \mbox{ on } E \right\rbrace  \]
where $\G_\a$ is the Bessel kernel of order $\a$ in $\R^N$ (see \cite{AdamsHedberg} or \cite{B-VHV14}). The identification $L^{p,p}\left( \R^N\right) =L^p\left( \R^N\right) $, which holds for $1<p<\infty$, shows that indeed these capacities generalize the standard Bessel capacities.

\begin{thm}\label{cor2_supercritical}
	Let $1<p<N$ and let $g$ be a continuous nondecreasing odd function such that
	\[ \int_1^\infty g(s)s^{-(q+1)}ds<\infty \]
	for some $q>p-1$. If $\mu \in \mathfrak{M}_b\left( \pd\R^N_+\right)$ is absolutely continuous with respect to $cap_{p-1, \frac{q}{q-p+1},1,N-1}$ then there exists a renormalized solution to \eqref{equation} with datum $\mu$.
\end{thm}
\begin{proof}
	 If $\mu$ is absolutely continuous with respect to $cap_{p-1,\frac{q}{q-p+1},1,N-1}$ then so are $\mu_1=\mu^+$, $\mu_2=\mu^-$, and $\mu_i^m=\left. \mu_i\right| _{B_m}$. For every $m$ apply Theorem 2.6 of \cite{B-VHV14} in dimension $N-1$ with $s_1=q$, $s_2=\infty$, $\a=1-\frac{1}{p}$, and $R=4(m+1)$ to obtain nondecreasing sequences $\left\lbrace\mu^m_{i,k}\right\rbrace _k$ of nonnegative measures in $\mathfrak{M}_b\left( \pd\R^N_+\right) $ compactly supported in $B_m\cap\pd\R^N_+$, converging to $\mu_i^m$ weakly-$\ast$ in $\mathfrak{M}_b\left( B_m\cap \pd\R^N_+\right)$ and such that $W_{1-\frac{1}{p},p,N-1}^{4(m+1)}\left[ \mu^m_{i,k}\right] \in L^{q,\infty}\left( \R^{N-1}\right)$. Just as in the proof of Theorem \ref{cor1_supercritical} we may also assume that $\mu^m_{i,k}\leq \mu^{m+1}_{i,k}$. 
	 
	 We observe that $f \in L^{q,\infty}\left( \R^{N-1}\right)$ implies $ \abs{\left\lbrace x' \in B_m \cap \R^{N-1} \ : \ \abs{f(x')}>t\right\rbrace }\leq Ct^{-q}$ for every $t>0$ and where $C$ depends on $\norm{f}_{L^{q,\infty}\left( \R^{N-1}\right) }$ (see \cite{T98}). Then, as in the proof of Theorem \ref{existence_subcritical_p<N}, we can obtain the inequality
	 \begin{multline*}
	 \int_{B_m\cap\pd\R^N_+} g\left( c_1W_{1-\frac{1}{p},p,N-1}^{4(m+1)}\left[ \mu^m_{i,k}\right]\right)dx'\leq g(1)\abs{B_m\cap \pd\R^N_+} + \\ C(q,c_1,c_{k,p,N,m,i})\int_1^\infty g(t)t^{-(q+1)} dt< \infty ,
	 \end{multline*}
	 where $c_{k,p,N,m,i}$ is the $L^{q,\infty}\left( \R^{N-1}\right)$ norm of $W_{1-\frac{1}{p},p,N-1}^{4(m+1)}\left[ \mu^m_{i,k}\right]$. Hence, we finish by applying Theorem \ref{B-VHV14_trace}.
\end{proof}

\begin{remark}
Let us make a few observations regarding this result. It is well-known that $L^{\frac{q}{p-1}}=L^{\frac{q}{p-1},\frac{q}{p-1}}\hookrightarrow L^{\frac{q}{p-1},\infty}$ and that $ L^{\frac{q}{p-1},\infty}$ is the dual of $L^{\frac{q}{q-(p-1)},1}$. Thus 
\[
\G_{p-1} \ast f (0)\geq 1 \Longrightarrow 1 \leq \norm{\G_{p-1}}_{\frac{q}{p-1},\infty}\norm{f}_{\frac{q}{q-(p-1)},1} \leq \norm{\G_{p-1}}_{\frac{q}{p-1}}\norm{f}_{\frac{q}{q-(p-1)},1} ,
\]
which implies that if $\G_{p-1}\in L^{\frac{q}{p-1}}\left( \R^{N-1}\right) $ then $cap_{p-1,\frac{q}{q-p+1},1,N-1}\left( \left\lbrace0\right\rbrace \right) >0$. Since this happens precisely when $q<q_c$, we conclude that if $q<q_c$ then any bounded Radon measure is admissible under the above theorem. In particular, this shows that Theorem \ref{cor2_supercritical} coincides with Theorem \ref{cor1_supercritical} when $\abs{g(s)}\leq C \abs{s}^q$ and $q<q_c$.

Proceeding in a similar way, one can use the fact that $\chi_{B_1^c}\G_{p-1} \in L^{\frac{q}{p-1}}\left( \R^{N-1}\right) $ whenever $0<p-1<q$ to prove (as in Proposition 2.6.1 of \cite{AdamsHedberg}) that the $(N-1)$-Lebesgue measure is absolutely continuous with respect to $cap_{p-1, \frac{q}{q-p+1},1,N-1}$. Thus, every measure in $L^1\left( \pd\R^N_+\right) $ is admissible according to Theorem \ref{cor2_supercritical}. 

Note that, in general, if $\abs{g(s)}\leq C \abs{s}^q$ then Theorem \ref{cor1_supercritical} guarantees existence of a solution to \eqref{equation} provided $\mu$ is absolutely continuous with respect to $cap_{p-1,\frac{q}{q-p+1},N-1}$, while Theorem \ref{cor2_supercritical} guarantees existence if $\mu$ is absolutely continuous with respect to $cap_{p-1,\frac{q'}{q'-p+1},1,N-1}$ for any $q'>q$. 

On the other hand, under the hypotheses on $g$, if the growth condition of the above theorem is satisfied with $q=q_c$ then $g$ satisfies Assumption \ref{assumption_subcritical}, and so it is subcritical. Hence, we expect that the above estimate $cap_{p-1,\frac{q}{q-p+1},1,N-1}\left( \left\lbrace0\right\rbrace \right) >0$, which implies existence for any bounded Radon measure, can be improved to the case $q=q_c$. It can be proven directly that this is true. Indeed, following the ideas above, it is enough to show that $\G_{p-1} \in L^{q_c/(p-1),\infty}\left( \R^{N-1}\right) $. This can be shown by definition using that $\G_{p-1}$ has exponential decay at infinity and that it is controlled by the Riesz kernel of the same order.
\end{remark}

As a final application we consider nonlinearities of exponential type. To this end we define the truncated $\eta$-fractional maximal operator as
\[\boldsymbol{M}^\eta_{s,R,N}\left[ \mu\right](x)= \sup_{0<t<R} \frac{\mu\left(B_t(x)\right)}{t^{N-s} h_\eta(t)}\]
where $0<s<N$, $0<R\leq\infty$, $\eta\geq 0$ and 
\[
h_\eta(t)=\begin{cases}
(-\ln t)^{-\eta} & , 0<t<1/2 \\
(\ln 2)^{-\eta} & , 1/2\leq t .
\end{cases} 
\]
Then we have the following result. 

\begin{thm}
	Let $1<p<N$ and let $g$ be a continuous nondecreasing odd function such that
	\[g(\abs{s})\leq e^{\tau\abs{s}^\l}-1 \mbox{ for all } s\geq s_0\]
	for some $\tau>0$, $\l\geq 1$, and $s_0\in \R$. Let $\mu \in \mathfrak{M}_b\left( \pd\R^N_+\right)$ be such that
	$\mu=f + \nu_1 - \nu_2$, where $f \in L^1\left( \pd\R^N\right)$ and $\nu_i \in \mathfrak{M}_b\left( \pd\R^N_+\right)$, $i=1,2$,  are nonnegative. There exists $M(N,p,\tau,\l)>0$ such that if 
	 \[ \norm{\boldsymbol{M}^{\frac{(p-1)(\l-1)}{\l}}_{p-1,\infty,N-1}\left[ \nu_i\right]}_{L^\infty\left( \R^{N-1}\right)} \leq M \]
	then there exists a renormalized solution to \eqref{equation} with datum $\mu$.
\end{thm}
\begin{proof}
	Let $f=f_1-f_2$ with $f_i\geq 0$, define $\mu_i=f_i+\nu_i$, and let $\mu_i^m=f_i^m + \nu_i^m$ be its restriction to $B_m$. Define $\mu^m_{i,k}=\left( T_k(f_i^m)+\nu_i^m\right) \chi_{B_{m-\frac{1}{k}}}$. Then $\mu^m_{i,k}$ are nonnegative, nondecreasing on $k$, compactly supported in $B_m\cap\pd\R^N_+$, and moreover $\mu^m_{i,k}\leq \mu^{m+1}_{i,k}$. It is also clear that $\mu^m_{i,k}\goesto  \mu^m_i$ weakly-$\ast$ in $\mathfrak{M}_b\left( B_m\cap\pd\R^N_+\right) $. Hence, to apply Theorem \ref{B-VHV14_trace} it remains to show $g\left( c_1W_{1-\frac{1}{p},p,N-1}^{4(m+1)}\left[ \mu^m_{i,k}\right] \right) \in L^1\left( B_m\right) $. 
	
	Let us first note that since $p>1$ and $p(1-\frac{1}{p})>0$ we have
	\[
	W^{4(m+1)}_{1-\frac{1}{p},p,N-1}\left[ T_k(f_i^m)\chi_{B_{m-\frac{1}{k}}}\right](x')\leq C(k,N,p,m) .
	\]
	On the other hand, it holds that for every $s\geq 1$, $\e>0$ there exists $C=C(\e,s)$ such that if $a,b\geq0$ then $(a+b)^s\leq a^sC+(1+\e)b^s$. Using this twice we conclude
	\[
	\left( W^{4(m+1)}_{1-\frac{1}{p},p,N-1}\left[ \mu^m_{i,k}\right] (x')\right)^\l \leq C(k,N,p,m,\l,\e) + (1+\e) \left(  W^{4(m+1)}_{1-\frac{1}{p},p,N-1}\left[ \nu^m_{i}\right] (x')\right) ^\l ,
	\]
	for some $\e>0$ to be fixed later, and so we have
	\begin{multline*}
	\exp\left(\tau\left(  c_1W^{4(m+1)}_{1-\frac{1}{p},p,N-1}\left[ \mu^m_{i,k}\right] (x')\right)^\l\right)\\ \leq C(\tau,k,N,p,m,\l,\e)\exp\left(\tau(1+\e)\left(  c_1W^{4(m+1)}_{1-\frac{1}{p},p,N-1}\left[ \nu^m_{i}\right] (x')\right)^\l\right)
	\end{multline*}
	since $c_1=c_1(N,p)$. 
	Now, an application of Theorem 2.4 of \cite{B-VHV14} in dimension $N-1$ with $\a=1-\frac{1}{p}$, $\eta=\frac{(p-1)(\l-1)}{\l}$, $r=m$, and $R=4(m+1)$, shows that there exists $0<\d_0(p,\l)$ such that 
	\[
	\int_{B_{m}\cap\pd\R^N_+}\exp\left( \d \left( W^{4(m+1)}_{1-\frac{1}{p},p,N-1}\left[ \nu^m_{i}\right]\right)^\l\norm{\boldsymbol{M}_{p-1,4(m+1),N-1}^{\frac{(p-1)(\l-1)}{\l}}\left[ \nu^m_i\right] }^\frac{\l}{1-p}_{L^\infty\left( B_m\cap\pd\R^N_+\right)}\right)dx' <\infty
	\] 
	for any $\d\in (0,\d_0)$. Hence, if we choose any $M(p,\l,N,\tau)$ such that 
	\[
	 M <\left( \frac{\d_0}{\tau c_1^\l}\right) ^\frac{p-1}{\l} 
	\]
	then by hypothesis and the fact that 
	\[
	\norm{\boldsymbol{M}_{p-1,4(m+1),N-1}^{\frac{(p-1)(\l-1)}{\l}}\left[ \nu^m_i\right] }_{L^\infty\left( B_m\cap\pd\R^N_+\right)}\leq 	\norm{\boldsymbol{M}_{p-1,\infty,N-1}^{\frac{(p-1)(\l-1)}{\l}}\left[ \nu_i\right] }_{L^\infty\left( \pd\R^N_+\right)}
	\]
	we conclude that there exist $\e>0$ such that
	\[
	\tau(1+\e)c_1^\l \leq \d\norm{\boldsymbol{M}_{p-1,4(m+1),N-1}^{\frac{(p-1)(\l-1)}{\l}}\left[ \nu^m_i\right] }_{L^\infty\left( B_m\cap\pd\R^N_+\right)}^{\frac{\l}{1-p}}
	\]
	for some $\d\in (0,\d_0)$ and so 
	\[
	\int_{B_{m}\cap\pd\R^N_+}\exp\left( \tau(1+\e)\left( c_1W^{4(m+1)}_{1-\frac{1}{p},p,N-1}\left[ \nu^m_{i}\right]\right)^\l\right) dx'<\infty
	\]
	which concludes the proof.
\end{proof}

\begin{remark}
	It is immediate that the above theorem guarantees existence for data in $L^1\left( \pd\R^N_+\right) $.
	
 	When $\l=1$, i.e. $g(\abs{s})\leq e^{\tau\abs{s}}-1$, the condition imposed on $\nu_i$ reads
	\[
	\sup_{x\in \pd\R^N_+}\sup_{t>0} \frac{\nu_i\left( B_t(x)\right) }{t^{N-p}}\leq M .
	\]
	This condition can be expressed in terms of the Riesz capacities (See Chapter \ref{Capacities}). Indeed, it is known that in the case $1<p<N$ it holds $t^{N-p}= C cap_{I_{1-\frac{1}{p}},p,N-1}\left( B_t(x)\right) $ for some $C>0$ independent of $t$ (see Chapter 5 of \cite{AdamsHedberg}), and so the above condition is equivalent to
	\[
	\nu_i\left( B_t(x)\right) \leq CM cap_{I_{1-\frac{1}{p}},p,N-1}\left( B_t(x)\right)
	\]
	for every $x\in \pd\R^N_+$ and $t>0$. 
\end{remark}

\chapter{Nonlinear problems with source}\label{source}

In this chapter we consider nonnegative solutions to the following problem with source
\be
\begin{cases}\label{eq_source}
	-\lap_pu=0 & \mbox{ in } \R^N_+ \\
	\abs{\nabla u}^{p-2}u_\n = \mu + u^q & \mbox{ on } \pd\R^N_+, 
\end{cases} 
\ee
where $1<p<N$, $p-1<q$, and $\mu \in \mathfrak{M}_b\left(\pd\R^N_+\right)$ is nonnegative.

We begin obtaining necessary conditions for existence of solutions. Then, we show that under a smallness assumption on the constants involved, these conditions imply existence of solutions. Lastly, we use these conditions to show nonexistence results and also to characterize removable sets.

\section{Necessary conditions for existence} 

To obtain necessary conditions for existence of solutions to \eqref{eq_source} we follow the ideas in \cite{PV08}. 

In order to state our first result we need to introduce the \textit{Riesz potential} $I_{\a,N}$ of order $\a$, $0<\a<N$, on $\R^N$, of a nonnegative Radon measure $\mu$ by 
\[
I_{\a,N}[\mu](x)=c(N,\a)\int_{\R^N}\abs{x-y}^{\a-N}d\mu(y)
\]
where $c(N,\a)$ is a normalized constant. We recall that the Riesz capacities were defined in Chapter \ref{Capacities}, while the Wolff potential $W_{\a,s,N}\left[ \cdot \right]$ was defined in Section \ref{supercritical}.

\begin{thm}\label{Necessity_supercritical_source_I}
	Let $1<p<N$ and $p-1<q$. Let $\mu$ in $\mathfrak{M}_b\left( \pd\R^N_+\right)$ be nonnegative and suppose there exists a nonnegative renormalized solution to \eqref{eq_source}. Then
	\be \label{testing_inequality_I}
	\int_B \left( I_{p-1,N-1}\left[ \mu_B\right] \right)^{\frac{q}{p-1}} dx'\leq C(N,p,q)  \mu\left( B\right) 
	\ee
	holds for all balls $B\subset \pd\R^N_+\simeq \R^{N-1}$ (where $\mu_B$ is the restriction of $\mu$ to $B$). 
\end{thm}
\begin{proof}
	We know by Remark \ref{renormalized_solutions_are_local} that if $u$ solves \eqref{eq_source} then $\bar{u}$, the extension of $u$ to $\R^N$ by even reflection across $\pd\R^N_+$, is a local renormalized solution to 
	\[
	-\lap_p \bar{u} = 2 \bar{u}^q\mathcal{H} + 2 \mu \mbox{ in } \R^N .
	\]
	Let $\w=2 \bar{u}^q\mathcal{H} + 2 \mu$. Combining Theorems 4.3.2 and 4.2.5 of \cite{Veron16}, we obtain that $\bar{u}$ coincides $a.e.$ with a $p-$ superharmonic function $\tilde{u}$ satisfying 
	\[
	W_{1,p,N}\left[ \w\right] \leq C(N,p)\tilde{u}.
	\]
	By Remark \ref{remark_on_cap_equality_with_psuperharmonic} we can conclude that
	\[
	W_{1,p,N}\left[ \w\right] \leq C(N,p)\bar{u}
	\]
	$cap_{1,p,N}-q.e.$ in $\R^N$ and so, by Proposition \ref{trace&extensionE}, $\mathcal{H}-$ $a.e.$ Thus, for any dyadic cube $P\subset \pd\R^N_+$ (i.e., $P=2^j\left( k + \left[ 0,1\right)^{N-1}\right) $ for some $j\in \mathbb{Z}$ and $k\in \mathbb{Z}^{N-1}$) we have
	\[
	\w(P)\geq \int_P2\bar{u}^q dx'\geq C(N,p)\int_PW_{1,p,N}\left[ \w\right]^q dx' = C(N,p)\int_PW_{1-\frac{1}{p},p,N-1}\left[ \w\right]^q dx'. 
	\]
	Using that, for any $\a>0$, $p>1$, and any $N$,
	\[
	W_{\a,s,N}\left[ \w\right] \sim \sum_{Q\subset P} \left( \frac{\w(Q)}{\abs{Q}^{1-\frac{\a s}{N}}}\right) ^{\frac{1}{s-1}}\chi_Q
	\]
	with $\abs{Q}$ the $N-$ dimensional measure of $Q$, and where the sum is taken over all dyadic cubes $Q$ contained in $P$ (see \cite{AdamsHedberg}), we conclude
	\[
	\int_P \left( \sum_{Q\subset P} \left( \frac{\w(Q)}{\abs{Q}^{1-\frac{p-1}{N-1}}}\right) ^{\frac{1}{p-1}}\chi_Q\right) ^q dx'\leq C(N,p) \w(P) .
	\]
	By Proposition 3.1 of \cite{PV08} the above implies
	\[
	\sum_{Q\subset P} \left( \frac{\w(Q)}{\abs{Q}^{1-\frac{p-1}{N-1}}}\right) ^{\frac{q}{p-1}}\abs{Q} \leq C(N,p) \w(P) 
	\]
	which, by an application of Theorem 3 of \cite{V96}, yields 
	\[
	\norm{\sum_{Q}\frac{f(Q)}{\abs{Q}^{1-\frac{p-1}{N-1}}}\chi_Q}_{L^{\frac{q}{q-(p-1)}}(d\w)}\leq C(N,p,q)\norm{f}_{L^{\frac{q}{q-(p-1)}}}
	\]
	for any nonnegative $f \in L^{\frac{q}{q-(p-1)}}\left( \R^{N-1}\right) $, where $f(Q)=\int_Q fdx'$ and the sum is taken over all dyadic cubes $Q$. Since 
	\[
	I_{\a,N}\left[ f\right] \sim \sum_{Q}\frac{f(Q)}{\abs{Q}^{1-\frac{\a}{N}}}\chi_Q
	\]
	and $\w\geq 2\mu$ we obtain
	\[
	\norm{I_{p-1,N-1}(f)}_{L^{\frac{q}{q-(p-1)}}(d\mu)}\leq C(q,p)\norm{I_{p-1,N-1}(f)}_{L^{\frac{q}{q-(p-1)}}(d\w)}\leq C(N,p,q)\norm{f}_{L^{\frac{q}{q-(p-1)}}}
	\]
	for any $f \in L^{\frac{q}{q-(p-1)}}\left( \R^{N-1}\right) $. Hence, $I_{p-1,N-1} :L^{\frac{q}{q-(p-1)}}\goesto L^{\frac{q}{q-(p-1)}}(d\mu)$ is a bounded linear operator and so its dual satisfies
	\[
	\norm{\left\langle I_{p-1,N-1}^\ast,g\right\rangle }_{L^{\frac{q}{p-1}}}\leq C(N,p,q)\norm{g}_{L^{\frac{q}{p-1}}(d\mu)} 		
	\]
	for any $g\in L^{\frac{q}{p-1}}(d\mu)$. Taking $g=\chi_B$ we obtain \eqref{testing_inequality_I}.
\end{proof}

\begin{remark}\label{remark_on_necessary_conditions}
	It is known that \eqref{testing_inequality_I} is equivalent with the condition
		\be \label{testing_inequality_cap_I}
		\mu\left( K\right) \leq C(N,p,q)cap_{I_{p-1},\frac{q}{q-(p-1)},N-1}\left( K\right) 
		\ee
		for all compact sets $K\subset \pd\R^N_+\simeq \R^{N-1}$.
	The proof of this equivalence, which we will use in the following sections, can be found in \cite{V99}. On the other hand, it is known that \eqref{testing_inequality_I} implies
	\[\int_{\R^{N-1}} \left( I_{p-1,N-1}\left[ \mu_B\right] \right)^{\frac{q}{p-1}} dx'\leq C(N,p,q)  \mu\left( B\right) \]
	(see \cite{VW98} or \cite{PV08}). By Proposition 5.1 of \cite{PV08} 
	 \[
	 	\int_{\R^{N-1}} \left( I_{p-1,N-1}\left[ \mu_B\right] \right)^{\frac{q}{p-1}} dx' \sim \int_{\R^{N-1}} \left( W_{1-\frac{1}{p},p,N-1}\left[ \mu_B\right] \right)^{q}dx' 
	 \]
	 so we see that \eqref{testing_inequality_I} implies
		\be \label{testing_inequality_W}
		\int_{\R^{N-1}} \left( W_{1-\frac{1}{p},p,N-1}\left[ \mu_B\right] \right)^{q}dx' \leq C(N,p,q)  \mu\left( B\right) 
		\ee
	for all balls $B\subset \pd\R^N_+\simeq \R^{N-1}$. Note that, by Monotone Convergence, the above condition implies that if $\mu\in\mathfrak{M}_b\left( \pd\R^N_+\right) $ then $W_{1-\frac{1}{p},p,N-1}\left[ \mu\right] \in L^q\left( \pd\R^N_+\right) $.
\end{remark}

As we will see, condition \eqref{testing_inequality_I} is `almost' sufficient to obtain existence of a solution. However, because of the method we use to show existence, it is convenient to work with another necessary condition which is actually a consequence of \eqref{testing_inequality_I}. 

\begin{thm}\label{Necessity_supercritical_source_W}
	Let $1<p<N$, $p-1<q$, and let $\mu$ in $\mathfrak{M}_b\left( \pd\R^N_+\right)$ be nonnegative. Then condition \eqref{testing_inequality_I} implies $W_{1-\frac{1}{p},p,N-1}\left[ \mu\right] \in L^q\left( \pd\R^N_+\right) $ and
	\be\label{condition_potential}
	W_{1-\frac{1}{p},p,N-1}\left[ \left( W_{1-\frac{1}{p},p,N-1}\left[ \mu\right] \right) ^q\right] \leq C_1 W_{1-\frac{1}{p},p,N-1}\left[ \mu\right] \ \mbox{ $a.e.$ in } \pd\R^N_+ ,
	\ee
	for some nonnegative constant $C_1$ depending on $p$, $q$, and $N$.
\end{thm}
\begin{proof}
 By Remark \ref{remark_on_necessary_conditions}, it is enough to check that \eqref{testing_inequality_W} implies \eqref{condition_potential}.
 To this end we decompose the Wolff potential as $W_{1-\frac{1}{p},p,N-1}[\mu]= U_r\mu + L_r \mu$, where
\begin{align*}
U_r\mu(x) &= \int_0^r\left( \frac{\mu\left( B_t(x)\right) }{t^{N-p}}\right) ^\frac{1}{p-1}\frac{dt}{t} \\
L_r\mu(x) &= \int_r^\infty\left( \frac{\mu\left( B_t(x)\right) }{t^{N-p}}\right) ^\frac{1}{p-1}\frac{dt}{t} \\			
\end{align*}
for any $r>0$. Setting 
\begin{align*}
\nu&=\left( W_{1-\frac{1}{p},p,N-1}\left[ \mu\right] \right) ^q \\
\tau_r&=\left( U_r\mu \right) ^q\\
\l_r&=\left( L_r\mu \right) ^q\\
\end{align*}
we see that $\nu\leq C(q)\left( \tau_r + \l_r\right) $ for any $r>0$. Note that these are $L^1\left( \pd\R^N_+\right) $ measures.

Now fix any $x\in \pd\R^N_+\simeq \R^{N-1}$ and write for simplicity $B_r=B_r(x)\subset \R^{N-1}$. If $y\in B_r$ and $0 < t\leq r$ then $B_t(y)\subset B_{2r}$, and so $U_r\mu=U_r\mu_{B_{2r}}$ in $B_r$. Hence, by \eqref{testing_inequality_W} we have
\begin{align*}
\tau_r\left( B_r\right) &=\int_{B_r}\left( U_r\mu\right) ^qdx'=\int_{B_r} \left( U_r\mu_{B_{2r}}\right)^qdx' \\ &\leq \int_{B_r} \left( W_{1-\frac{1}{p},p,N-1}\left[ \mu_{B_{2r}}\right] \right)^qdx'\leq C\mu\left( B_{2r}\right) 
\end{align*}
and so
\be\label{estimate_tau}
W_{1-\frac{1}{p},p,N-1}\left[ \tau_r\right] (x)\leq \int_0^\infty\left( \frac{\mu\left( B_{2r}\right) }{r^{N-p}}\right)^\frac{1}{p-1} \frac{dr}{r}=CW_{1-\frac{1}{p},p,N-1}\left[ \mu\right] (x) .
\ee
Next, we study the rate of decay of $\l_r$ as function of $r$. If $y\in B_t$ and $s\geq 2t$ then $B_t\subset B_s(y)$ and so
\begin{align*}
W_{1-\frac{1}{p},p,N-1}\left[ \mu_{B_t}\right] (y)&\geq \int_{2t}^\infty \left( \frac{\mu\left( B_t\cap B_s(y)\right) }{s^{N-p}}\right)^\frac{1}{p-1}\frac{ds}{s} \\
&\geq C\mu\left( B_t\right)^\frac{1}{p-1}t^\frac{p-N}{p-1} .
\end{align*}
Comparing the above with \eqref{testing_inequality_W} it follows that
\[
\mu\left( B_t\right) \leq Ct^{N-1-\frac{q(p-1)}{q-(p-1)}} 
\]
and then 
\[
L_r\mu \leq Cr^\frac{1-p}{q-(p-1)} .
\]
If $y\in B_r$ and $t\geq r$ then $B_t(y)\subset B_{2t}$ and so
\[
L_r\mu(y)\leq\int_r^\infty\left( \frac{\mu\left( B_{2t}\right) }{r^{N-p}}\right) ^\frac{1}{p-1}\frac{dr}{r} \leq L_r\mu(x) 
\]
which gives 
\[
\l_r\left( B_r\right) =\int_{B_r} \left( L_r\mu\right) ^qdx'\leq C\left( L_r\mu(x)\right) ^q r^{N-1} .
\]
Combining the above estimates and using integration by parts we get
\begin{align*}
\int_0^\infty \left( \frac{\l_r\left( B_r\right) }{r^{N-p}}\right) ^\frac{1}{p-1}\frac{dr}{r} & \leq C \int_0^\infty \left( L_r\mu(x)\right) ^\frac{q}{p-1}dr \\
&=C\int_0^\infty \left( L_r\mu(x)\right)^{\frac{q}{p-1}-1}\left( \frac{\mu\left( B_r\right) }{r^{N-p}}\right) ^\frac{1}{p-1}dr \\
&\leq  C\int_0^\infty \left( \frac{\mu\left( B_r\right) }{r^{N-p}}\right) ^\frac{1}{p-1}\frac{dr}{r} \\
\end{align*}
that is, 
\be\label{estimate_lambda}
W_{1-\frac{1}{p},p,N-1}\left[ \l_r\right](x)\leq C W_{1-\frac{1}{p},p,N-1}\left[ \mu\right](x) .
\ee
By combining \eqref{estimate_tau} and \eqref{estimate_lambda} we conclude
\[
W_{1-\frac{1}{p},p,N-1}\left[ \nu\right](x)\leq C W_{1-\frac{1}{p},p,N-1}\left[ \mu\right](x)
\]
which is the desired estimate \eqref{condition_potential}.
\end{proof}

\section{Sufficient conditions for existence} 

Our strategy for solving problem \eqref{eq_source} is to combine the techniques developed in \cite{PV08}, where the authors study the existence of $p-$ superharmonic solutions to 
\[
	-\lap_pu= u^q + \mu  \mbox{ in } \R^N, \\
\]
with our symmetry and existence results of Chapter \ref{Symmetry}. The results of \cite{PV08} are based on a careful study of the Wolff potential, and the existence of solutions is guaranteed under any one of some equivalent conditions, among which is that the measure $\mu$ satisfies 
\be\label{condition_PV09}
W_{1,p,N}\left[ \left( W_{1,p,N}\left[ \mu\right] \right) ^q\right] \leq C_0 W_{1,p,N}\left[ \mu\right]<\infty  \mbox{ $a.e.$ in } \R^N
\ee
for some small enough constant $C_0=C_0(N,p,q)$.

Unlike the problem with absorption, we do not use directly the existence result of \cite{PV08} to construct a global solution. Instead, we define a recursive sequence of solutions to 
\be
\begin{cases}
	-\lap_pu_m=  u_{m-1}^q\mathcal{H} +  \mu_{m} & \mbox{ in } B_m \\
	u_m=0 & \mbox{ on } \pd B_m, 
\end{cases}
\ee
where $u_0=0$ and $ \mu_m(E)=\mu\left( E\cap B_m \right) $, and then take limit as $m\goesto \infty$. In this way we dispense with the need to define a sequence of nonlinearities $g_n$ converging to $u^q$ (as was done in the previous chapter). As we show in the next theorem, this method gives a solution to \eqref{eq_source} under a natural adaptation of condition \eqref{condition_PV09}. But before, we need the following lemma.

\begin{lem}\label{modification_L.6.9_PV08}
	Let $\mu$, $\nu\in \mathfrak{M}_b\left( \O\right)$ be nonnegative measures and suppose $\mu\leq \nu$. Let $\O'\subset \subset \O$ and let $u$ be a renormalized solution to 
	\[
	\begin{cases}
		-\lap_pu=  \mu & \mbox{ in } \O' \\
		u=0 & \mbox{ on } \pd \O'. 
	\end{cases}
	\]
	then there exists a renormalized solution $v$ to 
	\[
	\begin{cases}
		-\lap_pv=  \nu & \mbox{ in } \O \\
		v=0 & \mbox{ on } \pd \O
	\end{cases}
	\]
	such that $u\leq v$ $a.e.$ in $\O'$. 
\end{lem}
\begin{proof}
	The lemma (and its proof) is a slight modification of Lemma 6.9 in \cite{PV08}, so we omit some details. Let $u_k=T_k(u)$. Then $u_k$ solves $-\lap_pu_k=  \mu_0\chi_{\{\abs{u}<k\}} + \l_k^+$ in $\O'$, $u_k=0$ on $\pd \O'$, where $\l_k^+$ is a nonnegative measure (see Remark \ref{remark_on_equivalent_definitions_DMOP99}). Let $v_k$ solve $-\lap_pv_k=  \mu_0 + \l_k^+ + \nu-\mu$ in $\O$, $v_k=0$ on $\pd \O$. By the stability results of \cite{DMOP99}, passing to a subsequence, $v_k$ converges $a.e.$ to a function $v$ solving the desired equation. Since $u$ is $a.e.$ finite, the result follows if we can show that $u_k\leq v_k$ $a.e.$ in $\O'$. For this we can proceed exactly as in the proof of Lemma 6.8 in \cite{PV08}. We only remark that $\min\left( \left( u_k-T_{h+M}(v_k)\right) ^+,h\right) $, where $h>0$ and $M=\sup_{\O'}u_k$, belongs to $W_0^{1,p}\left( \O'\right) $ because $v_k\in W^{1,p}\left( \O'\right) $ is nonnegative and $u_k\in W_0^{1,p}\left( \O'\right) $.
\end{proof}

\begin{thm}\label{Sufficiency_supercritical_source}
	Let $1<p<N$ and $p-1<q$. Let $\mu$ be a nonnegative measure in $\mathfrak{M}_b\left( \pd\R^N_+\right)$ satisfying $W_{1-\frac{1}{p},p,N-1}\left[ \mu\right] \in L^q\left( \pd\R^N_+\right) $ and condition \eqref{condition_potential} with 
	\[
	C_1\leq \left( \frac{q-p+1}{qc(N,p)C(p)2^\frac{1}{p-1}}\right) ^{\frac{q}{p-1}}\left( \frac{p-1}{q-p+1}\right) 
	\]
	where $C(p)=\max\left\lbrace 1,2^{\frac{2-p}{p-1}}\right\rbrace $ and $c(N,p)$ is the constant in Theorem \ref{B-VHV14}. Then there exists a nonnegative renormalized solution to \eqref{eq_source} satisfying 
	\be\label{estimate_target}
	u(x',x_N)\leq \left( \frac{qc(N,p)C(p)2^\frac{1}{p-1}}{q-p+1} \right) W_{1-\frac{1}{p},p,N-1}\left[\mu\right](x') \mbox{ in } \O\cap \overline{\R^N_+},
	\ee
	where $\O$ is a set of the form $\O=\O_1\cap \left( \O_2\times \R\right) $, $\O_2\subset \R^{N-1}$, with $cap_{1,p,N}\left( \O_1^c\right) =0$ and $\abs{\O_2^c}=0$. In particular, the above estimate holds $a.e.$ in any hyperplane $\R^{N-1}\times \{t\}$. 
\end{thm}
\begin{proof}
	Let $u_1$ be a renormalized solution of 
	\be
	\begin{cases}
		-\lap_pu_1= 2 \mu_1 & \mbox{ in } B_1 \\
		u_1=0 & \mbox{ on } \pd B_1 ,
	\end{cases} 
	\ee
	where $ \mu_1(E)=\mu\left( E\cap B_1\right) $. Such a function exists by, for example, the results in \cite{DMOP99}. By testing against $T_k(min(u_1,0))$ one can see that $\mu\geq 0$ implies that $u_1$ is nonnegative (see Remark 6.5 of \cite{PV08}). By Lemma \ref{trace_potential_estimate}, $u_1$ satisfies
	\[
	u_1\leq c(N,p)W_{1,p,N}^4\left[ 2 \mu_1\right] \leq c(N,p)W_{1,p,N}\left[ 2 \mu\right] 
	\]
	$cap_{1,p,N}-q.e.$ in $B_1$, where $c(N,p)$ is the constant in Theorem \ref{B-VHV14}. As observed in \eqref{identity_trace_potential}
	\[
	W_{1,p,N}\left[2  \mu\right] = W_{1-\frac{1}{p},p,N-1}\left[2 \mu\right] \mbox{ in } \pd\R^N_+
	\]
	so $u_1 \in L^q\left( B_1\cap\pd\R^N_+\right) $ and, by \eqref{bound_trace_potential}, 
	\[
	u_1 (x',x_N)\leq c(N,p)W_{1-\frac{1}{p},p,N-1}\left[2 \mu\right] (x')
	\]
	$cap_{1,p,N}-q.e.$ in $B_1$, and in the whole $\R^N$ if we extend the function by zero outside of $B_1$. Suppose $m>1$, $m\in \N$, and $u_m$ is a renormalized solution to
	\be\label{inductive_solution}
	\begin{cases}
		-\lap_pu_m= 2 u_{m-1}^q\mathcal{H} + 2 \mu_m & \mbox{ in } B_m \\
		u_m=0 & \mbox{ on } \pd B_m
	\end{cases}
	\ee
	where $ \mu_m(E)=\mu\left( E\cap B_m \right) $ and $u_{m-1}  \in L^q\left(\pd\R^N_+\right) $ is nonnegative, supported in $B_{m-1}$, and satisfies
	\[
	u_{m-1}(x',x_N)\leq \a_{m-1} W_{1-\frac{1}{p},p,N-1}\left[ 2\mu\right](x')  \mbox{ for every } (x',x_N)\in\O\cap B_{m-1}
	\]
	for some constant $\a_{m-1}$, where $\O$ is a set of the form $\O=\O_{1,m-1}\cap \left( \O_{2}\times \R\right) $ with $cap_{1,p,N}\left( \O_{1,m-1}^c\right) =0$ and $\O_{2}\subset \R^{N-1}$ the set where $W_{1-\frac{1}{p},p,N-1}\left[ \mu\right](x')$ is finite and condition \eqref{condition_potential} holds (note that $\abs{\O_2^c}=0$). Since the measure $ 2u_{m-1}^q\mathcal{H} + 2 \mu_m$ is nonnegative we have $u_m\geq 0$ and, again by Lemma \ref{trace_potential_estimate}, 
	\[
	u_m\leq c(N,p) W_{1,p,N}^{4m}\left[ 2 u_{m-1}^q\mathcal{H} + 2 \mu_m\right] \ cap_{1,p,N}-q.e. \mbox{ in } B_m.
	\]
	By definition of the Wolff potential one can see that 
	\[
	 W_{1,p,N}^{4m}\left[ 2 u_{m-1}^q\mathcal{H} + 2 \mu_m\right] \leq C(p)\left(  W_{1,p,N}^{4m}\left[2  u_{m-1}^q\mathcal{H} \right] +  W_{1,p,N}^{4m}\left[ 2 \mu_m\right]\right) 
	\] 
	where $C(p)=\max\left\lbrace 1,2^{\frac{2-p}{p-1}}\right\rbrace $. Therefore, we can use that $ u_{m-1}^q\mathcal{H}$ and $ \mu_m$ are supported in $\pd\R^N_+$, together with \eqref{bound_trace_potential}, the monotonicity of the Wolff potential, assumption \eqref{condition_potential}, and the induction hypothesis, to compute that
	\begin{align*}
	u_m (x',x_N)& \leq C(p)c(N,p)\left(  W_{1,p,N}^{4m}\left[ 2 u_{m-1}^q\mathcal{H} \right] +  W_{1,p,N}^{4m}\left[ 2 \mu_m\right]\right)(x',x_N) \\
		& \leq C(p)c2^{\frac{1}{p-1}}\left(  W_{1,p,N}^{4m}\left[  u_{m-1}^q\mathcal{H} \right] +  W_{1,p,N}^{4m}\left[  \mu_m\right]\right)(x',0) \\
		& \leq C(p)c2^{\frac{1}{p-1}}\left(  W_{1,p,N}\left[ \left( \a_{m-1} W_{1-\frac{1}{p},p,N-1}\left[2\mu\right] \right) ^q \mathcal{H}\right] + W_{1,p,N}\left[  \mu\right]\right)(x',0) \\
		&\leq C(p)c2^{\frac{1}{p-1}}\left(  \left( 2^\frac{1}{p-1} \a_{m-1}\right) ^{\frac{q}{p-1}}C_1 + 1 \right)  W_{1,p,N}\left[ \mu\right](x',0) \\
		&= C(p)c\left(  \left( 2^\frac{1}{p-1}\a_{m-1}\right) ^{\frac{q}{p-1}}C_1 + 1 \right)  W_{1-\frac{1}{p},p,N-1}\left[2 \mu\right] (x')
	\end{align*}
	for every $(x',x_N)\in \O\cap B_m$, where $\O=\O_{1,m}\cap \left( \O_2\times \R\right) $, $\O_2\subset \R^{N-1}$ is as described above, and $\O_{1,m}$ is the intersection of $\O_{1,m-1}$ with the set where the first inequality holds. Note that $cap_{1,p,N}\left( \O_{1,m}^c\right) =0$. Hence, by induction starting with $\a_1=c(N,p)$, we obtain a sequence of nonnegative functions $\left\lbrace u_m\right\rbrace _m \subset L^q\left( \pd\R^N_+\right) $ such that
	\[
	u_m(x',x_N)\leq \a_m W_{1-\frac{1}{p},p,N-1}\left[2\mu\right](x') \mbox{ in } \O\cap B_m,
	\]  
	with $\O$ as described above, and where 
	\[
	\a_m= C(p)c(N,p)\left(  \left( 2^\frac{1}{p-1}\a_{m-1}\right) ^{\frac{q}{p-1}}C_1 + 1 \right) .
	\]
	Since $C(p)\geq 1$, it is easy to show by induction that the assumption
	\[
	C_1\leq \left( \frac{q-p+1}{qc(N,p)C(p)2^\frac{1}{p-1}}\right) ^{\frac{q}{p-1}}\left( \frac{p-1}{q-p+1}\right) 
	\]
	implies that the sequence $\left\lbrace\a_m\right\rbrace _m$ satisfies
	\[
	\a_m\leq M:=\frac{qc(N,p)C(p)}{q-p+1} \mbox{ for all } m\in \N
	\]
	and so we obtain 
	\[
	u_m(x',x_N)\leq M W_{1-\frac{1}{p},p,N-1}\left[2\mu\right](x') \mbox{ in } \O\cap B_m .
	\] 
	Note that we may assume $u_{m-1}\leq u_m$ $cap_{1,p,N}-q.e.$ in $\R^N$. Indeed, assume $u_{m-1}$ is a solution of \eqref{inductive_solution} such that $u_{m-2}\leq u_{m-1}$ $cap_{1,p,N}-q.e.$ in $\R^N$. Set $\nu_m= 2 u_{m-1}\mathcal{H} + \mu_m$. Then, since $\nu_{m-1} \leq \nu_{m}$, Lemma \ref{modification_L.6.9_PV08} shows that we can obtain a renormalized solution $u_m$ of \eqref{inductive_solution} such that $u_{m-1}\leq u_{m}$ $a.e.$ in $B_{m-1}$. Extending by zero and using $cap_{1,p,N}-$ quasi-continuous representatives we conclude $u_{m-1}\leq u_m$ $cap_{1,p,N}-q.e.$ in $\R^N$.
	
	Now, since these solutions are nonnegative, we may identify them with their $p-$ superharmonic representatives and conclude $u_{m-1}\leq u_m$ everywhere in $\R^N$ (see Remark \ref{remark_on_cap_equality_with_psuperharmonic}). Then, by Lemma 7.3 of \cite{HKM} $u=\sup_mu_m$ defines a $p-$ superharmonic function which, by Theorem 10.9 of \cite{HKM}, is $cap_{1,p,N}-$ quasi-continuous in $\R^N$ (note that $u$ is finite in $\O$ and $\abs{\O^c\cap B_m}=0$ for every $m\in \N$). Moreover, it follows that $u\in L^q\left( \pd\R^N_+\right) $ and $u_m^q\goesto u^q$ in $L^1\left( \pd\R^N_+\right) $. Notice that $\left\lbrace u_m\right\rbrace _m$ is uniformly bounded in $L^q\left(\pd\R^N_+\right)$. Hence, by Lemma \ref{general_existence_of_limit_functions} $u$ satisfies properties $(1)$, $(2)$, and $(3)$ in the statement of that lemma. Note also that $u$ satisfies the desired estimate \eqref{estimate_target}.
	
	By Lemma \ref{general_passage_to_trace_nonlinearity_R^N}, to show that \eqref{condition_general_existence_global_solution} holds it is enough to have 
	\[\int_{\left\lbrace u_m\geq k\right\rbrace \cap B_{M} \cap \pd\R^N_+} u^q_m dx' + \int_{\left\lbrace u\geq k\right\rbrace \cap B_{M} \cap \pd\R^N_+} u^q dx' \goesto 0\]
	as $k\goesto 0$, uniformly in $m$. But this is clearly true since $u_m\uparrow u$ $a.e.$ in $B_M\cap\pd\R^N_+$ and $u^q\in L^1\left( \pd\R^N_+\right) $. Hence, we may apply Lemma \ref{general_existence_global_solution}, with $g_m=-2u_{m-1}^q$ and $g=-2u^q$, to conclude that $u$ is a local renormalized solution to 
	\[
	-\lap_pu=2 u^q \mathcal{H}+ 2 \mu \mbox{ in } \R^N .
	\]
	By Theorem \ref{symmetry} such a solution is symmetric, and so by applying Theorem \ref{thm_on_existence_from_symmetry} the result follows.
\end{proof}

Combining Theorems \ref{Sufficiency_supercritical_source}, \ref{Necessity_supercritical_source_I}, \ref{Necessity_supercritical_source_W}, and Remark \ref{remark_on_necessary_conditions} we obtain the following.
\begin{cor}\label{equivalence_supercritical_source}
		Let $1<p<N$, $p-1<q$ , and assume $\mu$ in $\mathfrak{M}_b\left( \pd\R^N_+\right)$ is nonnegative. Then the following are equivalent:
	\begin{enumerate}
		\item For some $\e>0$ there exists a nonnegative renormalized solution to
		\[
		\begin{cases}
			-\lap_pu=0 & \mbox{ in } \R^N_+ \\
			\abs{\nabla u}^{p-2}u_\n = \e\mu + u^q  & \mbox{ on } \pd\R^N_+
		\end{cases}
		\]
		satisfying
		\[
		u(x',x_N)\leq C(p,q,N,\e) W_{1-\frac{1}{p},p,N-1}\left[\mu\right](x') \mbox{ in } \O \cap \overline{\R^N_+},
		\] 
		where $\O=\O_1\cap \left( \O_2\times \R\right) $, $\O_2\subset \R^{N-1}$, with $cap_{1,p,N}\left( \O_1^c\right) =0$ and $\abs{\O_2^c}=0$.
		\item There exists $C>0$ such that for all balls $B\subset \pd\R^N_+\simeq \R^{N-1}$ 
		\[ 
		\int_B \left( I_{p-1,N-1}\left[ \mu_B\right] \right)^{\frac{q}{p-1}}dx' \leq C  \mu\left( B\right) 
		\]
		where $\mu_B$ is the restriction of $\mu$ to $B$. 
		\item There exists $C>0$ such that for all compact sets $K\subset \pd\R^N_+\simeq \R^{N-1}$
		\[
		\mu\left( K\right) \leq C cap_{I_{p-1},\frac{q}{q-(p-1)},N-1}\left( K\right) .
		\]
		\item There exists $C>0$ such that for all balls $B\subset \pd\R^N_+\simeq \R^{N-1}$
		\[
		\int_{\R^{N-1}} \left( W_{1-\frac{1}{p},p,N-1}\left[ \mu_B\right] \right)^{q}dx' \leq C  \mu\left( B\right) .
		\]
		\item $W_{1-\frac{1}{p},p,N-1}\left[ \mu\right] \in L^q\left( \pd\R^N_+\right) $ and 
		\[
		W_{1-\frac{1}{p},p,N-1}\left[ \left( W_{1-\frac{1}{p},p,N-1}\left[ \mu\right] \right) ^q\right] \leq C W_{1-\frac{1}{p},p,N-1}\left[ \mu\right] \ \mbox{ $a.e.$ in } \pd\R^N_+.
		\] 
	\end{enumerate}
\end{cor}
\begin{proof}
	We know $(1)$ implies $(2)$ from Theorem \ref{Necessity_supercritical_source_I}. We noted in Remark \ref{remark_on_necessary_conditions} that $(2)$ is equivalent with $(3)$ and implies $(4)$. That $(4)$ implies $(5)$ was shown in Theorem \ref{Necessity_supercritical_source_W}. Finally, suppose $(5)$ holds for some constant $C$. Then we see that for any $\e>0$ 
	\[
	W_{1-\frac{1}{p},p,N-1}\left[ \left( W_{1-\frac{1}{p},p,N-1}\left[ \e\mu\right] \right) ^q\right] \leq C \e^{\frac{q-(p-1)}{(p-1)^2}}W_{1-\frac{1}{p},p,N-1}\left[ \e\mu\right]
	\]
	 $a.e.$ in $\pd\R^N_+$, and so $(1)$ follows from Theorem \ref{Sufficiency_supercritical_source} provided $\e>0$ is chosen small enough.
\end{proof}

\section{Nonexistence for the subcritical case}

We now turn to the problem of nonexistence. 

Notice that when showing \eqref{testing_inequality_I}, in the proof of Theorem \ref{Necessity_supercritical_source_I}, we actually obtain
\[ 
\int_B \left( I_{p-1,N-1}\left[ \w_B\right] \right)^{\frac{q}{p-1}}dx' \leq C(N,p,q)  \w\left( B\right) 
\]
where $\w=2 \bar{u}^q\mathcal{H} + 2 \mu$. Note also that the argument could have been applied directly to a $p-$ superharmonic function $v$ solving $-\lap_p v=2 v^q\mathcal{H} + 2 \mu$ in $\R^N$. On the other hand, we obtained
\[
\sum_{Q\subset P} \left( \frac{\w(Q)}{\abs{Q}^{1-\frac{p-1}{N-1}}}\right) ^{\frac{q}{p-1}}\abs{Q} \leq C(N,p) \w(P) 
\]
for all dyadic cubes $Q,P\subset\pd\R^N_+$, which in the case $p=N$ implies
\[
\w(P)^{\frac{q}{N-1}}\abs{P}\leq C(N)\w\left( P\right).
\]
This last inequality cannot hold for a bounded nonnegative $\w$ defined in $\pd\R^N_+$ unless it is trivial. Hence, considering also the equivalences in Remark \ref{remark_on_necessary_conditions}, we can conclude the following.
\begin{cor}\label{improved_testing_inequality} 
	Let $1<p\leq N$ and $p-1<q$. Let $\mu$ in $\mathfrak{M}_b\left( \pd\R^N_+\right)$ be nonnegative and suppose $u$ is a nonnegative $p-$ superharmonic solution to $-\lap_p u=2 u^q\mathcal{H} + 2 \mu$ in $\R^N$. If $p<N$ then 
	\[
	\int_K u^qdx' + \mu\left( K\right) \leq C(N,p,q)cap_{I_{p-1},\frac{q}{q-(p-1)},N-1}\left( K\right) 
	\]
	for all compact sets $K\subset \pd\R^N_+$. If $p=N$ then $u(x',0)= 0$ $a.e.$ in $\pd\R^N_+$ and $\mu\equiv 0$.
\end{cor}

Since $cap_{I_{p-1},\frac{q}{q-(p-1)},N-1}\equiv 0$ whenever $\frac{(p-1)q}{q-(p-1)}\geq N-1$ (see \cite{AdamsHedberg}) we have the following Liouville-type theorem for subcritical problems with source.

\begin{thm}\label{Liouville_Theorem}
	Let $1<p\leq N$, $p-1<q$, and $\mu\in \mathfrak{M}_b\left( \pd\R^N_+\right) $ nonnegative. If $p=N$, or $p<N$ and $q\leq \frac{(N-1)(p-1)}{N-p}$, then there are no nontrivial nonnegative $p-$ superharmonic solutions of  $-\lap_p u=2 u^q\mathcal{H} + 2 \mu$ in $\R^N$. In particular, there are no nontrivial nonnegative renormalized solutions of \eqref{eq_source}. 
\end{thm}
\begin{proof}
	Since every nonnegative local renormalized solution coincides $a.e.$ with a $p-$ superharmonic solution of the same equation (see Remark \ref{remark_on_p_superharmonic_representative}), by Remark \ref{renormalized_solutions_are_local}, the hypothesis, and the previous corollary, we see that is enough to show that there are no nontrivial nonnegative $p-$ superharmonic solutions of $-\lap_pu=0$ in $\R^N_+$ whose trace vanishes $a.e.$ in $\pd\R^N_+$. As noted in Remark \ref{remark_on_wolff_potential_estimate}, any such solution $u$ satisfies 
	\[
	u(x)\leq C \inf_{B_M}u
	\]
	in $B_M$ for any $M>0$, and so $u\equiv0$. 

\end{proof}

\section{Characterization of removable sets}

In this section we obtain a characterization of removable sets for problem \eqref{eq_source} when $\mu\equiv0$. In order to properly define removable sets we first define what does it mean to have a renormalized solution up to a portion of the boundary. We give a definition which is a natural variant of definition \ref{defn}.

\begin{defn}\label{defn_solution_on_portion_of_boundary}
	Let $1<p\leq N$ and $p-1<q$. Given $K\subset \pd\R^N_+$ compact, a \textit{renormalized solution} of 
	\be\label{equation_on_portion_of_boundary}
	\begin{cases}
		-\lap_pu=0 & \mbox{ in } \R^N_+ \\
		\abs{\nabla u}^{p-2}u_\n = \abs{u}^{q-1}u & \mbox{ on } \pd\R^N_+\setminus K
	\end{cases}
	\ee
	is a function $u$ defined in $\R^N_+$ such that:
	\begin{enumerate}
		\item $u$ is measurable, finite $a.e.$, and $T_k(u)\in W^{1,p}_{loc}\left( \R^N_+\right) $ for all $k>0$;
		\item $\abs{\nabla u}^{p-1}\in L_{loc}^s\left( \R^N_+\right) $ for all $1\leq s<\frac{N}{N-1}$;
		\item $\abs{u}^{p-1}\in L_{loc}^s\left( \R^N_+\right) $ for all $1<s<\frac{N}{N-p}$ ($1<s<\infty$ if $p=N$);
		\item $u$ is finite $a.e.$ in $\pd\R^N_+\setminus K$, and $u\in L^q\left( \O\right) $ for any closed set $\O\subset \pd\R^N_+$ such that $\O\subset K^c$;
		\item there holds
		\[
		\int_{\R^N_+}\abs{\nabla u}^{p-2}\nabla u\cdot\nabla w dx = \int_{\pd\R_+^N\setminus K} \abs{u}^{q-1}u w  dx' 
		\]
		for all $w\in W^{1,p}\left( \R^N_+\right) $ compactly supported in $\overline{\R^N_+}\setminus K$, whose trace belongs to $L^\infty\left( \pd\R^N_+\setminus K\right) $, and satisfying the following condition: there exists $k>0$, $r>N$, and functions $w^{\pm\infty}\in W^{1,r}\left( \R^N_+\right) $ such that
		\[
		\begin{cases}
		w=w^{+\infty} & \mbox{ $a.e.$ in } \left\lbrace x\in \R^N_+ \ : \ u>k\right\rbrace \\
		w=w^{-\infty} & \mbox{ $a.e.$ in } \left\lbrace x\in \R^N_+ \ : \ u<-k\right\rbrace  .
		\end{cases} 
		\]
	\end{enumerate}
\end{defn}

\begin{remark}
	We note that, just as in Remark \ref{remark_quasicontinuity}, it makes sense to talk about the boundary values of $u$ in $\pd\R^N_+\setminus K$. 
\end{remark}

Now we define removable sets.

\begin{defn}
	We say that a compact set $K\subset\pd\R^N_+$  is \textit{removable} for \eqref{equation_on_portion_of_boundary} if every nonnegative renormalized solution of \eqref{equation_on_portion_of_boundary} is a nonnegative renormalized solution of 
	\be\label{equation_removed}
	\begin{cases}
		-\lap_pu=0 & \mbox{ in } \R^N_+ \\
		\abs{\nabla u}^{p-2}u_\n = u^q & \mbox{ on } \pd\R^N_+ .
	\end{cases} 
	\ee
\end{defn}

We have the following characterization of removable sets.

\begin{thm}\label{removability_criterion}
	If $1<p<N$ and $q> \frac{(N-1)(p-1)}{N-p}$ then a compact set $K\subset\pd\R^N_+$ is removable for \eqref{equation_on_portion_of_boundary} if and only if $cap_{I_{p-1},\frac{q}{q-(p-1)},N-1}\left( K\right) = 0$.
\end{thm}
\begin{proof}
	Let $u$ be a renormalized solution to \eqref{equation_on_portion_of_boundary} and suppose $cap_{I_{p-1},\frac{q}{q-(p-1)},N-1}\left( K\right)$ is equal to zero. Since $\frac{q(p-1)}{q-(p-1)}<N-1$ we can combine Theorems 5.1.4 and 5.5.1 of \cite{AdamsHedberg} to conclude that $cap_{1-\frac{1}{p},p,N-1}\left(K\right) =0$. Let $\bar{u}$ be the extension of $u$ to $\R^N$ by even reflection. Then $\bar{u}$ is a local renormalized solution to $-\lap_p\bar{u}=2 u^q\mathcal{H}$ in $\R^N\setminus K$. By Proposition \ref{trace&extensionE} $cap_{1,p,N}\left( K\right) =0$ and by Theorem 4.3.6 of \cite{Veron16} this implies that the $p-$ superharmonic representative of $\bar{u}$ can be extended to $\R^N$ as a nonnegative $p-$ superharmonic function. By Remark \ref{remark_on_cap_equality_with_psuperharmonic}, this $p-$ superharmonic representative coincides $cap_{1,p,N}-q.e.$ with $u$ in $\overline{\R^N_+}$. Let $\mu$ be the Radon measure associated to $\bar{u}$, i.e., the measure such that $-\lap_p \bar{u}=\mu$ in $\D'\left( \R^N\right)$. Let us show that $\mu= 2 u^q\mathcal{H}$.
	
	Take $\phi\in C_0^\infty \left( \R^N\right)$ nonnegative and let $\phi_n$ be such that $0\leq \phi_n\leq \phi$, $\phi_n\in C_0^\infty\left( \R^N\setminus K\right) $, and $\phi_n\goesto \phi$ point-wise in $\R^N\setminus K$. Note in particular that $\phi_n\goesto\phi$ $\mathcal{H}-$$a.e.$. Hence, by Fatou's Lemma,
	\begin{align*}
	\int_{\pd\R^N_+}2u^q\phi dx' &= \int_{\R^N}  2u^q\phi d\mathcal{H} \\
						&\leq \liminf_{n\goesto\infty}\int_{\R^N}2u^q\phi_nd\mathcal{H}\\
						&=\liminf_{n\goesto\infty}\int_{\R^N}\abs{\nabla \bar{u}}^{p-2}\nabla \bar{u} \cdot\nabla \phi_n dx \\
						&=\liminf_{n\goesto\infty}\int_{\R^N} \phi_nd\mu \\
						&\leq \int_{\R^N} \phi d\mu 
	\end{align*}
	and so we conclude $u^q\in L^1\left( \pd \R^N_+\right) $ and $\mu\geq 2 u^q\mathcal{H}$ in $\D'\left( \R^N\right) $ (recall that $u$ satisfies $(4)$ of Definition \ref{defn_solution_on_portion_of_boundary}). It follows at once from considering the equations solved by $\bar{u}$ that in fact $\mu=2 u^q\mathcal{H}$ in $\R^N\setminus K$. Then, setting $\mu^K=\mu-2 u^q\mathcal{H}$ we have that $\bar{u}$ is a $p-$ superharmonic solution of 
	\[
	-\lap_p\bar{u} = 2 u^q \mathcal{H}+ \mu^K \mbox{ in } \R^N
	\]
	where the measure $\mu^K$ is supported in $K$ (and hence bounded). Then, by Corollary \ref{improved_testing_inequality}, 
	\[
	\mu^K\left( K\right) \leq C cap_{I_{p-1},\frac{q}{q-(p-1)},N-1}\left( K\right) = 0
	\]
	and so $\mu^K\equiv 0$. By Theorem 4.3.4 of \cite{Veron16}, $\bar{u}$ is a local renormalized solution to $-\lap_p\bar{u} = 2 u^q \mathcal{H}$ in $\R^N$, and so, by Theorem \ref{thm_on_existence_from_symmetry}, the restriction of $\bar{u}$ to $\R^N_+$ is a renormalized solution of \eqref{equation_removed}.

	For the converse, suppose $cap_{I_{p-1},\frac{q}{q-(p-1)},N-1}\left( K\right)>0$. We let $\mu$ be the capacitary measure of $K$ (see Theorem 2.5.3 of \cite{AdamsHedberg}) and extend it to $\pd\R^N_+$ by setting $\mu\left( A\right) =\mu\left( A\cap K\right) $. By Theorem 2.5.5 of \cite{AdamsHedberg} we see that $\mu$ satisfies \eqref{testing_inequality_cap_I} and so, by Corollary \ref{equivalence_supercritical_source}, there exists a renormalized solution $u$ of \eqref{eq_source} with measure $\e\mu$ for some $\e>0$. Since $\mu$ is concentrated in $K$, $u$ is also a solution of \eqref{equation_on_portion_of_boundary} and thus $K$ is not removable.
\end{proof}

\newpage

\addcontentsline{toc}{chapter}{Bibliography}
\bibliography{Biblio} 
\bibliographystyle{abbrv}

\end{document}